\documentclass[reqno,a4paper,11pt]{amsart}

\usepackage[all,poly]{xy}
\usepackage{amsfonts}
\usepackage[mathcal]{eucal}
\usepackage{eufrak}
\usepackage{amssymb}
\usepackage{amsmath}
\usepackage{mathrsfs}
\usepackage{color}
\usepackage[pagebackref,colorlinks]{hyperref}
\usepackage{enumerate}

\theoremstyle{plain}
\newtheorem {lemma}{Lemma}[section] 
\newtheorem {theorem}[lemma]{Theorem}

\newtheorem {corollary}[lemma]{Corollary}
\newtheorem{conjecture}{Conjecture}
\newtheorem {cor}[lemma]{Corollary}
\newtheorem {proposition}[lemma]{Proposition}
\newtheorem {prop}[lemma]{Proposition}

\theoremstyle{definition}
\newtheorem {remark}[lemma]{Remark}

\newtheorem {remarks}[lemma]{Remarks}
\newtheorem {example}[lemma]{Example}

\theoremstyle{definition}

\newtheorem{deff}[lemma]{Definition}{}

\newcommand{\M}{\operatorname{\mathbb M}}
\newcommand{\LL}{\operatorname{\mathcal L}}

\newcommand{\IM}{\operatorname{\mathbb I}}
\newcommand{\gr}{\operatorname{gr}}
\newcommand{\V}{\operatorname{\mathcal V}}

\newcommand{\LCM}{\operatorname{LCM}}

\newcommand{\va}{\varphi}
\newcommand{\Ga}{\Gamma}
\newcommand{\ga}{\gamma}
\newcommand{\al}{\alpha}

\newcommand{\de}{\delta}
\newcommand{\la}{\lambda}

\newcommand{\out}{\operatorname{out}}

\newcommand{\ol}{\overline}

\newcommand{\End}{\operatorname{End}}

\newcommand{\Hom}{\operatorname{Hom}}

\newcommand{\st}{\operatorname{st}}
\newcommand{\sink}{\operatorname{Sink}}
\newcommand{\coker}{\operatorname{coker}}
\newcommand{\supp}{\operatorname{supp}}

\newcommand{\outdeg}{\operatorname{outdeg}}
\newcommand{\totdeg}{\operatorname{totdeg}}
\newcommand{\card}{\operatorname{card}}

\newcommand{\Pgr}{\mathcal P \mathrm{gr}}
\newcommand{\OO}{\mathcal{O}}

\long\def\forget#1\forgotten{}

\title[The graded Grothendieck group and Leavitt path algebras]{The graded Grothendieck group\\ and the classification of  Leavitt path algebras}

\author{Roozbeh Hazrat}\address{
Department of Pure Mathematics\\
Queen's University\\
Belfast BT7 1NN\\
United Kingdom} \email{r.hazrat@qub.ac.uk}


\begin{thanks}
{This work was supported by the Engineering and Physical
  Sciences Research Council [grant EP/I007784/1]. The author would like to thank the referee for the constructive suggestions.}
\end{thanks}

\begin{document}

\begin{abstract}
This paper is an attempt to show that, parallel to Elliott's classification of AF $C^*$-algebras by means of $K$-theory, the graded $K_0$-group classifies Leavitt path algebras completely. In this direction, we prove this claim at two extremes, namely, for the class of acyclic graphs (graphs with no cycles) and multi-headed comets or rose graphs  (graphs  in which each head is connected to a cycle or to a collection of loops), or a mixture of these graphs. 
\end{abstract}

\maketitle


\section{Introduction} \label{introf}

In \cite{vitt62} Leavitt considered the free associative $K$-algebra $A$ generated by symbols $\{x_i,y_i \mid 1\leq i \leq n\}$ subject to relations 
\begin{equation}\label{jh54320}
x_iy_j =\delta_{ij}, \text{ for all } 1\leq i,j \leq n,  \text{  and  } \sum_{i=1}^n y_ix_i=1,
\end{equation} where $K$ is a field, $n\geq 2$ and $\delta_{ij}$ is the Kronecker delta. 
The relations guarantee the right $A$-module homomorphism 
\begin{align}\label{is329ho}
\phi:A&\longrightarrow A^n\\
a &\mapsto (x_1a	,x_2a,\dots,x_na)\notag
\end{align}
has an inverse 
\begin{align}\label{is329ho9}
\psi:A^n&\longrightarrow A\\
(a_1,\dots,a_n) &\mapsto  y_1a_1+\dots+y_na_n, \notag 
\end{align}
so $A\cong A^n$ as a right $A$-module. He showed that $A$ is universal with respect to this property, of type $(1,n-1)$ (see \S\ref{gtr5654}) and it is a simple ring. 
Modeled on this, Leavitt path algebras are introduced \cite{aap05,amp}, which attach to a directed graph a certain algebra. In the case that the graph has one vertex and $n$ loops, it recovers Leavitt's algebra (\ref{jh54320}) (where $y_i$ are the loops and $x_i$ are the ghost loops).  

In \cite{cuntz1} Cuntz considered the universal unital $C^*$-algebra $\OO_n$ generated by isometries $\{s_i \mid 1\leq i\leq n\}$ subject to the relation \[ \sum_{i=1}^ns_is_i^*=1,\] where $n\geq 2$. He showed that $\mathcal O_n$ is purely infinite, simple $C^*$-algebras. Modeled on this, graph $C^*$-algebras are introduced \cite{paskrae,Raegraph},  which attach to a directed graph a certain $C^*$-algebra. In the case that the graph has one vertex and $n$ loops, it recovers  Cuntz's algebra.  

The study of Leavitt path algebras and their analytic counterpart develop parallel and strikingly similar.  
Cuntz computed $K_0$ of $\OO_n$ \cite{cuntz2} and in~\cite{Raeburn113} Raeburn and  Szyma\'nski carried out the calculation for graph $C^*$-algebras (see also~\cite[Ch.~7]{Raegraph}). Ara,  Brustenga and Cortin\~as~\cite{arawillie} calculated $K$-theory of Leavitt path algebras. It was shown that when $K=\mathbb C$, the $K_0$-group  of graph $C^*$-algebras and Leavitt path algebras coincide~\cite[Theorem~7.1]{amp}. 

The Grothendieck group $K_0$ has long been recognized as an essential tool in classifying certain types of $C^*$-algebras. Elliott proved that $K_0$ as a ``pointed'' pre-ordered group classifies the AF $C^*$-algebras completely. 
Since the ``underlying'' algebras of Leavitt path algebras are ultramatricial algebras (as in AF $C^*$-algebras) Elliott's work (see also~\cite[\S 8]{rordam}) has prompted to consider the $K_0$-group as the main tool of classifying certain types of Leavitt path algebras~\cite{aalp}. 

This note is an attempt to justify that parallel to Elliott's work on AF $C^*$-algebras, the ``pointed'' pre-order {\it graded} Grothedieck group, $K^{\gr}_0$,  classifies the Leavitt path algebras completely.  This note proves this claim at two extremes, namely for acyclic graphs (graphs with no cycle) and for graphs in which each head is connected to a cycle or to a collection of loops (see~(\ref{pid98})) or a mixture of these graphs (see~(\ref{monster})). 

Leavitt's algebra constructed in~(\ref{jh54320}) has a natural grading; assigning $1$ to $y_i$ and $-1$ to $x_i$, $1\leq i \leq n$, since the relations are homogeneous (of degree zero), the algebra $A$ is a $\mathbb Z$-graded algebra. The isomorphism~(\ref{is329ho}) induces a graded isomorphism   
\begin{align}\label{is329ho22}
\phi:A&\longrightarrow A(-1)^n\\
a &\mapsto (x_1a	,x_2a,\dots,x_na), \notag
\end{align}
where $A(-1)$ is the suspension of $A$ by $-1$. In the non-graded setting, the relation $A\cong A^n$ translates to $(n-1)[A]=0$ in $K_0(A)$. In fact for $n=2$, one can show that $K_0(A)=0$. However, when considering the grading, for $n=2$, $A \cong_{\gr} A(-1)\oplus A(-1)$ and $A(i)\cong_{\gr} A(i-1)\oplus A(i-1)$, $i\in \mathbb Z$, and therefore $A \cong_{\gr}  A(-i)^{2^i}$, $i \in \mathbb N$. This implies not only  $K^{\gr}_0(A)\not = 0$ but one can also show that $K^{\gr}_0(A)\cong \mathbb Z[1/2]$. This is an indication that graded $K$-groups can capture more information than the non-graded $K$-groups (see Example~\ref{smallgraphs} for more examples). In fact, we will observe that there is a close relation between graded $K$-groups of Leavitt path algebras and their graph $C^*$-algebra counterparts. For example, if the graph is finite with no sink, we have isomorphisms in the  middle of the diagram  below which induce isomorphisms on the right and left hand side of the diagram. This immediately implies 
$  K_0(C^*(E)) \cong K_0(\LL_{\mathbb C}(E))$. (See Remark~\ref{hjsonf} for the general case of row finite graphs.)
\[
\xymatrix{ 0 \ar[r] & K_1(C^*(E)) \ar[r] \ar@{.>}[d] &
K_0\big(C^*(E\times_1 \mathbb Z)\big) \ar[d]^{\cong}
\ar[r]^{1-\beta^{-1}_*} & K_0\big(C^*(E\times_1 \mathbb Z)\big) \ar[d]^{\cong} \ar[r]& K_0(C^*(E))\ar@{.>}[d]\ar[r]& 0\\
0 \ar[r] & \ker(1-N^t) \ar[r] & K^{\gr}_0(\LL_{\mathbb C}(E)) \ar[r]^{1-N^{t}} &
K^{\gr}_0(\LL_{\mathbb C}(E)) \ar[r]& K_0(\LL_{\mathbb C}(E)) \ar[r] &0}.
\]

In the graded $K$-theory we study in this note, suspensions play a pivotal role. For example, Abrams \cite[Proposition~1.3]{arock} has given a number theoretic criterion when matrices over Leavitt algebras with no suspensions are graded isomorphism. Abrams' criterion shows 
\[\M_3(\LL_2)\not \cong_{\gr} \M_4(\LL_2),\] where $\LL_2$ is the Leavitt algebra constructed in (\ref{jh54320}) for $n=2$. However using graded $K$-theory (Theorem~\ref{re282} and Theorem~\ref{mani543}), we shall see that 
the Leavitt path algebras of the following two graphs 
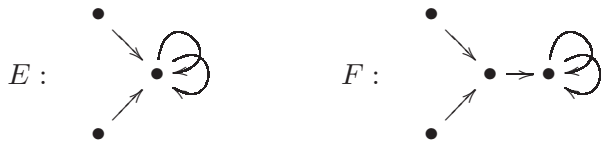
\begin{equation*}
\xymatrix@=10pt{
      &\bullet \ar[dr]   & &&&&  & \bullet \ar[dr] \\
 E: &&  \bullet \ar@(ur,rd)  \ar@(u,r) & & && F: & & \bullet  \ar[r] &  \bullet \ar@(ur,rd)  \ar@(u,r) &\\
      &\bullet \ar[ur]  & &&&&  & \bullet \ar[ur]
}
\end{equation*}
are graded isomorphic, i.e., $\LL(E) \cong_{\gr} \LL(F)$, which would then imply 
\begin{equation}
\M_3(\LL_2)(0,1,1) \cong_{\gr} \M_4(\LL_2)(0,1,2,2).
\end{equation}
This also shows, by considering suspensions, we get much wider classes of graded isomorphisms of matrices over Leavitt algebras (see Theorem~\ref{cfd2497}). In fact, the suspensions of graded modules induce a $\mathbb Z[x,x^{-1}]$-module structure on the graded $K_0$-group. This extra structure helps us to characterize the Leavitt path algebras. 

Elliott showed that $K_0$-groups classify ultramatricial algebras completely (\cite[\S15]{goodearlbook}). Namely, for two ultramatricial algebras $R$ and $S$, if $\phi:K_0(R)\rightarrow K_0(S)$ is an isomorphism such that  $\phi([R])=[S]$ and the set of isomorphism classes of finitely generated  projective modules of $R$ are sent to and covers the set of isomorphism classes of finitely generated   projective modules of $S$, i.e., $\phi$ and $\phi^{-1}$ are order preserving, then $R$ and $S$ are isomorphic.  This isomorphism is written as $\big(K_0(R),[R]\big)\cong \big (K_0(S),[S]\big)$. 

Replacing $K_0$ by $K^{\gr}_0$, we will prove that a similar statement holds for the class of Leavitt path algebras arising from acyclic, multi-headed comets and multi-headed rose graphs (see Figures~\ref{pid98}) or a mixture of these (called a polycephaly graph, see Figure~\ref{monster}) in Theorem~\ref{mani543}. Here the isomorphisms between $K^{\gr}_0$-groups are considered as ordered preserving $\mathbb Z[x,x^{-1}]$-modules (see Example~\ref{upst} for the action of $\mathbb Z[x,x^{-1}]$ on $K^{\gr}_0$-groups of acyclic and comets graphs). 

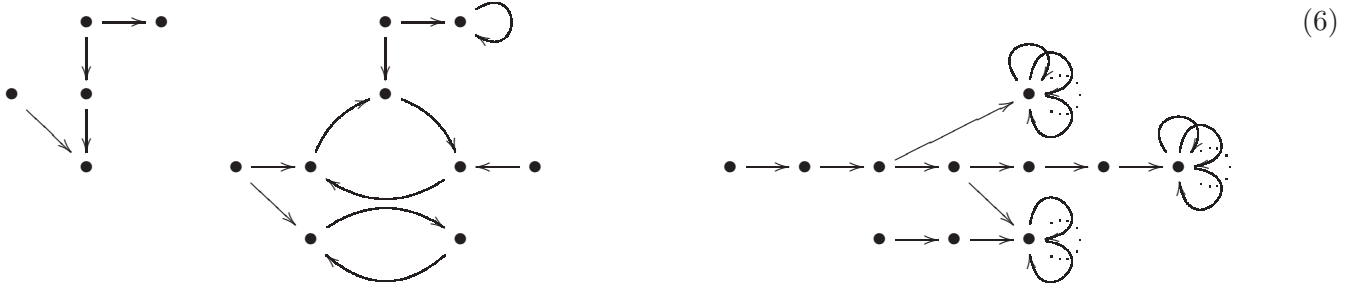
\begin{equation}\label{pid98}
\xymatrix@=16pt{
 & \bullet \ar[d] \ar[r] & \bullet  &      &&   \bullet \ar[d] \ar[r] & \bullet \ar@(ur,dr)\\
\bullet \ar[dr] & \bullet\ar[d] & &   &&\bullet\ar@/^/[dr] &&&     &&    & & &&   \bullet  \ar@(ul,ur)  \ar@(u,r) \ar@{.}@(ur,dr) \ar@(r,d)& \\
& \bullet &&    \bullet \ar[r] \ar[dr]  & \bullet \ar@/^/[ur]     &&\bullet \ar@/^1pc/[ll] & \bullet \ar[l] &    &&
    \bullet \ar[r]  &   \bullet \ar[r]  &  \bullet \ar[r] \ar[urr]   &  \bullet \ar[r] \ar[dr] & \bullet \ar[r]  &    \bullet \ar[r]  & \bullet  \ar@(ul,ur)  \ar@(u,r) \ar@{.}@(ur,dr) \ar@(r,d)& \\
 &&    & & \bullet \ar@/^1pc/[rr] &  & \bullet \ar@/^1.3pc/[ll]&&& &&&  \bullet \ar[r]  & \bullet \ar[r]  &   \bullet   \ar@(u,r) \ar@{.}@(ur,dr) \ar@(r,d)&  
 }
 \end{equation}

\bigskip

This paves the way to pose the following conjecture. 

\begin{conjecture}[Weak classification conjecture]\label{weakconj}
Let $E$ and $F$ be row-finite graphs. Then  $\LL(E)\cong_{\gr} \LL(F)$ if and only if there is an ordered preserving $\mathbb Z[x,x^{-1}]$-module  isomorphism
\[\big (K_0^{\gr}(\LL(E)),[\LL(E)]\big ) \cong \big (K_0^{\gr}(\LL(F)),[\LL(F)]\big ).\]
\end{conjecture}

Whereas, in Elliott's case, $K_0$ is a functor to the category of pre-ordered abelian groups, 
$K^{\gr}_0$ is a functor to $\Gamma$-pre-ordered abelian groups taking into account the $\Gamma$-grading of the rings (see~\S\ref{pregg5}).  


In fact the proof of Theorem~\ref{mani543} shows that an isomorphism between the graded Grothendieck groups induces a $K$-algebra isomorphism between the Leavitt path algebras. Therefore starting from a ring isomorphism, passing through $K$-theory, we obtain $K$-algebra isomorphisms between Leavitt path algebras. Therefore we obtain the following conjecture for the class of polycephaly graphs (see Corollary~\ref{griso7dgre}). 

\begin{conjecture}\label{cofian}
Let $E$ and $F$ be row-finite graphs. Then $\LL_K(E) \cong_{\gr} \LL_K(F)$ as  rings if and only if $\LL_K(E) \cong_{\gr} \LL_K(F)$ as $K$-algebras. 
\end{conjecture}

In fact we prove $K^{\gr}_0$ functor is a fully faithful functor from the category of acyclic Leavitt path algebras to the category of pre-ordered abelian groups (see \S\ref{pregg5} and Theorem~\ref{catgrhsf}). From the results of the paper, one is tempted to make the following conjecture.

\begin{conjecture}[Strong classification conjecture]\label{strongconj}
The graded Grothendieck group $K^{\gr}_0$ is a fully faithful functor from the category of  Leavitt path algebras with graded homomorphisms modulo inner-automorphisms to the category of pre-ordered abelian groups with order-units. 
\end{conjecture}

\section{The graded Grothendieck group} 

In this note all modules are considered right modules unless stated otherwise. For a set $\Gamma$ by $\mathbb Z ^{\Gamma}$ we mean $\Gamma$-copies of $\mathbb Z$,
i.e., $\bigoplus_{\gamma \in \Gamma} \mathbb Z_{\gamma}$ where
$\mathbb Z_{\gamma} = \mathbb Z$ for each $\gamma \in \Gamma$.

\subsection{Graded rings}\label{pregr529}

A ring $A = \textstyle{\bigoplus_{ \ga \in \Ga}} A_{\ga}$ is called a
\emph{$\Ga$-graded ring}, or simply a \emph{graded ring},
\index{graded ring} if $\Ga$ is an (abelian) group, each $A_{\ga}$ is
an additive subgroup of $A$ and $A_{\ga}  A_{\delta} \subseteq
A_{\ga + \delta}$ for all $\ga, \delta \in \Ga$. 
The
elements of $A_\ga$ are called \emph{homogeneous of degree $\ga$}
and we write deg$(a) = \ga$ if $a \in A_{\ga}$. We let $A^{h} =
\bigcup_{\ga \in \Ga} A_{\ga}$ be the set of homogeneous elements of
$A$. We call the set $\Gamma_A=\{ \ga \in \Gamma \mid A_\ga \not = 0 \}$ the {\it support} of $A$.
A $\Ga$-graded ring $A=\textstyle{\bigoplus_{ \ga \in \Ga}} A_{\ga}$
is called a \emph{strongly graded ring} if $A_{\ga} A_{\de} = A_{\ga +\de}$
for all $\ga, \de \in \Ga$. We say $A$ has a trivial grading, or $A$ is concentrated in
degree zero if the support of $A$ is the trivial group, i.e., $A_0=A$ and $A_\ga=0$ for $\ga \in \Gamma \backslash \{0\}$.

Let $A$ be a $\Gamma$-graded ring. A \emph{graded right $A$-module} $M$ is defined to be a right $A$-module
with a direct sum decomposition $M=\bigoplus_{\ga \in \Gamma}
M_{\ga}$, where each $M_{\ga}$ is an additive subgroup of $M$  such that 
$M_{\la} \cdot A_{\ga} \subseteq M_{\ga + \la}$ for all $\ga, \la \in
\Ga$. For $\de \in \Ga$, we define the
$\de$-{\it suspended} $A$-module $M(\de)$ \label{deshiftedmodule} as $M(\de)
=\bigoplus_{\ga \in \Ga} M(\de)_\ga$ where $M(\de)_\ga =
M_{\ga+\de}$. 
For two graded $A$-modules $M$ and $N$, a {\it graded $A$-module homomorphism of degree $\delta$} is an $A$-module homomorphism $f:M\rightarrow N$, such that $f(M_\ga)\subseteq N_{\ga+\delta}$ for any $\ga \in \Ga$. By a {\it graded homomorphism} we mean a graded homomorphism of degree $0$.  
By $\mbox{gr-}A$, we denote the category of graded right $A$-modules with graded homomorphisms.  
For $\alpha \in \Gamma$,  the {\it $\alpha$-suspension functor} $\mathcal T_\alpha:\mbox{gr-}A\rightarrow \mbox{gr-}A$, $M \mapsto M(\alpha)$  is an isomorphism with the property
$\mathcal T_\alpha \mathcal T_\beta=\mathcal T_{\alpha + \beta}$, $\alpha,\beta\in \Gamma$.

For graded modules $M$ and $N$, if $M$ is finitely generated, then $\Hom_{A}(M,N)$ has a natural $\Gamma$-grading 
\begin{equation}\label{hgd543p}
\Hom_{A}(M,N)=\bigoplus_{\gamma \in \Gamma} \Hom(M,N)_{\gamma}, 
\end{equation}
where $\Hom(M,N)_{\gamma}$ is the set of all graded $A$-module homomorphisms of degree $\gamma$ (see \cite[\S2.4]{grrings}). 

A graded $A$-module $P$ is called a {\it graded projective} module if $P$ is a projective module. One can check that $P$ is graded projective if and only if the functor $\Hom_{\mbox{\tiny gr-}A}(P,-)$ is an exact functor in $\mbox{gr-}A$ if and only if $P$ is  graded isomorphic to a direct summand of a graded free $A$-module. In particular, if $P$ is a graded finitely generated  projective $A$-module, then there is a graded finitely generated  projective $A$-module $Q$ such that $P\oplus Q\cong_{\gr} A^n(\overline \alpha)$, where $\overline \alpha =(\alpha_1,\cdots,\alpha_n)$, $\alpha_i \in \Ga$.
By $\Pgr(A)$ we denote the category of graded finitely generated  projective right $A$-modules with graded homomorphisms. 
Note that the suspension functor $\mathcal T_\alpha$ restricts to the category of graded finitely generated  
projective modules.

Let $A= \textstyle{\bigoplus_{\ga \in \Ga}} A_{\ga}$ and $B= \textstyle{\bigoplus_{\ga \in
\Ga}} B_{\ga}$ be graded rings.  Then $A\times B$ has a natural grading  given by $A\times B  = \textstyle{\bigoplus_{\ga \in \Gamma}} (A \times B)_{\ga}$ where $(A \times B)_{\ga}=A_\ga \times B_\ga$.  
Similarly,  if $A$ and $B$ are $K$-modules for a field $K$ (where here $K$ has a trivial grading), then $A
\otimes_K B$ has a natural grading  given by
$A \otimes_K B = \textstyle{\bigoplus_{\ga \in \Gamma}} (A \otimes_K B)_{\ga}$
where
\begin{equation}\label{tengr}
(A \otimes_K B)_{\ga} = \Big \{ \sum_i a_i \otimes b_i   \mid   a_i \in
A^h, b_i \in B^h, \deg(a_i)+\deg(b_i) = \ga \Big\}.
\end{equation}

Some of the rings we are  dealing with in this note are of the form $K[x,x^{-1}]$ where $K$ is a field. This is an example of a graded field. A nonzero commutative $\Ga$-graded  ring $A$  is called a {\it graded field} if  every nonzero homogeneous element has an inverse.  It follows that $A_0$ is a field and $\Gamma_A$ is an abelian group. Similar to the non-graded setting, one can show that any $\Gamma$-graded module $M$ over a graded field $A$ is graded free, i.e., it is generated by a homogeneous basis and the graded bases have the same cardinality (see \cite[Proposition~4.6.1]{grrings}). Moreover, if $N$ is a graded submodule of $M$, then 
\begin{equation}\label{dimcouti}
\dim_A(N)+\dim_A(M/N)=\dim_A(M).
\end{equation} In this note, all  graded fields have  torsion free abelian group gradings (in fact, in all our statements $\Gamma=\mathbb Z$,  $A=K[x^n,x^{-n}]$ is a $\Gamma$-graded field with $\Ga_A=n \mathbb Z$, for some $n\in \mathbb N$). However, this assumption is not necessary for the statements below.

\subsection{Grading on matrices}\label{matgrhe}

Let $A$ be a $\Gamma$-graded ring and $M=M_1\oplus \dots \oplus M_n$, where $M_i$ are graded finitely generated  right $A$-modules. So $M$ is also a graded $A$-module. Let $\pi_j:M\rightarrow M_j$ and $\kappa_j:M_j\rightarrow M$ be the (graded) projection and injection homomorphisms. Then there is a graded isomorphism 
\begin{equation}\label{fgair}
\End_A(M)\rightarrow \big[\Hom(M_j,M_i)\big]_{1\leq i,j\leq n}
\end{equation} defined by 
$\phi \mapsto [\pi_i\phi\kappa_j]$, $1\leq i,j \leq n$. 

For a graded ring $A$, observe that $\Hom_A\big (A(\delta_i),A(\delta_j)\big )\cong_{\gr} A(\delta_j-\delta_i)$ (see~\ref{hgd543p}). If 
\[V=A(-\delta_1)\oplus A(-\delta_2) \oplus \dots \oplus A(-\delta_n),\] then by ~(\ref{fgair}),
\[\End_A(V)\cong_{\gr} \big [\Hom\big (A(-\delta_j),A(-\delta_i)\big )\big ]\cong_{\gr}\big[A(\delta_j-\delta_i)\big]_{1\leq i,j\leq n}.\]
Denoting this graded matrix ring by $\M_n(A)(\de_1,\dots,\de_n)$, we have 

\begin{equation}\label{pjacko}
\M_n(A)(\de_1,\dots,\de_n) =
\begin{pmatrix}
A(\de_1 - \de_1) & A(\de_2  - \de_1) & \cdots &
A(\de_n - \de_1) \\
A(\de_1 - \de_2) & A(\de_2 - \de_2) & \cdots &
A(\de_n  - \de_2) \\
\vdots  & \vdots  & \ddots & \vdots  \\
A(\de_1 - \de_n) & A(\de_2 - \de_n) & \cdots &
A(\de_n - \de_n)
\end{pmatrix}.
\end{equation}
Therefore for $\la \in \Gamma$, ${\M_n (A)(\de_1,\dots,\de_n)}_{\la}$, the $\la$-homogeneous elements, are the $n \times n$-matrices over
$A$ with the degree shifted (suspended) as follows:
\begin{equation}\label{mmkkhh}
{\M_n(A)(\de_1,\dots,\de_n)}_{\la} =
\begin{pmatrix}
A_{ \la+\de_1 - \de_1} & A_{\la+\de_2  - \de_1} & \cdots &
A_{\la +\de_n - \de_1} \\
A_{\la + \de_1 - \de_2} & A_{\la + \de_2 - \de_2} & \cdots &
A_{\la+\de_n  - \de_2} \\
\vdots  & \vdots  & \ddots & \vdots  \\
A_{\la + \de_1 - \de_n} & A_{ \la + \de_2 - \de_n} & \cdots &
A_{\la + \de_n - \de_n}
\end{pmatrix}.
\end{equation}

This also shows that 
 \begin{equation}\label{hogr}
\deg(e_{ij}(x))=\deg(x)+\delta_i-\delta_j.
\end{equation}  
Setting $\overline \delta=(\de_1, \dots,\de_n)\in \Gamma^n$, one denotes the graded matrix ring~(\ref{pjacko}) as $\M_n(A)(\overline \de)$. Consider the graded $A$-bi-module 
\[A^n(\overline \de)=A(\de_1)\oplus \dots \oplus A(\de_n).\] Then one can check that $A^n(\overline \de)$ is a graded right $\M_n(A)(\overline \de)$-module and $A^n(-\overline \de)$ is a graded left $\M_n(A)(\overline \de)$-module, where 
$-\overline \de=(-\de_1, \dots,-\de_n)$. 
 
Note that if $A$ has a trivial grading,  this construction induces a {\it good grading} on $\M_n(A)$. 
These group gradings on matrix rings have been
studied by D\u{a}sc\u{a}lescu et al.~\cite{dascalescu}. Therefore for $x \in A$, 
\begin{equation}\label{oiuytr}
\deg(e_{ij}(x))=\delta_i - \delta_j.
\end{equation}

We will use these gradings on matrices to describe the graded structure of Leavitt path algebras of acyclic and comet graphs. 

For a $\Ga$-graded ring $A$, $\overline \alpha = (\al_1 , \ldots \al_m) \in \Ga^m$ and $\overline \delta=(\de_1, \ldots ,\de_n) \in \Ga^n$, let
\[
\M_{m \times n} (A)[\overline \alpha][\overline \delta] =
\begin{pmatrix}
A_{\al_1 -\de_1} & A_{\al_1 -\de_2} & \cdots &
A_{\al_1 -\de_n } \\
A_{\al_2 -\de_1} & A_{\al_2 -\de_2} & \cdots &
A_{\al_2 -\de_n  } \\
\vdots  & \vdots  & \ddots & \vdots  \\
A_{\al_m-\de_1} & A_{\al_m -\de_2} & \cdots & A_{\al_m -\de_n}
\end{pmatrix}.
\]
So $\M_{m\times n} (A)[\overline \alpha][\overline \delta]$ consists of matrices with the
$ij$-entry in $A_{\al_i -\de_j}$.

\begin{prop} \label{rndconggrrna}
Let $A$ be a $\Ga$-graded ring and let $\overline \alpha = (\al_1 , \ldots , \al_m) \in \Ga^m$, 
$\overline \delta=(\de_1, \ldots,\de_n)\in \Ga^n$. Then
the following are equivalent: 

\begin{enumerate}[\upshape(1)]

\item $A^m (\overline \alpha)  \cong_{\gr} A^n(\overline \delta)  $ as graded right $A$-modules. 

\smallskip 

\item $A^m (-\overline \alpha) \cong_{\gr} A^n(-\overline \delta)  $ as graded left $A$-modules. 



\smallskip 

\item There exist $a=(a_{ij}) \in \M_{n\times m} (A)[\overline \delta][\overline \alpha]$ and 
$b=(b_{ij}) \in \M_{m\times n} (A)[\overline \alpha][\overline \delta]$ such that $ab=\IM_{n}$ and $ba=\IM_{m}$.

\end{enumerate}
\end{prop}

\begin{proof}
(1) $\Rightarrow$ (3) Let $\phi: A^m(\overline \al) \rightarrow  A^n (\overline \de)$ and $\psi: A^n(\overline \delta)  \rightarrow  A^m (\overline \alpha)$ 
 be graded right $A$-modules isomorphisms such that $\phi\psi=1$ and $\psi\phi=1$. Let $e_j$ denote the standard basis element of $A^m(\overline \alpha)$ with  $1$ in the $j$-th entry and zeros elsewhere. Then let 
 $\phi(e_j)=(a_{1j},a_{2j},\dots,a_{nj})$, $1\leq j \leq m$. Since $\phi$ is a graded map, comparing the grading of both sides, one can observe that $\deg(a_{ij})=\de_i-\alpha_j$. Note that the map $\phi$ is represented by the left multiplication with the matrix $a=(a_{ij})_{n\times m} \in \M_{n\times m} (A)[\overline \delta][\overline \alpha]$.   In the same way one can construct $b  \in \M_{m\times n} (A)[\overline \alpha][\overline \delta]$ which induces $\psi$. Now $\phi\psi=1$ and $\psi\phi=1$ translate to $ab=\IM_{n}$ and $ba=\IM_{m}$. 

(3) $\Rightarrow$ (1)  If $a \in  \M_{n\times m} (A)[\overline \delta][\overline \alpha]$, then multiplication from the left, induces  a graded right 
$A$-module homomorphism $\phi_a:A^m(\overline \alpha) \longrightarrow A^n(\overline \delta)$. Similarly $b$ induces $\psi_b: A^n(\overline \delta) \longrightarrow A^m(\overline \alpha)$. Now $ab=\IM_{n}$ and $ba=\IM_{m}$ translate to 
$\phi_a\psi_b=1$ and $\psi_b\phi_a=1$.

(2) $\Longleftrightarrow$ (3) This part is similar to the previous cases by considering the matrix multiplication from the right. Namely, the graded left $A$-module homomorphism $\phi: A^m (-\overline \alpha) \rightarrow A^n(-\overline \delta)  $ represented by a matrix multiplication from the right of the form $\M_{m\times n} (A)[\overline \alpha][\overline \delta]$ and similarly $\psi$ gives a matrix  in 
$\M_{n\times m} (A)[\overline \de][\overline \al]$. The rest follows easily. 
 \end{proof}

When $m=n=1$ in Proposition~\ref{rndconggrrna} we obtain, 

\begin{cor} \label{rndcongcori}
Let $A$ be a $\Ga$-graded ring and $\alpha \in \Gamma$.  Then
the following are equivalent: 

\begin{enumerate}[\upshape(1)]

\item $A(\alpha) \cong_{\gr} A$ as graded right $A$-modules. 

\smallskip 

\item $A(-\alpha) \cong_{\gr} A$ as graded right $A$-modules. 

\smallskip 

\item $A(\alpha) \cong_{\gr} A$ as graded left $A$-modules. 

\smallskip 

\item $A(-\alpha) \cong_{\gr} A$ as graded left $A$-modules. 

\smallskip 

\item There is an invertible homogeneous  element of degree $\alpha$. 





\end{enumerate}

\end{cor}

\begin{proof}
This follows from Proposition~\ref{rndconggrrna}. 
\end{proof}

\subsection{Leavitt algebras as graded rings} \label{creekside}
Let $K$ be a field, $n$ and $k$ positive integers and $A$ the free associative $K$-algebra generated by symbols $\{x_{ij},y_{ji} \mid 1\leq i \leq n+k, 1\leq j \leq n \}$ subject to relations (coming from) 
\[ Y\cdot X=I_{n,n} \qquad \text{ and } \qquad X\cdot Y=I_{n+k,n+k}, \] where 
\begin{equation} \label{breaktr}
Y=\left( 
\begin{matrix} 
y_{11} & y_{12} & \dots & y_{1,n+k}\\ 
y_{21} & y_{22} & \dots & y_{2,n+k}\\ 
\vdots & \vdots & \ddots & \vdots\\ 
y_{n,1} & y_{n,2} & \dots & y_{n,n+k} 
\end{matrix} 
\right), ~~~~
X=\left( 
\begin{matrix} 
x_{11} & x_{12} & \dots & x_{1,n}\\ 
x_{21} & x_{22} & \dots & x_{2,n}\\ 
\vdots & \vdots & \ddots & \vdots\\ 
x_{n+k,1} & x_{n+k,2} & \dots & x_{n+k,n} 
\end{matrix} 
\right). 
\end{equation} 

This algebra was studied by Leavitt in relation with its type in \cite{vitt56,vitt57,vitt62}.  
In~\cite[p.190]{vitt56}, he studied this algebra for $n=2$ and $k=1$, where he showed that the algebra has no zero divisors, in
~\cite[p.322]{vitt57} for arbitrary $n$ and $k=1$ and
in~\cite[p.130]{vitt62}  for arbitrary $n$ and $k$ and established that these algebras are of type $(n,k)$ (see \S\ref{gtr5654}) and when $n\geq 2$ they are domains. 
We denote this algebra by $\LL(n,k+1)$. (Cohn's notation in~\cite{cohn11} for this algebra is $V_{n,n+k}$.)
Throughout the text we sometimes denote $\LL(1,k)$ by $\LL_k$. 

Assigning $\deg(y_{ji})=(0,\dots,0,1,0\dots,0)$ and $\deg(x_{ij})=(0,\dots,0,-1,0\dots,0)$,  $1\leq i \leq n+k$, $1\leq j \leq n$, in $\bigoplus_n \mathbb Z$, where $1$ and $-1$ are in $j$-th entries respectively, makes the free algebra generated by $x_{ij}$ and $y_{ji}$ a graded ring. Furthermore, one can easily observe that the relations coming from ~(\ref{breaktr}) are all homogeneous with respect to this grading, so that the Leavitt algebra $\LL(n,k+1)$  is a  $\bigoplus_n \mathbb Z$-graded ring. Therefore, $\LL(1,k+1)$ is a $\mathbb Z$-graded ring. In fact this is a strongly graded ring by Theorem~\ref{hazst}. 

\subsection{Graded IBN and graded type}\label{gtr5654} A ring $A$ has an {\it invariant basis number} (IBN) if any two bases of a free (right) $A$-module have the same cardinality, i.e., if $A^n \cong A^{m}$ as $A$-modules, then $n=m$. When $A$ does not have IBN, the {\it type} of $A$ is defined as a pair of positive integers $(n,k)$ such that $A^n \cong A^{n+k}$ as $A$-modules and these are the smallest number with this property. It was shown that if $A$ has type $(n,k)$, then $A^m\cong A^{m'}$ if and only if $m=m'$ or $m,m' >n$ and $m \equiv m' \pmod{k}$ (see \cite[p. 225]{cohn11}, \cite[Theorem~1]{vitt62}).

A graded ring $A$ has a {\it graded invariant basis number} (gr-IBN) if any two homogeneous bases of a graded free (right) $A$-module have the same cardinality, i.e., if $ A^m(\overline \alpha)\cong_{\gr}  A^n(\overline \delta) $, where  $\overline \alpha=(\alpha_1,\dots,\alpha_m)$ and $\overline \delta=(\delta_1,\dots,\delta_n)$, then $m=n$. Note that contrary to the non-graded case, this does not imply that two graded free modules with bases of the same cardinality are graded isomorphic (see Proposition~\ref{rndconggrrna}). A graded ring $A$ has {\it IBN in} $\mbox{gr-}A$, if $A^m \cong_{\gr} A^n$ then $m=n$. If $A$ has IBN in $\mbox{gr-}A$, then $A_0$ has IBN. Indeed, if $A_0^m \cong A_0^n$ as $A_0$-modules, then $A^m\cong_{\gr} A_0^m \otimes_{A_0} A \cong A_0^n \otimes_{A_0} A \cong_{\gr}A^n$, so $n=m$ (see \cite[p. 215]{grrings}).

When the graded ring $A$ does not have gr-IBN, the {\it graded type} of $A$ is defined as a pair of positive integers $(n,k)$ such that $A^n(\overline \delta) \cong_{\gr} A^{n+k}(\overline \alpha)$ as $A$-modules, for some 
$\overline \delta=(\delta_1,\dots,\delta_n)$ and $\overline \alpha=(\alpha_1,\dots,\alpha_{n+k})$ and these are the smallest number with this property. 

In Proposition~\ref{grtypel1} we show that the Leavitt algebra $\LL(n,k+1)$ has graded type $(n,k)$. Using graded $K$-theory we will also show that  (see Corollary~\ref{k43shab})  
\[\LL(1,n)^k (\la_1,\dots,\la_k) \cong_{\gr} \LL(1,n)^{k'}(\ga_1,\dots,\ga_{k'})\] if and only if 
$\sum_{i=1}^k n^{\la_i}=\sum_{i=1}^{k'} {n}^{\ga_i}$.

Let $A$ be a $\Gamma$-graded ring such that $A^m(\overline \alpha) \cong_{\gr} A^n(\overline \delta)$, where $\overline \alpha=(\alpha_1,\dots,\alpha_m)$ and $\overline \delta=(\delta_1,\dots,\delta_n)$. Then there is a universal $\Gamma$-graded ring $R$ such that $ R^m(\overline \alpha)  \cong_{\gr} R^n(\overline \delta)$, and a graded ring homomorphism $R\rightarrow A$ which induces the graded isomorphism 
$A^m(\overline \alpha) \cong_{\gr} R^m(\overline \alpha) \otimes_R A \cong_{\gr} R^n(\overline \de) \otimes_R A \cong_{\gr} A^n(\overline \de)$. Indeed, by Proposition~\ref{rndconggrrna}, there are matrices 
$a=(a_{ij}) \in \M_{n\times m} (A)[\overline \delta][\overline \alpha]$ and 
$b=(b_{ij}) \in \M_{m\times n} (A)[\overline \alpha][\overline \delta]$ such that $ab=\IM_{n}$ and $ba=\IM_{m}$.
The free ring generated by symbols in place of $a_{ij}$ and $b_{ij}$ subject to relations imposed by $ab=\IM_{n}$ and $ba=\IM_{m}$ is the desired universal graded ring. In detail, let $F$ be a free ring generated by $x_{ij}$, $1\leq i \leq n$, $1\leq j \leq m$  and $y_{ij}$, $1\leq i \leq m$, $1\leq j \leq n$. Assign the degrees $\deg(x_{ij})=\delta_i-\alpha_j$ and $\deg(y_{ij})=\alpha_i-\delta_j$. This makes $F$ a $\Gamma$-graded ring. 
Let $R$ be a ring $F$ modulo the relations $\sum_{s=1}^m x_{is}y_{sk}=\delta_{ik}$, $1\leq i,k \leq n$ and $\sum_{t=1}^n y_{it}x_{tk}=\delta_{ik}$, $1\leq i,k \leq m$, where $\delta_{ik}$ is the Kronecker delta. 
Since all the relations are homogeneous, $R$ is a $\Gamma$-graded ring. Clearly the map sending $x_{ij}$ to $a_{ij}$ and $y_{ij}$ to $b_{ij}$ induces a graded ring homomorphism $R \rightarrow A$. Again Proposition~\ref{rndconggrrna} shows that $R^m(\overline \alpha) \cong_{\gr} R^n(\overline \delta)  $.

\begin{proposition}\label{grtypel1} Let $R=\LL(n,k+1)$ be the  Leavitt algebra of type $(n,k)$. Then 
\begin{enumerate}[\upshape(1)]

\item $R$ is a universal $\bigoplus_n \mathbb Z$-graded ring  which does not have gr-IBN.

\smallskip

\item $R$ has  graded type $(n,k)$

\smallskip

\item For $n=1$, $R$ has IBN in  $R\mbox{-gr}$. 

\end{enumerate}
\end{proposition}
\begin{proof}
(1) Consider the Leavitt algebra $\LL(n,k+1)$ constructed in (\ref{breaktr}), which is a $\bigoplus_n \mathbb Z$-graded ring and is universal. Furthermore (\ref{breaktr}) combined by Proposition~\ref{rndconggrrna}(3) shows that $R^n \cong_{\gr} R^{n+k}(\overline \alpha)$. Here $\overline \alpha =(\alpha_1,\dots,\alpha_{n+k})$, where $\alpha_i=(0,\dots,0,1,0\dots,0)$ and $1$ is in $i$-th entry. This shows that $R=\LL(n,k+1)$ does not have gr-IBN. 

(2) By \cite[Theorem~6.1]{cohn11}, $R$ is of  type $(n,k)$. This immediately implies the graded type of 
$R$ is also $(n,k)$. 

(3) Suppose $R^n \cong_{\gr} R^m$. Then $R_0^n \cong R_0^m$ as $R_0$-modules. But $R_0$ is an ultramatricial algebra, i.e., direct limit of matrices over a field. Since IBN respects direct limits (\cite[Theorem~2.3]{cohn11}), $R_0$ has IBN. Therefore, $n=m$.
\end{proof}

\begin{remark}
Assigning $\deg(y_{ij})=1$ and $\deg(x_{ij})=-1$, for all $i,j$, makes $R=\LL(n,k+1)$ a $\mathbb Z$-graded algebra of graded type $(n,k)$ with $R^n\cong_{\gr} R^{n+k}(1)$. 
\end{remark}

Recall from~\S\ref{pregr529} that $\Pgr(A)$ denotes the category of graded finitely generated  projective right $A$-modules with graded homomorphisms and $\mathcal T_\alpha (P)=P(\alpha)$, where $\mathcal T_\alpha$ is a suspension functor and $P$ a graded finitely generated $A$-module.

\begin{prop}[Graded Morita Equivalence] \label{grmorita}
Let $A$ be a $\Gamma$-graded ring and let $\ol \delta = (\de_1 , \ldots ,
\de_n)$, where $\de_i \in \Ga$, $1\leq i \leq n$. Then the functors
\begin{align*}
\psi : \Pgr( \M_n(A)(\ol \delta) ) & \longrightarrow \Pgr( A) \\
P & \longmapsto P  \otimes_{\M_n(A)(\ol \delta)} A^n(-\ol \delta) \\
\textrm{ and \;\;\;\; } \va : \Pgr( A) & \longrightarrow \Pgr ( \M_n(A)(\ol \delta)) \\
Q &  \longmapsto Q \otimes_A A^n(\ol \delta)
\end{align*}
form equivalences of categories and commute with suspensions, i.e, $\psi \mathcal T_{\alpha}=\mathcal T_{\alpha} \psi$, $\alpha \in \Gamma$. 
\end{prop}

\begin{proof}
One can check that there is a graded $A$-module isomorphism 
\begin{align*}
\theta : \; A^n(\ol\de)\otimes_{\M_n(A)(\ol\de)} A^n(-\ol \de)  &  \longrightarrow A \\
(a_1 , \ldots , a_n ) \otimes (b_1, \ldots , b_n) & \longmapsto  a_1
b_1 + \cdots + a_n b_n. 
\end{align*}
Furthermore, there is a graded $\M_n(A)(\ol \de)$-isomorphism 
\begin{align*}
\theta' : \; A^n(-\ol\de) \otimes_{A} A^n(\ol\de) &  \longrightarrow \M_n(A)(\ol\de)\\
\begin{pmatrix}
a_1 \\ \vdots \\ a_n
\end{pmatrix}
\otimes
\begin{pmatrix}
b_1 \\ \vdots \\ b_n
\end{pmatrix}
& \longmapsto
\begin{pmatrix}
a_{1} b_1 & \cdots & a_{1} b_n  \\
\vdots  &   & \vdots  \\
a_{n} b_1  & \cdots & a_{n} b_n
\end{pmatrix}
\end{align*}
Now it easily follows that $\varphi \psi$ and $\psi \varphi$ are equivalent to identity functors. The general fact that $(P\otimes Q) (\alpha)= P(\alpha)\otimes Q =P \otimes Q(\alpha)$ shows that the suspension functor commutes with $\psi$ and $\phi$. 
\end{proof}

\subsection{The Graded Grothendieck group} Let $A$ be a $\Gamma$-graded ring with identity and let $\mathcal V^{\gr}(A)$ denote the monoid of isomorphism classes of graded finitely generated   projective modules over $A$. Then the graded Grothendieck group, $K_0^{\gr}(A)$, is defined as the group completion of  $\mathcal V^{\gr}(A)$. Note that for $\alpha \in \Gamma$,  the {\it $\alpha$-suspension functor} $\mathcal T_\alpha:\mbox{gr-}A\rightarrow \mbox{gr-}A$, $M \mapsto M(\alpha)$  is an isomorphism with the property
$\mathcal T_\alpha \mathcal T_\beta=\mathcal T_{\alpha + \beta}$, $\alpha,\beta\in \Gamma$.
Furthermore, $\mathcal T_\alpha$ restricts to the category of graded finitely generated  
projective modules. This induces a $\mathbb
Z[\Gamma]$-module structure on $K_i^{\gr}(A)$, $i\geq 0$, defined by $\alpha [P]
=[P(\alpha)]$ on the generators and extended naturally. 
In particular if $A$ is a $\mathbb Z$-graded then $K_i^{\gr}(A)$ is a
$\mathbb Z[x,x^{-1}]$-module for $i\geq 0$.

Recall that there is a description of $K_0$ in terms
of idempotent matrices. In the following, we can always enlarge matrices of different sizes by adding zeros in the lower right hand corner, so that they can be considered in a ring $\M_k(A)$ for a suitable $k\in \mathbb N$. Any idempotent matrix $p \in \M_n(A)$ gives rise to the finitely generated 
projective right $A$-module $pA^n$. On the other hand any finitely generated  projective
module gives rise to an idempotent matrix. We say two idempotent
matrices $p$ and $q$ are equivalent if (after suitably enlarging
them) there are matrices $x$ and $y$ such that $xy=p$ and $yx=q$.
One can show that $p$ and $q$ are equivalent if and only if they are
conjugate if and only if the corresponding finitely generated   projective modules
are isomorphic.  Therefore $K_0(A)$ can be defined as the group
completion of the monoid of equivalence classes of idempotent
matrices with addition, $[p]+[q]=\left(\begin{matrix} p & 0\\ 0 & q
\end{matrix}\right)$. In fact, this is the definition one adopts for
$K_0$ when the ring $A$ does not have identity.

A similar construction can be given in the setting of graded rings.
Since this does not seem to be documented in literature, we give the
details here for the convenience of the reader.

Let $A$ be a $\Gamma$-graded ring and $\ol
\alpha=(\alpha_1,\dots,\alpha_n)$, where $\alpha_i \in \Gamma$. In
the following if we need to enlarge a homogeneous matrix $p\in
\M_n(A)(\ol \alpha)$, by adding zeroes in the lower right hand
corner, then we add zeros in the right hand side of $\ol
\alpha=(\alpha_1,\dots,\alpha_n)$ as well accordingly (and call it
$\ol \alpha$ again) so that $p$ is a homogeneous matrix in
$\M_k(A)(\ol \alpha)$, where $k\geq n$. Recall the definition of $ \M_k(A)[\ol \alpha][\ol \delta]$ from \S\ref{matgrhe} and note that if $x \in  \M_k(A)[\ol \alpha][\ol \delta]$ and $y \in  \M_k(A)[\ol \de][\ol \alpha]$ then $xy \in \M_k(-\ol \alpha)_0$ and $yx\in \M_k(-\ol\de)_0$.

\begin{deff}\label{equison}
Let $A$ be a $\Gamma$-graded ring, $\ol
\alpha=(\alpha_1,\dots,\alpha_n)$ and $\ol
\delta=(\delta_1,\dots,\delta_m)$, where $\alpha_i, \delta_j \in
\Gamma$. Let $p\in \M_n(A)(\ol \alpha)_0$ and $q\in \M_m(A)(\ol
\delta)_0$ be idempotent matrices (which are homogeneous of degree
zero). Then $p$ and $q$ are {\it grade equivalent} if (after suitably
enlarging them) there are $x \in \M_k(A)[-\ol \alpha][-\ol \delta]$ and
$y \in \M_k(A)[-\ol \delta][-\ol \alpha]$ such that $xy=p$ and $yx=q$.
\end{deff}

\begin{lemma}\label{kzeiogr}\hfill
\begin{enumerate}[\upshape(1)]

\item Any graded finitely generated    projective module gives rise to a homogeneous idempotent matrix of degree zero.

\smallskip 

\item Any homogeneous idempotent matrix of degree zero gives rise to a graded finitely generated  projective module.

\smallskip 

\item Two  homogeneous idempotent matrices of degree zero are graded equivalent if and only if the corresponding graded finitely generated  projective modules are graded isomorphic.
\end{enumerate}
\end{lemma}
\begin{proof}
(1) Let $P$ be a graded finitely generated   projective (right) $A$-module. Then there is a graded module $Q$ such that $P\oplus
Q\cong_{\gr} A^n(-\ol \alpha)$ for some $n \in \mathbb N$ and $\ol
\alpha=(\alpha_1,\dots,\alpha_n)$, where $\alpha_i \in \Gamma$.
Define the homomorphism $f \in \End_A(A^n(-\ol \alpha))$ which sends
$Q$ to zero and acts as identity on $P$. Clearly, $f$ is an
idempotent and graded homomorphism of degree $0$. Thus (see~(\ref{hgd543p}) and \S\ref{matgrhe})
\[f \in \End_A(A^n(-\ol \alpha))_0\cong \M_n(A)(\ol \alpha)_0.\]

(2) Let $p\in \M_n(A)(\ol \alpha)_0$, $\ol
\alpha=(\alpha_1,\dots,\alpha_n)$, where $\alpha_i \in \Gamma$. Then
$1-p \in \M_n(A)(\ol \alpha)_0$ and
\[A^n(-\ol \alpha) = pA^n(-\ol \alpha) \oplus (1-p)A^n(-\ol \alpha).\] This shows that $pA^n(-\ol \alpha)$
is a graded finitely generated   projective $A$-module.

\smallskip 

(3) Let  $p\in \M_n(A)(\ol \alpha)_0$ and $q\in \M_m(A)(\ol
\delta)_0$ be graded equivalent idempotent matrices. By
Definition~\ref{equison}, there are $x' \in M_k(A)[-\ol \alpha][-\ol
\delta]$ and $y' \in M_k(A)[-\ol \delta][-\ol \alpha]$ such that
$x'y'=p$ and $y'x'=q$. Let $x=px'q$ and $y=qy'p$. Then
$xy=px'qy'p=p(x'y')^2p=p$ and similarly $yx=q$. Furthermore
$x=px=xq$ and $y=yp=qy$. Now the left multiplication by $x$ and $y$
induce right graded $A$-homomorphisms $qA^k(-\ol \delta)\rightarrow pA^k(-\ol \alpha)$ and 
$pA^k(-\ol \alpha)\rightarrow qA^k(-\ol
\delta)$, respectively, which are inverse of each other. Therefore $pA^k(-\ol
\alpha)\cong_{\gr} qA^k(-\ol \delta)$.

On the other hand if $f:pA^k(-\ol \alpha)\cong_{\gr} qA^k(-\ol
\delta)$, then extend $f$ to $A^k(-\ol \alpha)$ by sending $(1-p)A^k(-\ol
\alpha)$ to zero and thus define a map
\[\theta:A^k(-\ol \alpha)=pA^k(-\ol \alpha)\oplus (1-p)A^k(-\ol \alpha) \longrightarrow qA^k(-\ol \delta)\oplus (1-q)A^k(-\ol \delta)=A^k(-\ol \delta).\]
Similarly, extending $f^{-1}$ to $A^k(-\ol \delta)$, we get a map
\[\phi:A^k(-\ol \delta)=qA^k(-\ol \delta)\oplus (1-q)A^k(-\ol \delta) \longrightarrow pA^k(-\ol \alpha)\oplus (1-p)A^k(-\ol \alpha)=A^k(-\ol \alpha)\]
such that $\theta \phi=p$ and $\phi \theta=q$. It follows (see the proof of Proposition~\ref{rndconggrrna})
$\theta \in M_k(A)[-\ol \alpha][-\ol \delta]$ whereas  $\phi \in
M_k(A)[-\ol \delta][-\ol \alpha]$. This gives that $p$ and $q$ are
equivalent.
\end{proof}

Lemma~\ref{kzeiogr} shows that $K_0^{\gr}(A)$ can be defined as the group completion of the moniod of equivalent classes of homogeneous idempotent matrices of degree zero with addition, $[p]+[q]=\left(\begin{matrix}
p & 0\\ 0 & q \end{matrix}\right)$. In fact, this is the definition we adopt for $K_0^{\gr}$ when the graded ring $A$ does not have identity. Note that the action of $\Gamma$ on the idempotent matrices (so that $K^{\gr}_0(A)$ becomes a $\mathbb Z[\Gamma]$-module with this definition) is as follows: For a $\ga \in \Ga$ and $p\in \M_n(A)(\ol \alpha)_0$, $\ga p$ is represented by the same matrix as $p$ but considered in $\M_n(A)(\ol \alpha+\ga)_0$ where 
$\overline \alpha+\ga=(\alpha_1+\ga,\cdots, \alpha_n+\ga)$. A quick inspection of the proof of Lemma~\ref{kzeiogr} shows that the action of $\Gamma$ is compatible in both definitions of $K^{\gr}_0$. 

In the case of graded fields, one can compute the graded Grothendieck group completely. This will be used to compute $K^{\gr}_0$ of acyclic and comet graphs in Theorems~\ref{f4j5h6h8} and \ref{cothemp}. 

\begin{prop} \label{k0grof}
Let $A$ be a $\Gamma$-graded field with the support the subgroup $\Gamma_A$. Then 
the monoid of  isomorphism classes of $\Gamma$-graded finitely generated  projective $A$-modules is $\mathbb N[\Gamma /\Gamma_A]$
and  $K_0^{\gr}(A) \cong \mathbb Z[\Gamma /\Gamma_A]$ as a $\mathbb Z[\Gamma]$-modules. 
Furthermore, $[A^n(\delta_1,\dots,\delta_n)]\in K_0^{\gr}(A)$
corresponds to
 $\sum_{i=1}^n \underline \delta_i$ in $\mathbb Z[\Gamma/\Gamma_A]$, where 
 $\underline \delta_i=\Gamma_A+\delta_i$.
 In particular if $A$ is a (trivially graded) field, $\Gamma$ a group and $A$ considered as 
a graded field concentrated in
degree zero, then $K_0^{\gr}(A) \cong \mathbb Z[\Gamma]$ and  $[A^n(\delta_1,\dots,\delta_n)]\in K_0^{\gr}(A)$
corresponds to
 $\sum_{i=1}^n \delta_i$ in $\mathbb Z[\Gamma]$.
\end{prop}

\begin{proof}
By Proposition~\ref{rndconggrrna}, $A(\delta_1) \cong_{\gr} A(\delta_2)$ if and only if $\delta_1-\delta_2 \in \Gamma_A$. 
Thus any graded free module of rank $1$ is graded isomorphic to some $A(\delta_i)$, where $\{\delta_i\}_{i \in I}$ is a complete set of coset representative of the subgroup $\Gamma_A$ in $\Gamma$, i.e., $\{\Gamma_A+\delta_i, i\in I\}$ represents $\Gamma/\Gamma_A$. Since any graded finitely generated  module $M$ over $A$ is graded free (see~\S\ref{pregr529}),  
\begin{equation}\label{m528h}
M \cong_{\gr} A(\delta_{i_1})^{r_1} \oplus \dots \oplus A(\delta_{i_k})^{r_k},
\end{equation}
where $\delta_{i_1},\dots \delta_{i_k}$ are distinct elements of the coset representative. Now suppose 
\begin{equation}\label{m23gds}
M \cong_{\gr} A(\delta_{{i'}_1})^{{s}_1} \oplus \dots \oplus A(\delta_{{i'}_{k'}})^{s_{k'}}.
\end{equation}
Considering the $A_0$-module $M_{-\delta_{i_1}}$, from~(\ref{m528h}) we have $M_{-\delta_{i_1}}=A_0^{r_1}$. This implies one of $\delta_{{i'}_j}$, $1\leq j \leq k'$, say, $\delta_{{i'}_1}$, has to be $\delta_{i_1}$ and so $r_1=s_1$ as $A_0$ is a field.  Repeating  the same argument for each $\delta_{i_j}$, $1\leq j \leq k$, we see  
 $k=k'$, $\delta_{{i'}_j}=\delta_{i_j}$ and $r_j=s_j$, for all $1\leq j \leq k$ (possibly after suitable permutation).  Thus any graded finitely generated  projective $A$-module can be written uniquely as $M \cong_{\gr} A(\delta_{i_1})^{r_1} \oplus \dots \oplus A(\delta_{i_k})^{r_k}
$,  where $\delta_{i_1},\dots \delta_{i_k}$ are distinct elements of the coset representative. The (well-defined) map 
\begin{align}
\mathcal V^{\gr}(A) & \rightarrow \mathbb N[\Gamma/\Gamma_A] \label{attitude}\\
[A(\delta_{i_1})^{r_1} \oplus \dots \oplus A(\delta_{i_k})^{r_k}
] & \mapsto r_1(\Gamma_A+\de_{i_1})+\dots+ r_k(\Gamma_A+\de_{i_k}), \notag
\end{align}
gives a $\mathbb N[\Gamma]$-monoid isomorphism between the monoid of  isomorphism classes of $\Gamma$-graded finitely generated  projective $A$-modules  $\mathcal V^{\gr} (A)$ and  $\mathbb N[\Gamma /\Gamma_A]$. The rest of the proof follows easily. 
\end{proof}

\begin{example}\label{upst}
Using Proposition~\ref{k0grof}, we calculate the graded $K_0$ of two types of graded fields and we determine the action of $\mathbb Z[x,x^{-1}]$ on these groups. These are graded fields obtained from graded $K_0$ of Leavitt path algebras of acyclic and $C_n$-comet graphs, respectively (see Definitions~\ref{mulidef} and~\ref{cometi}). 
\begin{enumerate}
\item Let $K$ be a field. Consider $A=K$ as a $\mathbb Z$-graded field with the support $\Gamma_A=0$, i.e., $A$ is concentrated in degree $0$.  By Proposition~\ref{k0grof}, $K_0^{\gr}(A)\cong \mathbb Z[x,x^{-1}]$ as a $\mathbb Z[x,x^{-1}]$-module. 

\smallskip 

\item Let $A=K[x^n,x^{-n}]$ be $\mathbb Z$-graded field with $\Gamma_A=n\mathbb Z$. By Proposition~\ref{k0grof}, 
\[ K_0^{\gr}(A)\cong \mathbb Z \big [ \mathbb Z / n \mathbb Z \big ]\cong \bigoplus_n \mathbb Z,\] is a $\mathbb Z[x,x^{-1}]$-module. The action of $x$ on $(a_1,\dots,a_n) \in \bigoplus_n \mathbb Z$ is $x (a_1,\dots,a_n)= (a_n,a_1,\dots,a_{n-1})$. 
\end{enumerate}
\end{example}

\subsection{Dade's theorem}\label{dadestmal} 
Let $A$ be a strongly $\Gamma$-graded ring. (This implies $\Gamma_A=\Gamma$.) By Dade's Theorem \cite[Thm.~3.1.1]{grrings}, the functor $(-)_0:\mbox{gr-}A\rightarrow \mbox{mod-}A_0$, $M\mapsto M_0$, is an additive functor with an  inverse $-\otimes_{A_0} A: \mbox{mod-}A_0 \rightarrow \mbox{gr-}A$ so that it induces an equivalence of categories.  
 This implies
that $K_i^{\gr}(A)  \cong K_i(A_0)$, for $i\geq 0$.  Furthermore, since $A_\alpha \otimes_{A_0}A_\beta \cong A_{\alpha +\beta}$ as $A_0$-bimodule, the functor $\mathcal T_\alpha$ on $\mbox{gr-}A$ induces a functor on the level of
$\mbox{mod-}A_0$, $\mathcal T_\alpha:\mbox{mod-}A_0 \rightarrow \mbox{mod-}A_0$, $M\mapsto  M \otimes_{A_0} A_\alpha$ such that $\mathcal T_\alpha \mathcal T_\beta\cong\mathcal T_{\alpha +\beta}$, $\alpha,\beta\in \Gamma$, so that the following diagram is commutative. 
\begin{equation}\label{veronaair}
\xymatrix{
\mbox{gr-}A \ar[r]^{\mathcal T_\alpha} \ar[d]_{(-)_0}& \mbox{gr-}A \ar[d]^{(-)_0}\\
\mbox{mod-}A_0 \ar[r] & \mbox{mod-}A_0
}
\end{equation}
Therefore $K_i(A_0)$ is also a $\mathbb Z[\Gamma]$-module and
\begin{equation} \label{dade}
K_i^{\gr}(A) \cong K_i(A_0),
\end{equation} as $\mathbb Z[\Gamma]$-modules.

\subsection{} \label{kidenh}

Let $A$ be a graded ring and $f:A\rightarrow A$ an inner-automorphism defined by $f(a)=r a r^{-1}$, where $r$ is a homogeneous
and invertible element of $A$ of degree $\tau$. Then considering
$A$ as a graded left $A$-module via $f$, it is easy to observe that there is a   
right graded $A$-module isomorphism $P(-\tau) \rightarrow P \otimes_f A$, $p
\mapsto p \otimes r$ (with the inverse $p\otimes a \mapsto p r^{-1}a $). 
This induces an isomorphism between the functors $-\otimes_f A :\Pgr(A) \rightarrow \Pgr(A)$ and 
$\tau$-suspension functor $\mathcal T_\tau:\Pgr(A) \rightarrow \Pgr(A)$. Recall that a Quillen's $K_i$-group, $i\geq 0$,  is a functor from the category of exact categories with exact functors to the category of abelian groups. Furthermore, isomorphic functors induce the same map on the $K$-groups \cite[p.19]{quillen}. Thus $K^{\gr}_i(f)=K^{\gr}_i(\mathcal T_\tau)$. Therefore if $r$ is a homogeneous element of degree zero, i.e., $\tau=0$, then 
$K^{\gr}_i(f):K^{\gr}_i(A)\rightarrow K^{\gr}_i(A)$ is the identity map. This will be used in Theorem~\ref{mani543}. 

\subsection{Homogeneous idempotents}\label{idemptis}

The following facts about idempotents are well known in a non-graded setting and one can check that they translate in the graded setting with the similar proofs (cf. \cite[\S 21]{lamfc}). Let $P_i$, $1\leq i \leq l$,  be homogeneous right ideals of $A$ such that $A=P_1\oplus\dots \oplus P_l$. (Recall that a homogeneous ideal, or a graded ideal, is an ideal which is generated by homogeneous elements.) Then there are homogeneous orthogonal idempotents $e_i$  (hence of degree zero) such that $1=e_1+\dots+e_l$ and $e_iA=P_i$. 

Let $e$ and $f$ be  homogeneous idempotent elements in the graded ring $A$. (Note that, in general, there are non-homogeneous idempotents in a graded ring.) Let $\theta:eA \rightarrow fA$ be a graded $A$-module homomorphism. 
Then $\theta(e)=\theta(e^2)=\theta(e)e=fae$ for some $a\in A$ and for $b \in eA$, $\theta(b)=\theta(eb)=\theta(e)b$. 
This shows there is a map 
\begin{align}\label{hgwob}
\Hom_A(eA,fA) & \rightarrow fAe\\
\theta & \mapsto \theta(e) \notag
\end{align}
and one can prove this is a graded group isomorphism. Now if $\theta:eA\rightarrow fA$ is a graded $A$-isomorphism, then $x=\theta(e)\in fAe$ and  $y=\theta^{-1}(f)\in eAf$, where $x$ and $y$ are homogeneous of degree zero, such that $yx=e$ and $xy=f$.  
Finally, for $f=e$, the map~(\ref{hgwob}) gives a graded ring isomorphism $ \End_A(eA) \cong_{\gr} eAe$.  
These facts will be used in Theorem~\ref{catgrhsf}. 
 
\subsection{$\mathbf \Gamma$-pre-ordered groups}  \label{pregg5}
Here we define the category of $\Gamma$-pre-ordered abelian groups. Classically, the objects are abelian groups with a pre-ordering. However, since we will consider graded Grothendieck groups, our groups are $\mathbb Z[x,x^{-1}]$-modules, so we need to adopt the definitions to this setting. 

Let $\Gamma$ be a group and $G$ be a (left) $\Gamma$-module. Let $\leq$ be a reflexive and transitive relation on $G$ which respects the group and the module structures, i.e., for $\gamma \in \Gamma$ and $x,y,z \in G$, if $x\leq y$, then $\gamma x \leq \gamma y$ and $x+z\leq y+z$. We call $G$ a {\it $\Gamma$-pre-ordered group}. We call $G$ a pre-ordered group when $\Gamma$ is clear from the context. The {\it cone} of $G$ is defined as $\{x \in G \mid x\geq 0\}$ and denoted by $G^{+}$.  The set $G^{+}$ is a $\Gamma$-submonoid of $G$, i.e., a submonoid which is closed under the action of $\Gamma$.  In fact, $G$ is a $\Gamma$-pre-ordered group if and only if there exists a $\Gamma$-submonoid of $G$. (Since $G$ is a $\Gamma$-module, it can be  considered as a $\mathbb Z[\Gamma]$-module.)  An {\it order-unit} in $G$ is an element $u \in G^{+}$ such that for any $x\in G$, there are $\alpha_1,\dots,\alpha_n \in \Gamma$, $n \in \mathbb N$, such that $x\leq \sum_{i=1}^n \alpha_i  u$.  As usual, we only consider homomorphisms which preserve the pre-ordering, i.e., a $\Gamma$-homomorphism $f:G\rightarrow H$, such that $f(G^{+}) \subseteq H^{+}$. We denote by $\mathcal P_{\Gamma}$ the category of pointed $\Gamma$-pre-ordered abelian groups with pointed ordered preserving homomorphisms, i.e., the objects are the pairs  $(G,u)$, where $G$ is a  $\Gamma$-pre-ordered abelian group and $u$ is an order-unit, and $f:(G,u)\rightarrow (H,v)$ is an ordered-preserving $\Gamma$-homomorphism such that $f(u)=v$. Note that when $\Gamma$ is a trivial group, we are in the classical setting of pre-ordered abelian groups. In this paper, $\Gamma=\mathbb Z$ and so $\mathbb Z[\Gamma]=\mathbb Z[x,x^{-1}]$. We simply write $\mathcal P$ for $\mathcal P_{\mathbb Z}$ throughout. 

\begin{example}
Let $A$ be a $\Gamma$-graded ring. Then $K^{\gr}_0(A)$ is a $\Gamma$-pre-ordered abelian group with  the set of isomorphic classes of  graded finitely generated  projective right $A$-modules  as the cone of ordering and $[A]$ as an order-unit. Indeed, if $x\in K^{\gr}_0(A)$, then there are graded finitely generated  projective modules $P$ and $P'$ such that $x=[P]-[P']$. But there is a graded module $Q$ such that $P\oplus Q \cong A^n(\overline \alpha)$, where $\overline \alpha =(\alpha_1,\dots,\alpha_n)$, $\alpha_i \in \Gamma$ (see \S\ref{pregr529}).  Now 
\[ [A^n(\overline \alpha)] -x = [P]+[Q]-[P]+[P']=[Q]+[P']=[Q\oplus P'] \in K^{\gr}_0(A)^{+}.\] This shows that $\sum_{i=1}^n \alpha_i [A]=[A^n(\overline \alpha)]\geq x$. 

\end{example}

\subsection{Graded matricial algebras}\label{grmat43}

The Leavitt path algebras of acyclic and $C_n$-comet graphs are  finite direct product of matrix rings over graded fields of the form $K$ and $K[x^n,x^{-n}]$, respectively (see~\cite[Theorem~4.11 and Theorem~4.17]{haz}). In this section we establish that the graded $K_0$ is a complete invariant for these algebras. This will be used in the main Theorem~\ref{mani543} to settle Conjecture~\ref{weakconj} in the case of acyclic and multi-headed comet graphs. 

\begin{deff}
Let  $A$ be a $\Gamma$-graded field. A $\Gamma$-{\it graded matricial} $A$-algebra is a graded $A$-algebra of the form \[\M_{n_1}(A)(\overline \delta_1) \times \dots \times   \M_{n_l}(A)(\overline \delta_l),\] where $\overline \delta_i=(\delta_{1}^{(i)},\dots,\delta_{n_i}^{(i)})$, $\delta_{j}^{(i)} \in \Gamma$, $1\leq j \leq n_i$ and $1\leq i \leq l$. 
\end{deff}

\begin{lemma}\label{fdpahf}
Let $A$ be a $\Gamma$-graded field and $R$ be a $\Gamma$-graded matricial $A$-algebra.  Let $P$ and $Q$ be graded finitely generated  projective $R$-modules. Then $[P]=[Q]$ in $K^{\gr}_0(R)$, if and only if  $P \cong _{\gr} Q$. 
\end{lemma}

\begin{proof}
Since the functor $K^{\gr}_0$ respects the direct sum, it suffices to prove the statement for a matricial algebra of the form 
$R=\M_{n}(A)(\overline {\delta})$. Let $P$ and $Q$ be graded finitely generated  projective $R$-modules such that $[P]=[Q]$ in $K^{\gr}_0(R)$. By Proposition~\ref{grmorita},  $R=\M_{n}(A)(\overline {\delta})$ is graded Morita equivalent to $A$. So there are equivalent functors $\psi$ and $\phi$ such that $\psi\phi\cong 1$ and $\phi \psi \cong 1$, which also induce an isomorphism $K^{\gr}_0(\psi):K^{\gr}_0(R) \rightarrow K^{\gr}_0(A)$ such that $[P]\mapsto [\psi(P)]$. Now since $[P]=[Q]$, it follows $[\psi(P)]=[\psi(Q)]$  in $K^{\gr}_0(A)$. But since $A$ is a graded field, by the proof of Proposition~\ref{k0grof}, any graded finitely generated  projective $A$-module can be written uniquely 
as a direct sum of appropriate shifted $A$. Writing $\psi(P)$ and $\psi(Q)$ in this form, the homomorphism~(\ref{attitude}) shows that $\psi(P)\cong_{\gr} \psi(Q)$. Now applying the functor $\phi$ to this we obtain $P\cong_{\gr} Q$.
\end{proof}

Let $A$ be a $\Gamma$-graded field (with the support $\Gamma_A$) and $\mathcal C$ be a category consisting of $\Gamma$-graded matricial $A$-algebras as objects and  
$A$-graded algebra homomorphisms  as morphisms. We consider the quotient category $\mathcal C^{\out}$ obtained from $\mathcal C$ by identifying homomorphisms which differ up to a degree zero graded inner-automorphim. That is, the graded homomorphisms $\phi, \psi \in \Hom_{\mathcal C}(R,S)$ represent the same morphism in $\mathcal C^{\out}$ if there is an inner-automorpshim $\theta:S\rightarrow S$, defined by $\theta(s)=xsx^{-1}$, where $\deg(x)=0$, such that $\phi=\theta\psi$. The following theorem shows that $K^{\gr}_0$ ``classifies'' the category of $\mathcal C^{\out}$. This will be used in Theorem~\ref{mani543} where Leavitt path algebras of acyclic and multi-headed comet graphs are classified by means of  their $K$-groups. This is a graded analog of a similar theorem for matricial algebras (see~\cite{goodearlbook}).

\begin{theorem}\label{catgrhsf}
Let $A$ be a $\Gamma$-graded field and $\mathcal C^{\out}$ be the category consisting of $\Gamma$-graded matricial $A$-algebras as objects and  
$A$-graded algebra homomorphisms modulo graded inner-automorphisms as morphisms.  Then $K^{\gr}_0: \mathcal C^{\out}\rightarrow \mathcal P$ is a fully faithful functor. Namely,

\begin{enumerate}[\upshape(1)]

\item (well-defined and faithful) For any graded matricial $A$-algebras $R$ and $S$ and $\phi,\psi\in \Hom_{\mathcal C}(R, S)$,  we have  $\phi(r)=x\psi(r)x^{-1}$, $r\in R$, for some invertible homogeneous element $x$ of $S$, if and only if $K^{\gr}_0(\phi)=K^{\gr}_0(\psi)$. 

\smallskip

\item(full) For any graded matricial $A$-algebras $R$ and $S$ and  the morphism $f:(K^{\gr}_0(R),[R])\rightarrow (K^{\gr}_0(S),[S])$ in $\mathcal P$, there is $\phi\in \Hom_{\mathcal C}(R, S)$ such that $K^{\gr}_0(\phi)=f$.

\smallskip

\end{enumerate} 

\end{theorem}

\begin{proof}
(1) (well-defined.) Let $\phi,\psi\in \Hom_{\mathcal C}(R, S)$  such that $\phi=\theta\psi$, where $\theta(s)=xsx^{-1}$ for some invertible homogeneous element $x$ of $S$ of degree $0$. Then $K^{\gr}_0(\phi)=K^{\gr}_0(\theta\psi)=K^{\gr}_0(\theta)K^{\gr}_0(\psi)=K^{\gr}_0(\psi)$ since $K^{\gr}_0(\theta)$ is the identity map (see \S\ref{kidenh}). 

(faithful.) The rest of the proof is similar to the non-graded version with an extra attension given to the grading (cf. \cite[p.218]{goodearlbook}). We give it here for the completeness. Let $K^{\gr}_0(\phi)=K^{\gr}_0(\psi)$. Let $R=\M_{n_1}(A)(\overline \delta_1) \times \dots \times   \M_{n_l}(A)(\overline \delta_l)$ and set $g_{jk}^{(i)}=\phi(e^{(i)}_{jk})$ and 
$h_{jk}^{(i)}=\psi(e^{(i)}_{jk})$ for  $1\leq i \leq l$ and $1\leq j,k \leq n_i$, where $e^{(i)}_{jk}$ are the standard basis for $\M_{n_i}(A)$. Since $\phi$ and $\psi$ are graded homomorphism, 
$\deg(e^{(i)}_{jk})=\deg(g^{(i)}_{jk})=\deg(h^{(i)}_{jk})=\delta^{(i)}_j-\delta^{(i)}_k$. 

Since $e^{(i)}_{jj}$ are pairwise graded orthogonal idempotents (of degree zero)  in $R$  and  
$\sum_{i=1}^l \sum_{j=1}^{n_i} e^{(i)}_{jj}=1$,  the same is also the case for $g^{(i)}_{jj}$ and $h^{(i)}_{jj}$. 
Then 
\[[g^{(i)}_{11}S]=K^{\gr}_0(\phi)([e^{(i)}_{11}R])=K^{\gr}_0(\psi)([e^{(i)}_{11}R])=[h^{(i)}_{11}S].\]
By Lemma~\ref{fdpahf}, $g^{(i)}_{11}S \cong_{\gr} h^{(i)}_{11}S$. Thus there are homogeneous elements of degree zero $x_i$ and $y_i$ such that $x_iy_i=g^{(i)}_{11}$ and $y_ix_i=h^{(i)}_{11}$ (see \S\ref{idemptis}). Let 
$x=\sum_{i=1}^{l}\sum_{j=1}^{n_i}g^{(i)}_{j1}x_ih^{(i)}_{1j}$ and 
$y=\sum_{i=1}^{l}\sum_{j=1}^{n_i}h^{(i)}_{j1}y_ig^{(i)}_{1j}$. Note that $x$ and $y$ are homogeneous elements of degree zero. One can easily check that $xy=yx=1$. Now for  $1\leq i \leq l$ and $1\leq j,k \leq n_i$ we have 
\begin{align*}
x h^{(i)}_{jk}&=\sum_{s=1}^{l}\sum_{t=1}^{n_s}g^{(s)}_{t1}x_sh^{(s)}_{1t}h^{(i)}_{jk}\\
& = g^{(i)}_{j1}x_ih^{(i)}_{1j}h^{(i)}_{jk}=g^{(i)}_{jk}g^{(i)}_{k1}x_ih^{(i)}_{1k}\\
&= \sum_{s=1}^{l}\sum_{t=1}^{n_s}g^{(i)}_{jk}g^{(s)}_{t1}x_sh^{(i)}_{1t}= g^{(i)}_{jk}x
\end{align*}
Let $\theta:S\rightarrow S$ be the graded inner-automorphism $\theta(s)=xsy$. Then $\theta\psi(e^{(i)}_{jk}))=
xh^{(i)}_{jk}y=g^{(i)}_{jk}=\phi(e^{(i)}_{jk})$. Since $e^{(i)}_{jk}$, $1\leq i \leq l$ and $1\leq j,k \leq n_i$, 
form a homogeneous $A$-basis for $R$,  $\theta\psi=\phi$. 

\smallskip

(2) Let $R=\M_{n_1}(A)(\overline \delta_1) \times \dots \times   \M_{n_l}(A)(\overline \delta_l).$
Consider $R_i= \M_{n_i}(A)(\overline \delta_i)$, $1\leq i \leq l$. 
Each $R_i$ is a graded finitely generated   projective $R$-module, so  $f([R_i])$ is in the cone of $K^{\gr}_0(S)$, i.e., there is a graded finitely generated  projective $S$-module $P_i$ such that $f([R_i])=[P_i]$. Then 
\[[S]=f([R])=f([R_1]+\dots+[R_l])=[P_1]+\dots+[P_l]=[P_1\oplus\dots\oplus P_l].\] Since $S$ is a graded matricial algebra, by Lemma~\ref{fdpahf}, $P_1\oplus\dots\oplus P_l \cong_{\gr} S$ as a right $S$-module. So there are homogeneous orthogonal idempotents $g_1,\dots,g_l \in S$ such that $g_1+\dots+g_l=1$ and $g_iS\cong_{\gr} P_i$ (see \S\ref{idemptis}). 
Note that each of $R_i=\M_{n_i}(A)(\overline \delta_i)$ are graded simple algebras. Set $\overline \delta_i=(\delta_1^{(i)},\dots,\delta_n^{(i)})$ (here $n=n_i$). Let $e_{jk}^{(i)}$,$1\leq j,k\leq n$, be matrix units of $R_i$ and consider 
the graded finitely generated   projective (right) $R_i$-module, $V=e_{11}^{(i)} R_i=A(\de_1^{(i)} - \de_1^{(i)}) \oplus  A(\de_2^{(i)}  - \de_1^{(i)}) \oplus \cdots \oplus
A(\de_n^{(i)} - \de_1^{(i)})$. Then~(\ref{pjacko}) shows \[R_i \cong_{\gr} V(\de_1^{(i)} - \de_1^{(i)})\oplus V(\de_1^{(i)}  - \de_2^{(i)}) \oplus \dots \oplus V(\de_1^{(i)} - \de_n^{(i)}),\] as graded $R$-module. Thus 
\begin{equation}\label{lkheohg}
[P_i]=[g_iS]=f([R_i])=f([V(\de_1^{(i)} - \de_1^{(i)})])+f([V(\de_1^{(i)} - \de_2^{(i)})])+\dots+f([V(\de_1^{(i)} - \de_n^{(i)})]).
\end{equation}
There is a graded finitely generated  projective $S$-module $Q$ such that $f([V])=f([V(\de_1^{(i)} - \de_1^{(i)})])=[Q]$. Since $f$ is a $\mathbb Z[\Gamma]$-module homomorphism, for $1\leq k \leq n$, \[f([V(\de_1^{(i)} - \de_k^{(i)})])=f((\de_1^{(i)} - \de_k^{(i)})[V])=(\de_1^{(i)} - \de_k^{(i)})f([V])=(\de_1^{(i)} - \de_k^{(i)})[Q]=[Q(\de_1^{(i)} - \de_k^{(i)})].\] 
From ~(\ref{lkheohg}) and Lemma~\ref{fdpahf} now follows 
\begin{equation}\label{bvgdks}
g_iS\cong_{\gr} Q(\de_1^{(i)} - \de_1^{(i)})\oplus Q(\de_1^{(i)}  - \de_2^{(i)}) \oplus \dots \oplus Q(\de_1^{(i)} - \de_n^{(i)}).
\end{equation}
Let $g^{(i)}_{jk}\in \End(g_iS)\cong_{\gr} g_iSg_i$  maps the $j$-th summand of the right hand side of~(\ref{bvgdks}) to its $k$-th summand and everything else to zero. Observe that $\deg(g^{(i)}_{jk})=\de^{(i)}_j-\de^{(i)}_k$ and $g^{(i)}_{jk}$, $1\leq j,k\leq n$, form  the matrix units. Furthermore, $g^{(i)}_{11}+\dots+g^{(i)}_{nn}=g_i$ and $g^{(i)}_{11}S=Q(\delta^{(i)}_1-\delta^{(i)}_1)=Q$. Thus $[g^{(i)}_{11}S]=[Q]=f([V])=f([e_{11}^{(i)}R_i])$.

Now for $1\leq i\leq l$, define the $A$-algebra homomorphism $R_i\rightarrow g_iSg_i$, $e_{jk}^{(i)} \mapsto g_{jk}^{(i)}$. This is a graded homomorphism, and induces a graded homorphism $\phi:R\rightarrow S$ such that 
$\phi(e_{jk}^{(i)})= g_{jk}^{(i)}$. Clearly \[K_0^{\gr}(\phi)([e_{11}^{(i)}R_i])=[\phi(e_{11}^{(i)})S]=[g_{11}^{(i)}S]=f([e_{11}^{(i)}R_i]),\] for $1\leq i \leq l$. Now $K^{\gr}_0(R)$ is generated by $[e_{11}^{(i)}R_i]$, $1\leq i \leq l$, as $\mathbb Z[\Gamma]$-module. This implies that $K^{\gr}_0(\phi)=f$. 
\end{proof}

 
\forget 
 
 \begin{theorem}\label{gfw76}
Let $E$ and $F$ be finite graphs with no sinks.
Then $\LL(E)\cong_{\gr} \LL(F)$ if and only if there is a $\mathbb Z[x,x^{-1}]$-module isomorphism
\begin{equation}\label{ncfd1}
\big (K_0^{\gr}(\LL(E)),[\LL(E)]\big ) \cong \big (K_0^{\gr}(\LL(F)),[\LL(F)]\big ).
\end{equation}
\end{theorem}

\begin{proof}
Set $R=\LL(E)$ and $S=\LL(F)$. By Theorem~\ref{hazst}, $R$ and $S$ are strongly graded rings. Using Dade's theorem (see \S\ref{dadestmal}), the isomorphism (\ref{ncfd1}), call it $f$,  induces an ordered preserving isomorphism 
\begin{equation}\label{lkadff}
f_0: \big (K_0(R_0),[R_0]\big ) \cong \big (K_0(S_0),[S_0]\big ),
\end{equation}
(see Diagram~\ref{gsqqwq}), defined on generators by $f_0([M])=f([M\otimes_{R_0}R])_0$, where $M$ is a finitely generated   projective $R_0$-module. 
\begin{equation}\label{gsqqwq}
\xymatrix{
K^{\gr}_0(R) \ar[rr]^{f} && K^{\gr}_0(S) \ar[d]^{(-)_0}\\
K_0(R_0) \ar[u]^{-\otimes_{R_0}R} \ar@{.>}[rr]^{f_0} && K_0(S_0)
}
\end{equation}

Since $R_0$ and $S_0$
 are ultramatricial algebras, by Elliott's theorem (see \cite[Theorem 15.26]{goodearl}), there exists an isomorphism 
$\rho:R_0 \rightarrow S_0$ which induces the isomorphism $f_0$ on $K_0$-groups, i.e., $f_0=K_0(\rho)$. 
Now define the functor $\mathcal G: \mbox{gr-}R \rightarrow \mbox{gr-}S$  as the composition of the functors of the diagram below. 
 \[
\xymatrix{
\mbox{gr-}R \ar@{.>}[rr]^{\mathcal G} \ar[d]_{(-)_0}&& \mbox{gr-}S \\
\mbox{mod-}R_0 \ar[rr]^{-\otimes_{R_0} {S_0}} && \mbox{mod-}S_0 \ar[u]_{-\otimes_{S_0}S}
}
\]
Thus for graded finitely generated  projective $R$-module $P$, $\mathcal G(P)=P_0\otimes_{R_0}S$. 
 Note that since all the functors involved in the diagram are equivalences, so is $\mathcal G$.

We show that $\mathcal G$ commutes with the suspension functors $\mathcal T_\alpha$ for any $\alpha \in \Gamma$, up to isomorphism, i.e., there are natural graded $S$-isomorphisms $\mathcal G (P(\alpha)) \cong \mathcal G(P)(\alpha)$, where $P$ is a graded finitely generated   projective $R$-module. 
Since $S$ is strongly graded, it suffices to prove $\mathcal G \big (P(\alpha)\big)_0 \cong \big (\mathcal G (P)(\alpha)\big)_0$ as $S_0$-modules. Since $S_0$ is an ultramatricial algebra and ultramatricial algebras are unit-regular, by Proposition~15.2 in~\cite{goodearlbook}, it suffices to show  
\begin{equation} \label{ks5543}
\big [\mathcal G \big (P(\alpha)\big)_0\big] = \big [\big(\mathcal G (P)(\alpha)\big)_0\big]
\end{equation}
 in $K_0(S_0)$. But this is the case as all the functors involved (Diagram~\ref{gsqqwq}) induce $\mathbb Z[\Gamma]$-module homomorphism on the level of $K$-groups. In details, if   
$\big [\mathcal G \big (P(\alpha)\big)\big] = \big [\big(\mathcal G (P)(\alpha)\big)\big]$ in $K^{\gr}_0(S)$ then applying the functor $(-)_0$ we obtain (\ref{ks5543}). But since $K_0(\mathcal G)=f$, we have 
\[\big [\mathcal G \big (P(\alpha)\big)\big]=f([P(\alpha)])= f(\alpha [P]) =\alpha f([P]) =\alpha [\mathcal G (P)]=\big [(\mathcal G (P)(\alpha)\big].\]

So $\mathcal G$ is a graded equivalence between $\mbox{gr-}R$ and $\mbox{gr-}S$ which commutes with suspension functors up to isomorphism, i.e.,  $\mathcal G\mathcal T_\alpha\cong \mathcal T_\alpha \mathcal G$ for $\alpha \in \Gamma$. 
Consider the group homomorphism induced by the functor $\mathcal G_\alpha$,
\[\mathcal G_\alpha: \Hom_{\mbox{gr-}R}\big (R,R(\alpha)\big )\longrightarrow \Hom_{\mbox{gr-}S}\big (S,S(\alpha)\big ),\]
Thus for $h:R \rightarrow R(\alpha)$ we have 
\[\mathcal G_\alpha(h): S=(R_0\otimes_{R_0} S_0)\otimes_{S_0} S \stackrel{f_0\otimes 1}\longrightarrow (R_n\otimes_{R_0} S_0)\otimes_{S_0} S=S(\alpha).\]
Check that if $h\in  \Hom_{\mbox{gr-}R}\big (R,R(\alpha)\big )$ and $k\in  \Hom_{\mbox{gr-}R}\big (R,R(\beta)\big )$, then \[\mathcal G_{\alpha+\beta}(kh)=\mathcal G_{\beta} (k)\mathcal G_{\alpha} (h) \in \Hom_{\mbox{gr-}S}\big (S,S(\alpha+\beta)\big ).\]
Since $\mathcal G$ is an equivalence of the categories, (so it does have an inverse), then $\mathcal G_\alpha$ are isomorphism. Therefore \[\LL(E)= R\cong_{\gr} \Hom_R(R,R) = \bigoplus_{\alpha \in \Gamma} \Hom(R,R(\alpha)) \longrightarrow  \bigoplus_{\alpha \in \Gamma} \Hom(S,S(\alpha)) = \Hom_S(S,S) \cong_{\gr} S=\LL(F).\] 
\end{proof}

\begin{remark}
The last part of the proof was inspired by Proposition~5.3 in \cite{greengordon}. This proposition states that if 
$\mathcal U:\mbox{gr-}R\rightarrow \mbox{gr-}S$ is a graded equivalence, then $\mathcal U$ is isomorphism to a graded functor $\mathcal U'$ such that $\mathcal U'$ also induces an equivalence between $\mbox{mod-}R$ and $\mbox{mod-}S$, i.e., there is an equivalent functor $\mathcal M$ such that the following diagram commutes
 \[
\xymatrix{
\mbox{gr-}R \ar[rr]^{\mathcal U'} \ar[d]_{\mathcal F}&& \mbox{gr-}S \ar[d]_{\mathcal F}\\
\mbox{mod-}R \ar[rr]^{\mathcal M} && \mbox{mod-}S 
}
\] where $\mathcal F$ is the forgetful functor. 

 This can't be used directly here.  For, in \cite{greengordon} a functor is defined to be graded if it commutes with suspensions, i.e., 
$\mathcal U \mathcal T_\alpha=\mathcal T_\alpha \mathcal U$. In our setting this is not the case, i.e., we only have $\mathcal U\mathcal T_\alpha\cong \mathcal T_\alpha \mathcal U$ for $\alpha \in \Gamma$
\end{remark}

\forgotten 


\section{Leavitt path algebras}

In this section we gather some  graph-theoretic definitions and recall the basics on Leavitt path algebras, including the calculation of Grothendieck group $K_0$.  The reader familiar with this topic can skip to Section~\ref{daryak}. 

A {\it directed graph} $E=(E^0,E^1,r,s)$ consists of two countable sets $E^0$, $E^1$ and maps $r,s:E^1\rightarrow E^0$. The elements of $E^0$ are called {\it vertices} and the elements of $E^1$ {\it edges}. If $s^{-1}(v)$ is a finite set for every $v \in E^0$, then the graph is called {\it row-finite}. In this note we will consider only row-finite graphs. In this setting, if the number of vertices, i.e.,  $|E^0|$,  is finite, then the number of edges, i.e.,  $|E^1|$, is finite as well and we call $E$ a {\it finite} graph. 

 For a graph $E=(E^0,E^1,r,s)$, a vertex $v$ for which $s^{-1}(v)$ is empty is called a {\it sink}, while a vertex $w$ for which $r^{-1}(w)$ is empty is called a {\it source}. An edge with the same source and range is called a {\it loop}. A path $\mu$ in a graph $E$ is a sequence of edges $\mu=\mu_1\dots\mu_k$, such that $r(\mu_i)=s(\mu_{i+1}), 1\leq i \leq k-1$. In this case, $s(\mu):=s(\mu_1)$ is the {\it source} of $\mu$, $r(\mu):=r(\mu_k)$ is the {\it range} of $\mu$, and $k$ is the {\it length} of $\mu$ which is  denoted by $|\mu|$. We consider a vertex $v\in E^0$ as a {\it trivial} path of length zero with $s(v)=r(v)=v$. 
If $\mu$ is a nontrivial path in $E$, and if $v=s(\mu)=r(\mu)$, then $\mu$ is called a {\it closed path based at} $v$. If $\mu=\mu_1\dots\mu_k$ is a closed path based at $v=s(\mu)$ and $s(\mu_i) \not = s(\mu_j)$ for every $i \not = j$, then $\mu$ is called a {\it cycle}.  

For two vertices $v$ and $w$, the existence of a path with the source $v$ and the range $w$ is denoted by $v\geq w$. Here we allow paths of length zero. By $v\geq_n w$, we mean there is a path of length $n$ connecting these vertices. Therefore $v\geq_0 v$ represents the vertex $v$. Also, by $v>w$, we mean a path from $v$ to $w$ where $v\not = w$. In this note, by $v\geq w' \geq w$, it is understood that there is a path connecting $v$ to $w$ and going through $w'$ (i.e., $w'$ is on the path connecting $v$ to $w$). For $n\geq 2$, we define $E^n$ to be the set of paths of length $n$ and $E^*=\bigcup_{n\geq 0} E^n$, the set of all paths.

\begin{deff}\label{llkas}{\sc Leavitt path algebras.} \label{LPA} \\For a graph $E$ and a field $K$, we define the {\it Leavitt path algebra of $E$}, denoted by $\LL_K(E)$, to be the algebra generated by the sets $\{v \mid v \in E^0\}$, $\{ \alpha \mid \alpha \in E^1 \}$ and $\{ \alpha^* \mid \alpha \in E^1 \}$ with the coefficients in $K$, subject to the relations 

\begin{enumerate}
\item $v_iv_j=\delta_{ij}v_i \textrm{ for every } v_i,v_j \in E^0$.

\item $s(\alpha)\alpha=\alpha r(\alpha)=\alpha \textrm{ and }
r(\alpha)\alpha^*=\alpha^*s(\alpha)=\alpha^*  \textrm{ for all } \alpha \in E^1$.

\item $\alpha^* \alpha'=\delta_{\alpha \alpha'}r(\alpha)$, for all $\alpha, \alpha' \in E^1$.

\item $\sum_{\{\alpha \in E^1, s( \alpha)=v\}} \alpha \alpha^*=v$ for every $v\in E^0$ for which $s^{-1}(v)$ is non-empty.

\end{enumerate}
\end{deff}
Here the field $K$ commutes with the generators $\{v,\alpha, \alpha^* \mid v \in E^0,\alpha \in E^1\}$. Throughout the note, we sometimes  write $\LL(E)$ instead of $\LL_K(E)$. The elements $\alpha^*$ for $\alpha \in E^1$ are called {\it ghost edges}.

Setting $\deg(v)=0$, for $v\in E^0$, $\deg(\alpha)=1$ and $\deg(\alpha^*)=-1$ for $\alpha \in E^1$, we obtain a natural $\mathbb Z$-grading on the free $K$-ring generated by  $\{v,\alpha, \alpha^* \mid v \in E^0,\alpha \in E^1\}$. Since the relations in the above definition are all homogeneous, the ideal generated by these relations is homogeneous and thus we have a natural $\mathbb Z$-grading on $\LL_K(E)$. 

If $\mu=\mu_1\dots\mu_k$, where $\mu_i \in E^1$, is an element of $\LL(E)$, then we denote by $\mu^*$ the element $\mu_k ^*\dots \mu_1^* \in \LL(E)$. Since $\alpha^* \alpha'=\delta_{\alpha \alpha'}r(\alpha)$, for all $\alpha, \alpha' \in E^1$, any word can be written as $\mu \gamma ^*$ where $\mu$ and $\gamma$ are paths in $E$.  The elements of the form $\mu\gamma^*$ are called {\it monomials}. 

Taking the grading into account, one can write $\LL(E) =\textstyle{\bigoplus_{k \in \mathbb Z}} \LL(E)_k$ where,
\begin{equation}\label{grrea}
\LL(E)_k=  \Big \{ \sum_i r_i \alpha_i \beta_i^*\mid \alpha_i,\beta_i \textrm{ are paths}, r_i \in K, \textrm{ and } |\alpha_i|-|\beta_i|=k \textrm{ for all } i \Big\}.
\end{equation}
For simplicity we denote $\LL(E)_k$, the homogeneous elements of degree $k$, by $\LL_k$.

We define an (anti-graded) involution 
on $\LL(E)$ by $\overline {\mu\gamma^*}=\gamma\mu^*$ for the monomials and extend it to the whole $\LL(E)$ in the obvious manner. Note that if $x\in \LL(E)_n$, then $\overline x \in \LL(E)_{-n}$.

By constructing a representation of $\LL_K(E)$ in $\End_K(V)$, for a suitable $K$-vector space $V$, one can show that the vertices of the  graph $E$ are linearly independent in $\LL_K(E)$ and the edges and ghost edges are not zero (see Lemma~1.5 in \cite{goodearl}).

Throughout the note we need some more definitions which we gather here.

\begin{deff}\hfill  \label{mulidef}
\begin{enumerate}

\item A path  which does not contain a cycle is called a {\it acyclic} path. 

\item A graph without cycles is called a {\it acyclic} graph.

\item Let $v \in E^0$. Then the {\it out-degree} and the {\it total-degree} of $v$ are defined as $\outdeg(v)=\card(s^{-1}(v))$ and 
$\totdeg(v)=\card(s^{-1}(v) \cup r^{-1}(v))$, respectively.

\item A finite graph $E$ is called a {\it line graph} if it is connected, acyclic and $\totdeg(v)\leq 2$ for every 
$v \in E^0$.  If we want to emphasize the 
number of vertices, we say that $E$ is an $n$-line graph whenever $n=\card(E^0)$. An {\it oriented} $n$-line graph $E$ is an $n$-line graph such that  $E^{n-1} \not =\emptyset$.  

\item For  any vertex $v \in E^0$, the cardinality of the set $R(v)=\{\alpha \in E^* \mid r(\alpha)=v\}$ is denoted by $n(v)$. 

\item For any vertex $v \in E^0$, the {\it tree} of $v$, denoted by $T(v)$,  is the set $\{w\in E^0 \mid v\geq w \}$. Furthermore, for $X\subseteq E^0$, we define $T(X)=\bigcup_{x \in X}T(x)$. 

\item A subset $H$ of $E^0$ is called {\it hereditary} if $v\geq w$ and $v \in H$ imply that $w \in H$.

\item A hereditary set $H$ is {\it saturated}  if $s^{-1}(v) \not = \emptyset$ and $r(s^{-1}(v)) \subseteq  H$, then $v \in H$.
\end{enumerate}
\end{deff}

\begin{deff}\label{petaldef}
A {\it rose with $k$-petals} is a graph which consists of one vertex and $k$ loops. We denote this graph by $L_k$ and its vertex by $s(L_k)$. The Leavitt path algebra of this graph with coefficient in $K$ is denoted by $\LL_K(1,k)$. We allow $k$ to be zero and in this case $L_0$ is just a vertex with no loops. With this convention, one can easily establish that $\LL_K(1,0)\cong K$ and $\LL_K(1,1)\cong K[x,x^{-1}]$.  In a graph which contains a rose $L_k$,  we say $L_k$ does not have an exit, if there is no edge $e$ with $s(e)=s(L_k)$ and $r(e) \not = s(L_k)$. 
\end{deff} 

 We need to recall the definition of morphisms between two graphs in order to consider the category of directed graphs.

For two directed graphs $E$ and $F$, a {\it complete  graph homomorphism} $f:E\rightarrow F$ consists of a map $f^0:E^0 \rightarrow F^0$ and $f^1:E^1 \rightarrow F^1$ such that $r(f^1(\alpha))=f^0(r(\alpha))$ and $s(f^1(\alpha))=f^0(s(\alpha))$ 
for any $\alpha \in E^1$, additionally, $f^0$ is injective and $f^1$ restricts to a bijection from $s^{-1}(v)$ to $s^{-1}(f^0(v))$ for every $v\in E^0$ which emits edges. One can check that such a map induces a graded homomorphism on the level of Leavitt path algebras. i.e, there is a graded homomorphism $\LL(E) \rightarrow \LL(F)$.

\subsection{Polycephaly graphs}\label{poiyre}
Two distinguished types of strongly graded Leavitt path algebras are $C_n$-comet graphs and multi-headed rose graphs (see Figure~\ref{monster2}). We consider  a polycephaly graph (Definition~\ref{popyt}) which is a mixture of these graphs (and so include all these types of graphs). The   graded structure of Leavitt path algebras of polycephaly graphs were determined in \cite[Theorem~4.7]{haz}. In this note we prove Conjectures~\ref{weakconj} and ~\ref{cofian} for the category of polycephaly graphs.

\begin{deff}\label{cometi}
A finite graph $E$ is called a {\it $C_n$-comet}, if $E$ has exactly one cycle $C$ (of length $n$), and $C^0$ is dense, i.e., 
$T(v) \cap (C)^0 \not = \emptyset$ for any vertex $v \in E^0$. 
\end{deff} 

Note that the uniqueness of the cycle $C$ in the definition of $C_n$-comet together with its density implies that the
cycle has no exits.

\begin{deff}\label{wert}
A finite graph $E$ is called a {\it multi-headed comets} if $E$ consists of $C_{l_s}$-comets, $1\leq s\leq t$, of length $l_s$, such that cycles are mutually disjoint and any vertex is connected to at least a cycle. (Recall that the graphs in this paper are all connected).  More formally, $E$ consists of $C_{l_s}$-comets, $1\leq s\leq t$,  and for any vertex $v$ in $E$, there is at least a cycle, say,  $C_{l_k}$, such that $T(v) \cap C_{l_k}^0 \not = \emptyset$, and furthermore no cycle has an exit. 
\end{deff}

\begin{deff} \label{popyt}
A finite graph $E$ is called a {\it polycephaly} graph if $E$ consists of a multi-headed comets and/or an acyclic graph with its sinks attached to roses such that all the cycles and roses are mutually disjoint and any vertex is connected to at least a cycle or a rose. More formally, $E$ consists of $C_{l_s}$-comets, $1\leq s\leq h$, and a finite acyclic graph with sinks $\{v_{h+1},\dots,v_t\}$ together with  $n_s$-petal graphs $L_{n_s}$ attached to $v_s$, where  $n_s \in \mathbb N$ and $h+1\leq s \leq t$.  Furthermore any vertex $v$ in $E$ is connected to at least one of $v_s$, $h+1\leq s \leq t$, or  at least a cycle, i.e., there is  $C_{l_k}$, such that $T(v) \cap C_{l_k}^0 \not = \emptyset$, and  no cycle or a rose has an exit (see Definition~\ref{petaldef}). When $h=0$, $E$ does not have any cycle, and when $t=h$, $E$ does not have any roses.  
\end{deff}

\begin{remark}
Note that a cycle of length one can also be considered as a rose with one petal. This should not cause any confusion in the examples and theorems below. In some proofs (for example proof of Theorem~\ref{mani543}), we collect all the cycles of length one in the graph as comet types and thus all the roses have either zero or more than one petals, i.e.,  $n_s=0$ or $n_s>1$, for any $h+1\leq s \leq t$. 
\end{remark}

\begin{example}\label{popyttr}
 Let $E$ be a polycephaly graph. 
\begin{enumerate}
\item If $E$ contains no cycles, and for any rose $L_{n_s}$ in $E$, $n_s=0$, then $E$ is a finite acyclic graph. 

\item If $E$ consists of only one cycle (of length $n$) and no roses, then $E$ is a $C_n$-comet. 

\item If $E$ contains no roses, then $E$ is a multi-headed comets. 

\item If $E$ contains no cycles, and for any rose $L_{n_s}$ in $E$, $n_s \geq 1$, then $E$ is a multi-headed rose graph (see Figure~\ref{monster2}). 
\end{enumerate}
\end{example}

\begin{example}
The following graph is a (three-headed rose) polycephaly graph.

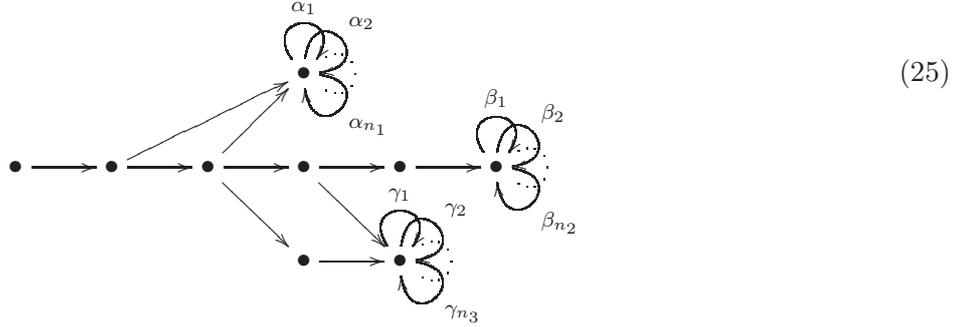
\begin{equation}\label{monster2}
\xymatrix{
 & & &&   \bullet  \ar@(ul,ur)^{\alpha_{1}}  \ar@(u,r)^{\alpha_{2}} \ar@{.}@(ur,dr) \ar@(r,d)^{\alpha_{n_1}}& \\
   &   \bullet \ar[r]  &  \bullet \ar[r] \ar[urr]   &  \bullet \ar[r] \ar[dr]  \ar[ur] & \bullet \ar[r]  \ar[dr] &    \bullet \ar[r]  & \bullet  \ar@(ul,ur)^{\beta_{1}}  \ar@(u,r)^{\beta_{2}} \ar@{.}@(ur,dr) \ar@(r,d)^{\beta_{n_2}}& \\
 &  &   &    & \bullet \ar[r]  &   \bullet  \ar@(ul,ur)^{\gamma_{1}} \ar@(u,r)^{\gamma_{2}} \ar@{.}@(ur,dr) \ar@(r,d)^{\gamma_{n_3}}&  
 }
 \end{equation}
\end{example}

\begin{example}
The following graph is a (five-headed) polycephaly  graph, with two roses, two sinks and a comet.  
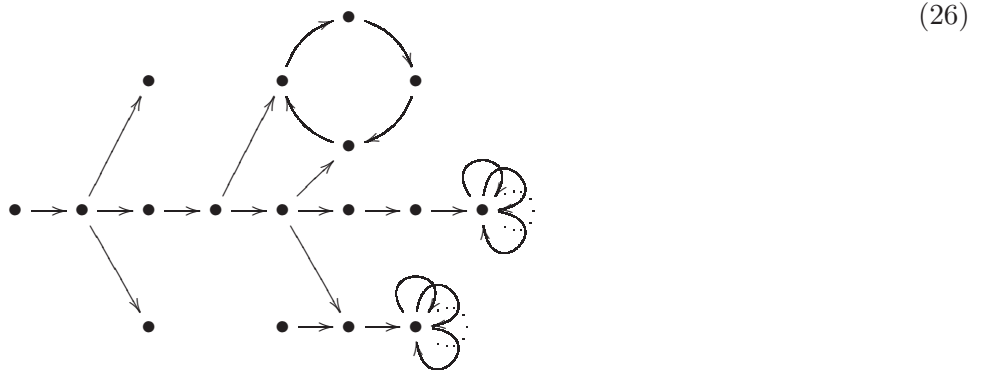
\begin{equation}\label{monster}
\xymatrix@=13pt{
& & & & & \bullet \ar@/^/[dr]      \\
& &  \bullet  & & \bullet \ar@/^/[ur] &&\bullet \ar@/^/[dl] &  \\
& &  &  & & \bullet \ar@/^/[ul]      \\
 \bullet \ar[r]  &   \bullet \ar[r]  \ar[uur] \ar[ddr] &   \bullet \ar[r]  &  \bullet \ar[r]   \ar[ruu] &  \bullet \ar[r] \ar[ddr] \ar[ru] & \bullet \ar[r]  &    \bullet \ar[r]  & \bullet  \ar@(ul,ur)  \ar@(u,r) \ar@{.}@(ur,dr) \ar@(r,d)& \\
  \\
& & \bullet  &   &  \bullet \ar[r]  & \bullet \ar[r]  &   \bullet  \ar@(ul,ur) \ar@(u,r) \ar@{.}@(ur,dr) \ar@(r,d)&  
 }
 \end{equation}
\end{example}

\medskip 

The following theorem classifies Leavitt path algebras of polycephaly graphs. This is Theorem~4.7 of ~\cite{haz}. This will be used in \S\ref{daryak} to calculate graded $K_0$ of polycephaly graphs (including acyclic, comets and multi-headed rose graphs)  and prove Conjecture~\ref{weakconj} for the category of polycephaly graphs (see Theorem~\ref{mani543}).  

\begin{theorem}\label{polyheadb}
Let $K$ be a  field and    $E$ be a polycephaly graph consisting of cycles $C_{l_s}$, $1\leq s\leq h$, of length $l_s$
and an acyclic graph  with sinks $\{v_{h+1},\dots,v_t\}$ which are attached to $L_{n_{h+1}},\dots,L_{n_t}$, respectively.
For any $1\leq s\leq h$ choose $v_s$ (an arbitrary vertex) in $C_{l_s}$ and remove the edge $c_s$ with $s(c_s)=v_s$ from the cycle $C_{l_s}$.  Furthermore, for any $v_s$, $h+1\leq s \leq t$, remove the rose $L_{n_s}$ from the graph. 
In this new acyclic graph $E_1$, let $\{p^{v_s}_i \mid 1\leq i \leq n(v_s)\}$ be the set of all paths which end in $v_s$, $1\leq s \leq t$.
 Then there is a $\mathbb Z$-graded isomorphism
\begin{equation}\label{dampaib}
\LL_K(E) \cong_{\gr}  \bigoplus_{s=1}^h  \M_{n(v_s)}\big(K[x^{l_s},x^{-l_s}]\big)(\overline{p_s}) \bigoplus_{s=h+1}^t \M_{n(v_s)} \big(\LL_K(1,n_s)\big) (\overline{p_s}),
\end{equation}
where $\overline{p_s}=\big(|p_1^{v_s}|,\dots, |p_{n(v_{s})}^{v_s}|\big)$.
\end{theorem}

\begin{remark}\label{kjsdyb}
Theorem~\ref{polyheadb} shows that the Leavitt path algebra of a polycephaly graph decomposes into direct sum of three types of algebras. Namely, matrices over the field $K$ (which corresponds to acyclic parts of the graph, i.e, $n_s=0$ in~(\ref{dampaib})), matrices over $K[x^l,x^{-l}]$, $l \in \mathbb N$, (which corresponds to comet parts of the graph) and matrices over Leavitt algebras $\LL(1,n_s)$, $n_s \in \mathbb N$ and $n_s \geq 1$ (which corresponds to rose parts of the graph). Also note that a cycle of length one can also be considered as a rose with one petal, which in either case, on the level of Leavitt path algebras, we obtain matrices over the algebra $K[x,x^{-1}]$. 
\end{remark}

\begin{example} \label{nopain}
Consider the polycephaly graph $E$

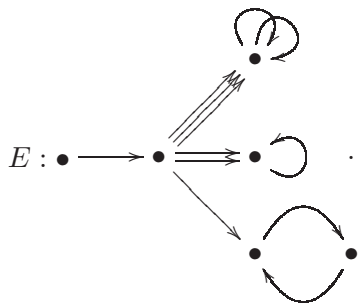
\[
 \xymatrix{
& & \bullet  \ar@(ul,ur)  \ar@(u,r) \\
E: \bullet \ar[r]  & \bullet   \ar@<1.5pt>[r]  \ar@<-1.5pt>[r] \ar@<0.5ex>[ur] \ar@<-0.5ex>[ur] \ar@<0ex>[ur] \ar[dr] & \bullet \ar@(rd,ru) &. \\ 
& &   \bullet \ar@/^1.5pc/[r] & \bullet \ar@/^1.5pc/[l]&  \\
 }
\]
\medskip 
Then by Theorem~\ref{polyheadb}, 
\[\LL_K(E)\cong_{\gr} \M_5(K[x,x^{-1}])(0,1,1,2,2)  \oplus \M_4(K[x^2,x^{-2}])(0,1,1,2) \oplus \M_7(\LL(1,2))(0,1,1,1,2,2,2).\]
\end{example}

\subsection{$K_0$ of Leavitt path algebras}

For an abelian monoid $M$, we denote by $M^{+}$ the group completion of $M$. This gives a left adjoint functor to the forgetful functor from the category of abelian groups to abelian monoids. Starting from a ring $R$, the isomorphism classes of finitely generated (right) $R$-modules equipped with the direct sum gives an abelian monoid, denoted by $\mathcal V(R)$. The group completion of this monoid is denoted by $K_0(R)$ and called the Grothendieck group of $R$. For a Leavitt path algebra $\LL_K(E)$, the monoid $\mathcal V(\LL_K(E))$ is studied in~\cite{amp}. In particular using~\cite[Theorem~3.5]{amp}, one can calculate the Grothendieck group of a Leavitt path algebras from the adjacency matrix of a graph (see~\cite[p.1998]{aalp}). We include this calculation in this subsection for the completeness.

Let $F$ be a free abelian monoid generated by a countable set $X$. The nonzero elements of $F$ can be written as $\sum_{t=1}^n x_t$ where $x_t \in X$. Let $r_i, s_i$, $i\in I \subseteq \mathbb N$, be elements of $F$. We define an equivalence relation on $F$ denoted by $\langle r_i=s_i\mid i\in I \rangle$ as follows: Define a binary relation $\rightarrow$ on $F\backslash \{0\}$,  $r_i+\sum_{t=1}^n x_t \rightarrow s_i+\sum_{t=1}^n x_t$, $i\in I$ and generate the equivalence relation on $F$ using this binary relation. Namely, $a \sim a$ for any $a\in F$ and for $a,b \in F \backslash \{0\}$, $a \sim b$ if there is a sequence $a=a_0,a_1,\dots,a_n=b$ such that for each $t=0,\dots,n-1$ either $a_t \rightarrow a_{t+1}$ or $a_{t+1}\rightarrow a_t$. We denote the quotient monoid by
$F/\langle r_i=s_i\mid i\in I\rangle$.
Then one can see that there is a canonical group isomorphism
\begin{equation} \label{monio}
\Big (\frac{F}{\langle r_i=s_i\mid i\in I \rangle}\Big)^{+} \cong \frac{F^{+}}{\langle r_i-s_i\mid i\in I \rangle}.
\end{equation}

\begin{deff}\label{adji}
Let $E$ be a graph and  $N'$ be the adjacency  matrix
$(n_{ij}) \in \mathbb Z^{E^0\oplus E^0}$ where $n_{ij}$ is the
number of edges from $v_i$ to $v_j$. Furthermore let
$I'$ be the identity matrix. Let $N^t$ and $I$ be
the matrices obtained from $N'$ and $I'$ by first taking the transpose and then removing the columns
corresponding to sinks, respectively.
\end{deff}
Clearly the adjacency matrix depends on the ordering we put on
$E^0$. We usually fix
 an ordering on $E^0$ such that the elements of
$E^0 \backslash \sink$ appear first in the list follow with elements
of the $\sink$.

 Multiplying the matrix $N^t-I$ from the left defines a
homomorphism $\mathbb Z^{E^0\backslash \sink} \longrightarrow
\mathbb Z^{E_0}$, where  $\mathbb Z^{E^0\backslash \sink} $ and
$\mathbb Z^{E^0}$ are the direct sum of copies of $\mathbb Z$
indexed by $E^0\backslash \sink$ and $E^0$, respectively. The next
theorem shows that the cokernel of this map gives the Grothendieck
group of Leavitt path algebras. 

\begin{theorem}\label{wke}
Let $E$ be a graph and $\LL(E)$ a Leavitt path algebra. Then
\begin{equation}
K_0(\LL(E))\cong \coker\big(N^t-I:\mathbb Z^{E^0\backslash \sink} \longrightarrow \mathbb Z^{E^0}\big).
\end{equation}
\end{theorem}
\begin{proof}
Let $M_E$ be the abelian monoid generated by $\{v \mid v \in E^0 \}$ subject to the relations
\begin{equation}\label{phgqcu}
v=\sum_{\{\alpha\in E^{1} \mid s(\alpha)=v \}} r(\alpha),
\end{equation}
for every $v\in E^0\backslash \sink$, where $n_v=
\max\{w(\alpha)\mid \alpha\in E^{\st}, \, s( \alpha)=v\}$. Arrange
$E^0$ in a fixed order such that the elements of $E^0\backslash
\sink$ appear first in the list follow with elements of $\sink$. The
relations ~(\ref{phgqcu}) can be then written as $N^t \overline v_i= I
\overline v_i$, where  $v_i \in E^0\backslash \sink $ and
$\overline v_i$ is the $(0,\dots,1,0,\dots)$ with $1$ in the $i$-th
component.   Therefore,
\[M_E\cong \frac{F}{ \langle N^t \overline v_i = I \overline v_i , v_i \in E^0 \backslash \sink \rangle},\] where $F$ is the free abelian monoid generated by the vertices of $E$.
By~\cite[Theorem~3.5]{amp} there is a natural monoid isomorphism $\V(\LL_K(E))\cong M_E$. So using~(\ref{monio}) we have,
\begin{equation}\label{pajd}
K_0(\LL(E))\cong \V(\LL_K(E))^{+}\cong M_E^{+}\cong \frac{F^+}{ \langle (N^t-I) \overline v_i, v_i \in E^0 \backslash \sink \rangle}.
\end{equation}
Now $F^{+}$ is $\mathbb Z^{E^0}$ and it is easy to see that the
denominator in ~(\ref{pajd}) is the image of $N^t-I:\mathbb
Z^{E^0\backslash \sink} \longrightarrow \mathbb Z^{E^0}$.
\end{proof}

\section{Graded $K_0$ of Leavitt path algebras}\label{daryak}

For a strongly graded ring $R$, Dade's Theorem implies that
$K_0^{\gr}(R) \cong K_0(R_0)$ (\S\ref{dadestmal}). In~\cite{haz}, we determined the
strongly graded Leavitt path algebras. In particular we proved

\begin{theorem}[\cite{haz}, Theorem~3.15]\label{hazst}
Let $E$ be a finite graph. The Leavitt path algebra $\LL(E)$ is
strongly graded if and only if $E$ does not have a sink.
\end{theorem}

This theorem shows that many interesting classes of Leavitt path
algebras fall into the category of strongly graded LPAs, such as
purely infinite simple Leavitt path algebras (see~\cite{aap06}).

The following theorem shows that there is a close relation between
graded $K_0$ and the non-graded $K_0$ of strongly graded Leavitt path algebras.

\begin{theorem}\label{ultrad}
Let $E$ be a row-finite graph. Then there is an exact sequence
\begin{equation}\label{zerhom}
K_0(\LL(E)_0) \stackrel{N^t-I}{\longrightarrow}
K_0(\LL(E)_0) \longrightarrow K_0(\LL(E)) \longrightarrow 0.
\end{equation}

Furthermore if $E$ is a finite graph with no sinks, the exact sequence takes the form
\begin{equation}\label{kaler}
K_0^{\gr}(\LL(E)) \stackrel{N^t-I}{\longrightarrow}
K_0^{\gr}(\LL(E)) \longrightarrow K_0(\LL(E)) \longrightarrow 0.
\end{equation} In this case,
\begin{equation}\label{imip}
K_0^{\gr}(\LL(E)) \cong \varinjlim \mathbb Z^{E^0}
\end{equation}
of the inductive system $\mathbb Z^{E^0}
\stackrel{N^t}{\longrightarrow} \mathbb Z^{E^0}
\stackrel{N^t}\longrightarrow \mathbb Z^{E^0}
\stackrel{N^t}\longrightarrow \cdots$.
\end{theorem}
\begin{proof}
First let $E$ be a finite graph. Write $E^0=V\cup W$ where $W$ is the set of sinks and $V=E^0 \backslash W$. The ring $\LL(E)_0$ can be described as follows (see the proof of
Theorem~5.3 in \cite{amp}):
$\LL(E)_0=\bigcup_{n=0}^{\infty}L_{0,n}$, where $L_{0,n}$ is the
linear span of all elements of the form $pq^*$ with $r(p)=r(q)$ and
$|p|=|q|\leq n$. The transition inclusion $L_{0,n}\rightarrow L_{0,n+1}$ is
the identity on all elements $pq^*$ where $r(p) \in W$ and  is to
identify $pq^*$ with $r(p) \in V$ by
\[\sum_{\{ \alpha | s(\alpha)=v\}} p\alpha (q\alpha)^*.\] Note that since $V$ is a set of vertices which are not
sinks, for any $v\in V$ the set $\{ \alpha | s(\alpha)=v\}$ is not
empty.

For a fixed $v \in E^0$, let
$L_{0,n}^v$ be the linear span of all elements of
the form $pq^*$ with $|p|=|q|=n$ and $r(p)=r(q)=v$. Arrange the paths
of length $n$ with the range $v$ in a fixed order
$p_1^v,p_2^v,\dots,p_{k^v_n}^v$, and observe that the correspondence
of  $p_i^v{p_j^v}^*$ to the matrix unit $e_{ij}$ gives rise to a ring
isomorphism $L_{0,n}^v\cong\M_{k^v_n}(K)$. Furthermore, $L_{0,n}^v$, $v\in E^0$ and
$L_{0,m}^w$, $m<n$, $w\in W$ form a direct sum.  This implies that
\[L_{0,n}\cong
\Big ( \bigoplus_{v\in V}\M_{k_n^v}(K) \times \bigoplus_{v\in
W}\M_{k_n^v}(K) \Big) \times \Big( \bigoplus_{m=0}^{n-1} \big ( \bigoplus_{v
\in W} \M_{k_m^v}(K) \big )\Big),\] where $k_m^v$, $v \in E^0$, 
$0\leq m\leq n$, is the number of paths of length $m$ with the range
$v$. The inclusion map $L_{0,n}\rightarrow L_{0,n+1}$ is then identity on
factors where $v \in W$ and is
\begin{equation}\label{volleyb}
N^t: \bigoplus_{v\in V} \M_{k^v_n}(K) \longrightarrow \bigoplus_{v\in V}
\M_{k^v_{n+1}}(K) \times \bigoplus_{v\in W} \M_{k^v_{n+1}}(K),
\end{equation}
where $N$ is the adjacency matrix of the graph $E$ (see Definition~\ref{adji}).  This
means $(A_1,\dots,A_l)\in \bigoplus_{v\in V} \M_{k^v_n}(K) $ is sent to
$(\sum_{j=1}^l n_{j1}A_j,\dots,\sum_{j=1}^l n_{jl}A_j) \in
\bigoplus_{v\in E^0} \M_{k^v_{n+1}}(K)$, where $n_{ji}$ is the number of
edges connecting $v_j$ to $v_i$ and
\[\sum_{j=1}^lk_jA_j=\left(
\begin{matrix}
A_1 &       &             & \\
       & A_1 &            & \\
       &         & \ddots &\\
       &         &            & A_l\\
\end{matrix}
\right)
\]
in which each matrix is repeated $k_j$ times down the leading
diagonal and if $k_j=0$, then $A_j$ is omitted.
 This can be seen as follows: for $v\in E^0$, arrange the
the set of paths of length $n+1$ with the range $v$ as
\begin{align}
\big \{ & p_1^{v_1}\alpha_1^{v_1v}, \dots,
p_{k_n^{v_1}}^{v_1}\alpha_1^{v_1v}, p_1^{v_1}\alpha_2^{v_1v}, \dots,
p_{k_n^{v_1}}^{v_1}\alpha_2^{v_1v},\dots,
p_1^{v_1}\alpha_{n_{1v}}^{v_1v}, \dots,
p_{k_n^{v_1}}^{v_1}\alpha_{n_{1v}}^{v_1v},\\ \notag &
p_1^{v_2}\alpha_1^{v_2v}, \dots, p_{k_n^{v_2}}^{v_2}\alpha_1^{v_2v},
p_1^{v_2}\alpha_2^{v_2v}, \dots,
p_{k_n^{v_2}}^{v_2}\alpha_2^{v_2v},\dots,
p_1^{v_2}\alpha_{n_{2v}}^{v_2v}, \dots,
p_{k_n^{v_2}}^{v_2}\alpha_{n_{2v}}^{v_2v},\\\notag & \dots \big \},
\end{align}
where $\{p_1^{v_i},\dots,p_{k_n^{v_i}}^{v_i}\}$ are paths of length
$n$ with the range $v_i$ and
$\{\alpha_1^{v_iv},\dots,\alpha_{n_{iv}}^{v_iv}\}$ are all the edges
with the source $v_i\in V$ and range $v$. Now this shows that if
$p_s^{v_i}(p_t^{v_i})^*\in L_{0,n}^{v_i}$ (which corresponds to
the matrix unit $e_{st}$), then $\sum_{j=1}^{n_{iv}} p_s^{v_i}\alpha_{j}^{v_iv}
(p_t^{v_1}\alpha_{j}^{v_iv})^*\in L_{0,n+1}^v$ corresponds to the matrix with 
matrix unit
$e_{st}$ repeated down the diagonal $n_{iv}$ times.

Writing $\LL(E)_0=\varinjlim_{n} L_{0,n}$, since the Grothendieck group $K_0$ respects the
direct limit, we have $K_0(\LL(E)_0)\cong
\varinjlim_{n}K_0(L_{0,n})$. Since $K_0$ of  (Artinian) simple
algebras are $\mathbb Z$, the ring homomorphism $L_{0,n}\rightarrow
L_{0,n+1}$ induces the group homomorphism \[\mathbb Z^V \times \mathbb Z^W \times
\bigoplus_{m=0}^{n-1} \mathbb Z^W\stackrel{N^t\times I}{\longrightarrow}
\mathbb Z^V \times \mathbb Z^W \times \bigoplus_{m=0}^{n}\mathbb Z^W,\]
where $N^t:\mathbb Z^V \rightarrow \mathbb Z^V \times \mathbb Z^W$ is multiplication of the matrix
$N^t=\left(\begin{array}{c}
B^t  \\
C^t
\end{array}\right)
$ from left which is induced by the homomorphism~(\ref{volleyb})
and $I_n: \mathbb Z^W \times \bigoplus_{m=0}^{n-1} \mathbb Z^W \rightarrow
\bigoplus_{m=0}^{n}\mathbb Z^W$ is the identity map. Consider the commutative diagram
\[
\xymatrix{
K_0(L_{0,n}) \ar[r]  \ar[d]^{\cong} & K_0(L_{0,n+1}) \ar[d]^{\cong}\ar[r] & \cdots \\
\mathbb Z^V \times \mathbb Z^W \times
\bigoplus_{m=0}^{n-1} \mathbb Z^W  \ar[r]^{N^t\oplus I_{n}} \ar@{^{(}->}[d] &
\mathbb Z^V \times \mathbb Z^W \times \bigoplus_{m=0}^{n}\mathbb Z^W \ar@{^{(}->}[d] \ar[r]  & \cdots \\
\mathbb Z^V \times \mathbb Z^W \times \mathbb Z^W \times \cdots  \ar[r]^{N^t \oplus I_{\infty}}& \mathbb Z^V \times \mathbb Z^W \times \mathbb Z^W \times \cdots  \ar[r]  & \cdots
}
\]
where
\[N^t\oplus I_{\infty}=
\left(\begin{array}{ccccc}
B^t & 0 & 0 & 0 & \cdot\\
C^t & 0 & 0 & 0 & \cdot \\
0    & 1 & 0 & 0 &  \cdot\\
0    & 0 & 1 & 0 &   \cdot \\
\cdot    & \cdot & \cdot  &  \cdot &   \ddots
\end{array}\right).\]
It is easy to see that the direct limit of the second row in the diagram is isomorphic to the direct limit of the third row.
Thus $K(\LL(E)_0) \cong \varinjlim \big (\mathbb Z^V \times \mathbb Z^W \times \mathbb Z^W \times \cdots \big)$.


Consider the commutative diagram on the left below
which induces a natural map, denoted by $N^t$ again, on the direct
limits.
\[
\xymatrix{
\mathbb Z^{V}  \times \mathbb Z^W \times \cdots  \ar[r]^{N^t} \ar[d]^{N^t} & \mathbb Z^{V}  \times \mathbb Z^W \times \cdots   \ar[r]^{\quad \qquad N^t}  \ar[d]^{N^t} &  \cdots &  & \varinjlim \cong K_0(\LL(E)_0)  \ar@{.>}[d]^{N^t}\\
\mathbb Z^{V} \times \mathbb Z^W \times \cdots\ar[r]^{N^t}   & \mathbb Z^{V} \times \mathbb Z^W \times \cdots\ar[r]^{\quad \qquad N^t}   &  \cdots&  &\varinjlim \cong K_0(\LL(E)_0)
}
\]
Now Lemma 7.15 in~\cite{Raegraph} implies that
\begin{equation*}
\coker\Big
(K_0(\LL(E)_0)\stackrel{N^t-1}{\longrightarrow} K_0(\LL(E)_0\Big )
\cong \coker\Big(\mathbb Z^{V}  \times \mathbb Z^W \times \cdots \stackrel{N^t-1}{\longrightarrow}
\mathbb Z^{V}  \times \mathbb Z^W \times \cdots \Big).
\end{equation*}
And Lemma~3.4 in~\cite{Raeburn113} implies that the
\begin{equation}\label{thlap}
\coker\Big(\mathbb Z^{V}  \times \mathbb Z^W \times \cdots \stackrel{N^t-1}{\longrightarrow}
\mathbb Z^{V}  \times \mathbb Z^W \times \cdots \Big)\cong \coker\Big(\mathbb Z^V \stackrel{N^t-1}{\longrightarrow} \mathbb Z^V \times \mathbb Z^W\Big).
\end{equation}
Finally by Theorem~\ref{wke}, replacing
$\coker\big (\mathbb Z^V \stackrel{N^t-1}{\longrightarrow} \mathbb Z^V \times \mathbb Z^W\big)$ by $K_0(\LL(E))$ in ~(\ref{thlap}) we
obtain the exact sequence
\begin{equation}\label{zerhom11}
K_0(\LL(E)_0) \stackrel{N^t-I}{\longrightarrow}
K_0(\LL(E)_0) \longrightarrow K_0(\LL(E)) \longrightarrow 0
\end{equation}
when the graph $E$ is finite. Now a row-finite graph can be written as a direct limit of finite complete subgraphs~\cite[Lemmas 3.1, 3.2]{amp}, i.e., $E=\varinjlim E_i$, $\LL(E)\cong \varinjlim \LL(E_i)$ and  $\LL(E)_0\cong \varinjlim \LL(E_i)_0$. Since the direct limit is an exact functor, applying $\varinjlim$ to the exact sequence~(\ref{zerhom11}) for $E_i$, we get the first part of the theorem.

When the graph $E$ is finite with no sinks, then $\LL(E)$ is strongly graded (Theorem~\ref{hazst}). Therefore
\[ K_0^{\gr}(\LL(E)) \cong K_0(\LL(E)_0).\] This gives the exact sequence~(\ref{kaler}).
\end{proof}

\begin{example} Let $E$ be the following graph. 

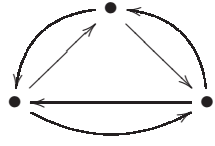
\begin{equation*}
\xymatrix{
& \bullet  \ar[dr] \ar@/_1pc/[dl] & \\
\bullet \ar[ur] \ar@/_1pc/[rr]   & & \bullet \ar[ll] \ar@/_1pc/[ul] 
 }
\end{equation*}
\medskip 

By Theorem~\ref{hazst}, $\LL(E)$ is strongly graded and so by Thereom~\ref{ultrad}(\ref{imip}), 
\[K_0^{\gr}(\LL(E)) \cong \varinjlim_{N^t} \mathbb Z\oplus \mathbb Z \oplus \mathbb Z,\] where
$N^t=
\left(\begin{array}{ccc}
0 & 1 & 1 \\
1 & 0 & 1  \\
1 & 1 & 0 \\
\end{array}\right). $ 
Since $\det(N)=2 \not = 0$, 
one can see that \[\varinjlim_{N^t} \mathbb Z\oplus \mathbb Z \oplus \mathbb Z   \cong \mathbb Z\Big [\frac{1}{\det(N)}\Big]  \oplus  \mathbb Z\Big [\frac{1}{\det(N)}\Big] \oplus  \mathbb Z\Big [\frac{1}{\det(N)}\Big] \cong \bigoplus_3\mathbb Z[1/2].\]
On the other hand by Theorem~\ref{wke}, $K_0(\LL(E)) \cong \mathbb Z /2\mathbb Z \oplus \mathbb Z /2\mathbb Z$ (see~\cite[Example~3.8]{aalp} for the detailed computation). By Thereom~\ref{ultrad}(\ref{kaler}), these groups fit into the following exact sequence.  
\[ \bigoplus_3\mathbb Z[1/2] \stackrel{N^t-I}{\longrightarrow} \bigoplus_3\mathbb Z[1/2] \longrightarrow \mathbb Z /2\mathbb Z \oplus \mathbb Z /2\mathbb Z \longrightarrow 0.\]

\end{example}

\begin{remark}\label{hjsonf}
Let $E$ be a row finite graph. The $K_0$ of graph $C^*$-algebra of $E$ was first computed in~\cite{Raeburn113}, where Raeburn and Szyma\'nski
obtained the top exact sequence of the following commutative
diagram. Here $E\times_1 \mathbb Z$ is a graph with $(E\times_1
\mathbb Z)^0=E^0\times \mathbb Z$, $(E\times_1 \mathbb
Z)^1=E^1\times \mathbb Z$ and $s(e,n)=(s(e),n-1))$ and
$r(e,n)=(r(e),n))$. Looking at the proof of \cite[Theorem~3.2]{Raeburn113}, and Theorem~\ref{ultrad}, one can see that,
$K_0\big(C^*(E\times_1 \mathbb Z)\big) \cong K_0(\LL_{\mathbb C}(E)_0)$  and
the following diagram commutes.
\begin{equation}\label{lol1532}
\xymatrix{ 0 \ar[r] & K_1(C^*(E)) \ar[r] \ar@{.>}[d] &
K_0\big(C^*(E\times_1 \mathbb Z)\big) \ar[d]^{\cong}
\ar[r]^{1-\beta^{-1}_*} & K_0\big(C^*(E\times_1 \mathbb Z)\big) \ar[d]^{\cong} \ar[r]& K_0(C^*(E))\ar@{.>}[d]\ar[r]& 0\\
0 \ar[r] & \ker(1-N^t) \ar[r] & K_0(\LL_{\mathbb C}(E)_0) \ar[r]^{1-N^{t}} &
K_0(\LL_{\mathbb C}(E)_0) \ar[r]& K_0(\LL_{\mathbb C}(E)) \ar[r] &0.}
\end{equation}

This immediately shows that for a row-finite graph $E$, $K_0(C^*(E)) \cong K_0(\LL_{\mathbb C}(E))$. This
was also proved in \cite[Theorem~7.1]{amp}.
In particular if $E$ is a finite graph with no sinks, then
\[K_0\big(C^*(E\times_1 \mathbb Z)\big) \cong K_0^{gr}(\LL_{\mathbb C}(E)).\]
\end{remark}


The graded structure of Leavitt path algebras of polycephaly graphs (which include acyclic and comet graphs) were classified in~\cite{haz} (see Theorem~\ref{polyheadb}). This, coupled by the graded Morita theory~(cf. Proposition~\ref{grmorita}) and~(\ref{dade})  enable us to calculate graded $K_0$ of these graphs. We record them here.

\begin{theorem}[$K_0^{\gr}$ of acyclic graphs]\label{f4j5h6h8}
Let $E$ be a finite acyclic graph with sinks $\{v_1,\dots,v_t\}$. For any sink $v_s$, let $\{p^{v_s}_i \mid 1\leq i \leq n(v_s)\}$ be the set of all paths which end in $v_s$.  Then there is a $\mathbb Z[x,x^{-1}]$-module isomorphism 
\begin{equation}\label{skyone231}
\Big (K_0^{\gr}(\LL(E)),[\LL(E)] \Big ) \cong \Big ( \bigoplus_{s=1}^t \mathbb Z[x,x^{-1}], (d_1,\dots,d_t)\Big ),
\end{equation}
where $d_s=\sum_{i=1}^{n(v_s)}x^{-|p_i^{v_s}|}$.
\end{theorem}
\begin{proof}
By~\cite[Theorem~4.11]{haz} (see also Theorem~\ref{polyheadb}, in the absence of comet and rose graphs),
\begin{equation}\label{skyone23}
\LL(E) \cong_{\gr} \bigoplus_{s=1}^t \M_{n(v_s)} (K)\big (|p^{v_s}_1|,\dots, |p^{v_s}_{n(v_s)}|\big).
\end{equation}
Since $K_0^{\gr}$ is an exact functor, $K_0^{\gr}(\LL(E))$ is
isomorphic to the direct sum of $K_0^{\gr}$ of matrix algebras of
~(\ref{skyone23}).  The graded Morita equivalence
(Proposition~\ref{grmorita}), induces a $\mathbb Z[x,x^{-1}]$-isomorphism 
\[K_0^{\gr}\big
(\M_{n(v_s)} (K)\big (|p^{v_s}_1|,\dots, |p^{v_s}_{n(v_s)}|\big)
\big)\longrightarrow K_0^{\gr}(K),\] with \[\big [\M_{n(v_s)} (K)\big
(|p^{v_s}_1|,\dots, |p^{v_s}_{n(v_s)}|\big)\big ] \mapsto
[K(-|p^{v_s}_1|)\oplus \dots,\oplus K( -|p^{v_s}_{n(v_s)}|)].\] Now the
theorem follows from Proposition~\ref{k0grof}, by considering $A=K$ as a graded field concentrated in degree zero, i.e., $\Gamma_A=0$, and $\Gamma=\mathbb Z$ (see Example~\ref{upst}). 
\end{proof}


Leavitt path algebras of $C_n$-comets were classified in
\cite[Theorem~4.17]{haz} as follows: Let $E$ be a
$C_n$-comet with the cycle $C$ of length $n \geq 1$. Let $u$ be a
vertex on the cycle $C$. Eliminate the edge in the cycle whose
source is $u$ and consider the set $\{p_i \mid 1\leq i \leq m\}$  of
all paths which end in $u$. Then
\begin{equation}\label{ytsnf}
\LL(E) \cong_{\gr} \M_{m}\big(K[x^n,x^{-n}] \big)\big (|p_1|,\dots,
|p_m|\big).
\end{equation}

Let $d_l$, $0 \leq l \leq n-1$, be the number of $i$ such that
$|p_i|$ represents $\overline l$ in $\mathbb Z/n \mathbb Z$. Then
\begin{equation}\label{ytsnfpos}
\LL(E) \cong_{\gr} \M_{m}\big(K[x^n,x^{-n}] \big)\big (0,\dots,0,1,\dots,1,\dots,
n-1,\dots,n-1\big),
\end{equation}
where $0\leq l \leq n-1$ occurring $d_l$ times. It is now easy to see
\begin{equation}\label{zgdak}
\LL(E)_0 \cong \M_{d_0}(K)\times \dots \times \M_{d_{n-1}}(K).
\end{equation}

Furthermore, let $F$ be another $C_{n'}$-comet with the cycle $C'$
of length $n' \geq 1$ and $u'$ be a vertex on the cycle $C'$.
Eliminate the edge in the cycle whose source is $u'$ and consider
the set $\{q_i \mid 1\leq i \leq m'\}$  of all paths which end in
$u'$. Then $\LL(E) \cong_{\gr}\LL(F)$ if and only if $n=n'$,
$m=m'$ and for a suitable permutation  $\pi \in S_m$, and $r \in
\mathbb N$,  $r+|p_i|+n\mathbb Z=|q_{\pi(i)}|+n \mathbb Z$, where
$1\leq i \leq m$.

Before determining the $\mathbb Z[x,x^{-1}]$-module structure of the graded $K_0$ of a $C_n$-comet graph, we define a $\mathbb Z[x,x^{-1}]$-module structure on the group 
$\bigoplus_n \mathbb Z$.  Let $\phi: \mathbb Z[x,x^{-1}] \longrightarrow \M_n(\mathbb Z)=\End_{\mathbb Z}(\mathbb Z^n)$, be the evaluation ring homomorphism, where 
\[\phi(x)=\left(\begin{array}{ccccc}
0 & \dots & 0 & 0 & 1 \\
1 & 0 & \dots  & \dots & 0\\
0 & 1 & 0 & \dots & 0 \\
\vdots & 0  & \ddots
  & 0 & 0\\
0 & \dots & 0 & 1 & 0
\end{array}\right).\] This homomorphism induces a $\mathbb Z[x,x^{-1}]$-module structure on $\bigoplus_n \mathbb Z$, where 
$x (a_1,\dots,a_n)=(a_n,a_1,\dots, a_{n-1})$. 

\begin{theorem}[$K_0^{\gr}$ of $C_n$-comet graphs]\label{cothemp}
Let $E$  be a $C_n$-comet with the cycle $C$ of length $n \geq 1$.
Then there is a  $\mathbb Z[x,x^{-1}]$-module isomorphism
\begin{equation}\label{ytsnf1}
\Big ( K_0^{\gr} (\LL(E)),  [\LL(E)] \Big ) \cong \Big (
\bigoplus_{i=0}^{n-1} \mathbb Z, (d_0,\dots,d_{n-1}) \Big ),
\end{equation}
where $d_l$, $0 \leq l \leq n-1$, are as in
{\upshape~(\ref{ytsnfpos})}.
\end{theorem}
\begin{proof}
By~(\ref{ytsnf}), $K^{\gr}_0(\LL(E)) \cong_{\gr} K^{\gr}_0\big( \M_{m}\big(K[x^n,x^{-n}] \big)\big (|p_1|,\dots,
|p_m|\big)\big).$
The graded Morita equivalence
(Proposition~\ref{grmorita}), induces a $\mathbb Z[x,x^{-1}]$-isomorphism 
\[K_0^{\gr}\big
(\M_{m}\big(K[x^n,x^{-n}] \big)\big (|p_1|,\dots,
|p_m|\big)\big) \longrightarrow K_0^{\gr}(K[x^n,x^{-n}]),\] with 
\[\big [\M_{m}\big(K[x^n,x^{-n}] \big)\big (|p_1|,\dots,
|p_m|\big)\big ] \mapsto
\big [K[x^n,x^{-n}](-|p_1|)\oplus \dots,\oplus K[x^n,x^{-n}]( -|p_m|)\big].\] Now the
theorem follows from Proposition~\ref{k0grof}, by considering $A=K[x^n,x^{-n}]$ as a graded field with  $\Gamma_A=n\mathbb Z$ and $\Gamma=\mathbb Z$ (see Example~\ref{upst}). 
\end{proof}


Before determining the $\mathbb Z[x,x^{-1}]$-module structure of the graded $K_0$ of multi-headed rose graphs, we define a $\mathbb Z[x,x^{-1}]$-module structure on the group $\mathbb Z[1/n]$. Consider the evaluation ring homomorphism $\phi:\mathbb Z[x,x^{-1}] \rightarrow \mathbb Z[1/n]$, by $\phi(x)=n$. This gives a $\mathbb Z[x,x^{-1}]$-module structure on $\mathbb Z[1/n]$, where $xa=na$, $a \in \mathbb Z[1/n]$. 

\begin{theorem}[$K_0^{\gr}$ of multi-headed rose graphs]\label{re282}
Let   $E$ be a polycephaly graph consisting of an acyclic graph $E_1$ with sinks $\{v_1,\dots,v_t\}$ which are attached to $L_{n_1},\dots,L_{n_t}$, respectively, where $n_s\geq 1$, $1\leq s \leq t$. 
For any  $v_s$,  let $\{p^{v_s}_i \mid 1\leq i \leq n(v_s)\}$ be the set of all paths in $E_1$ which end in $v_s$.  Then there is a  $\mathbb Z[x,x^{-1}]$-module isomorphism
\begin{equation}\label{dampai}
\big(K_0^{\gr}(\LL(E)),[\LL(E)]\big)  \cong \big ( \bigoplus_{s=1}^t \mathbb Z[1/n_s], (d_1,\dots,d_t)\big),
\end{equation}
where $d_s=\sum_{i=1}^{n(v_s)}n_s^{-|p_i^{v_s}|}$.
\end{theorem}
\begin{proof} By Theorem~\ref{polyheadb} (in the absence of acyclic and comet graphs with cycles of length greater than one), 
\begin{equation}\label{dampai2}
\LL(E) \cong_{\gr} \bigoplus_{s=1}^t \M_{n(v_s)} \big(\LL(1,n_s)\big)\big (|p^{v_s}_1|,\dots, |p^{v_s}_{n(v_s)}|\big).
\end{equation}
Since $K_0^{\gr}$ is an exact functor, $K_0^{\gr}(\LL(E))$ is
isomorphic to the direct sum of $K_0^{\gr}$ of matrix algebras of
~(\ref{dampai2}).  The graded Morita equivalence
(Proposition~\ref{grmorita}), induces a $\mathbb Z[x,x^{-1}]$-isomorphism \[K_0^{\gr}\big
(\M_{n(v_s)} (\LL(1,n_s))\big (|p^{v_s}_1|,\dots, |p^{v_s}_{n(v_s)}|\big)
\big)\longrightarrow K_0^{\gr}(\LL(1,n_s)),\] with $\big [\M_{n(v_s)} (\LL(1,n_s))\big
(|p^{v_s}_1|,\dots, |p^{v_s}_{n(v_s)}|\big)\big ] \mapsto
[\LL(1,n_s)^{n(v_s)}(-|p^{v_s}_1|,\dots,- |p^{v_s}_{n(v_s)}|)]$. Now by Theorem~\ref{ultrad}(\ref{imip}),
$K_0^{\gr}(\LL(1,n_s))\cong \mathbb Z[1/n_s]$. Observe that $[\LL(1,n_s)]$ represents $1$ in  $K_0^{\gr}(\LL(1,n_s))$.
Since for any $i \in \mathbb Z$, 
\[\LL(1,n_s)(i)\cong \bigoplus_{n_s}\LL(1,n_s)(i-1),\]  it immediately follows that  $\big [\LL(1,n_s)^{n(v_s)}(-|p^{v_s}_1|,\dots,- |p^{v_s}_{n(v_s)}|)\big]$ maps to  $\sum_{i=1}^{n(v_s)}n_s^{-|p_i^{v_s}|}$ in $\mathbb Z[1/n_s]$.
This also shows that the isomorphism is a $\mathbb Z[x,x^{-1}]$-module isomorphism. 
\end{proof}

\begin{theorem}\label{mani543}
Let $E$ and $F$ be polycephaly graphs. Then $\LL(E)\cong_{\gr} \LL(F)$ if and only if there is an order preserving $\mathbb Z[x,x^{-1}]$-module isomorphism
\[\big (K_0^{\gr}(\LL(E)),[\LL(E)]\big ) \cong \big (K_0^{\gr}(\LL(F)),[\LL(F)]\big ).\] \end{theorem}

\begin{proof}
One direction is straightforward.  For the other direction, suppose there is an
order-preserving $\mathbb Z[x,x^{-1}]$-module isomorphism
$\phi: K_0^{\gr}(\LL(E)) \rightarrow K_0^{\gr}(\LL(F))$ such that
$\phi([\LL(E)])=[\LL(F)]$. 

By Theorem~\ref{polyheadb}, $\LL(E)$ is a direct sum of matrices over the field $K$, the Laurent polynomials $K[x^l,x^{-l}]$, $l \in \mathbb N$,  and Leavitt algebras, corresponding to acyclic, comets and rose graphs in $E$, respectively (see Remark~\ref{kjsdyb}). 
Since $K_0^{\gr}$ is an exact functor, by Theorems~\ref{f4j5h6h8},~\ref{cothemp} and ~\ref{re282}, there is an order-preserving $\mathbb Z[x,x^{-1}]$-isomorphism
\[ K_0^{\gr}(\LL(E))\cong \Big ( \bigoplus_k \mathbb Z[x,x^{-1}] \Big ) \Big ( \bigoplus_h \bigoplus_{l_h} \mathbb Z  \Big ) 
\Big (  \bigoplus_t \mathbb Z [1/n_t] \Big ).  \]  Similarly for the graph $F$ we have 
\[ K_0^{\gr}(\LL(F))\cong \Big ( \bigoplus_{k'} \mathbb Z[x,x^{-1}] \Big ) \Big ( \bigoplus_{h'} \bigoplus_{l'_{h'}} \mathbb Z  \Big ) 
\Big (  \bigoplus_{t'} \mathbb Z [1/n'_{t'}] \Big ).  \]  
Throughout the proof, we consider a cycle of length one in the graphs as a comet and thus all $n_s>1$, $1\leq s \leq t$ and $n'_s>1$, $1\leq s \leq t'$.

{\it Claim:} The isomorphism $\phi$ maps each type of groups in $K_0^{\gr}(\LL(E))$ to its corresponding type in $K_0^{\gr}(\LL(F))$ and so 
the restriction of $\phi$ to each type of groups gives an isomorphism, i.e, $\phi: \bigoplus_k \mathbb Z[x,x^{-1}] \rightarrow \bigoplus_{k'} \mathbb Z[x,x^{-1}]$, $\phi:  \bigoplus_h \bigoplus_{l_h} \mathbb Z \rightarrow  \bigoplus_{h'} \bigoplus_{l'_{h'}} \mathbb Z$ and 
$\phi: \bigoplus_t \mathbb Z [1/n_t]  \rightarrow  \bigoplus_{t'} \mathbb Z [1/n'_{t'}]$  are order-preserving $\mathbb Z[x,x^{-1}]$-isomorphisms. 

Let $a \in \mathbb Z[1/n_s]$, for some $1\leq s \leq t$ (note that $n_s>1$) and suppose \[\phi(0,\dots,0,a,0\dots,0)=(b_1,\dots,b_{k'},c_1,\dots,c_{h'},d_1,\dots d_{t'}),\] where $(b_1,\dots,b_{k'}) \in \bigoplus_{k'} \mathbb Z[x,x^{-1}]$, $(c_1,\dots,c_{h'}) \in  \bigoplus_{h'} \bigoplus_{l'_{h'}} \mathbb Z$ and 
$(d_1,\dots d_{t'}) \in \bigoplus_{t'} \mathbb Z [1/n'_{t'}]$. We show that $b_i$'s and $c_i$'s are all zero. Let $m=\LCM(l'_{s})$, $1\leq s \leq h'$. Recall that the action of $x$ on $a\in  \mathbb Z[1/n_s]$, $x a= n_s a$. Also the action of $x$ on $(a_1,\dots,a_l) \in \bigoplus_l \mathbb Z$, is $x (a_1,\dots,a_l)=(a_l,a_{1},\dots,a_{l-1})$ (see Example~\ref{upst}). Consider
\begin{multline*}
(n_s^m b_1,\dots, n_s^m b_{k'}, n_s^m c_1,\dots,n_s^m c_{h'},n_s^m d_1,\dots n_s^m d_{t'})= \\ \phi (0,\dots,0,n_s^m a,0\dots,0)=
\phi (0,\dots,0,x^m a,0\dots,0)=x^m \phi (0,\dots,0, a,0\dots,0)=\\(x^m b_1,\dots, x^m b_{k'},  c_1,\dots, c_{h'},{n'_1}^m d_1,\dots {n'_{t'}}^m d_{t'}).
\end{multline*}
This immediately implies that $b_i$'s and $c_i$'s are all zero. Thus the restriction of $\phi$ on the Leavitt type algebras, gives the isomorphism  \begin{equation}\label{lkara}
\phi: \bigoplus_t \mathbb Z [1/n_t]  \longrightarrow  \bigoplus_{t'} \mathbb Z [1/n'_{t'}]. 
\end{equation}

Now consider $a \in \mathbb \bigoplus_{l_s} \mathbb Z$, for some $1\leq s \leq h$, and  suppose \[\phi(0,\dots,0,a,0\dots,0)=(b_1,\dots,b_{k'},c_1,\dots,c_{h'},d_1,\dots d_{t'}).\] As above, let  $m=\LCM(l'_{s})$, $1\leq s \leq h'$ (note that $m$ could be 1) and consider
\begin{multline*}
(b_1,\dots, b_{k'},  c_1,\dots, c_{h'}, d_1,\dots d_{t'})= \\ \phi (0,\dots,0, a,0\dots,0)=
\phi (0,\dots,0,x^m a,0\dots,0)=x^m \phi (0,\dots,0, a,0\dots,0)=\\(x^m b_1,\dots, x^m b_{k'},  c_1,\dots, c_{h'},{n'_1}^m d_1,\dots {n'_{t'}}^m d_{t'}).
\end{multline*}
This immediately implies that $b_i$'s and $d_i$'s are all zero. Thus the restriction of $\phi$ on the comet types, gives the isomorphism 
\begin{equation}\label{lkara2}
\phi:  \bigoplus_h \bigoplus_{l_h} \mathbb Z \longrightarrow  \bigoplus_{h'} \bigoplus_{l'_{h'}} \mathbb Z.
\end{equation} 

Finally we show that $\phi$ induces  an isomorphism on acyclic parts. Let \[\phi(1,0,\dots,0)=(b_1,\dots, b_{k'},  c_1,\dots, c_{h'}, d_1,\dots d_{t'}).\] Since $(1,0,\dots,0)$ is in the positive cone of $K_0^{\gr}(\LL(E))$, 
$(b_1,\dots, b_{k'},  c_1,\dots, c_{h'}, d_1,\dots d_{t'})$ is in the positive cone of $K_0^{\gr}(\LL(F))$.
Since the restriction of $\phi$ on Leavitt and comet types are isomorphisms (see~(\ref{lkara}) and (\ref{lkara2})), there is $(0,\dots,0,c_1',\dots,c_h',d_1',\dots,d_t')$ in the positive cone of $K_0^{\gr}(\LL(E))$ such that 
$\phi(0,\dots,0,c_1',\dots,c_h',d_1',\dots,d_t')=(0,\dots, 0,  c_1,\dots, c_{h'}, d_1,\dots d_{t'})$. It follows 
\[\phi(1,0,\dots,0, -c_1',\dots,-c_h',-d_1',\dots,-d_t')=(b_1,\dots, b_{k'},0,\dots,0, 0,\dots,0).\] But $(b_1,\dots, b_{k'},0,\dots,0, 0,\dots,0)$ is in the positive cone. This forces $c_i'$'s and $d_i'$'s to be zero. This shows that $\phi(1,0,\dots,0) \in  \bigoplus_{k'} \mathbb Z[x,x^{-1}]$ and thus $\phi\big(\mathbb Z[x,x^{-1}],0,\dots,0\big)\subseteq \bigoplus_{k'} \mathbb Z[x,x^{-1}].$ A similar argument can be written for the other components of acyclic parts. Thus the restriction of $\phi$ on the acyclic types, gives the isomorphism  
\[\phi: \bigoplus_k \mathbb Z[x,x^{-1}] \rightarrow \bigoplus_{k'} \mathbb Z[x,x^{-1}].\]
This finishes the proof of the claim. Thus we can reduce the theorem to graphs of the same types, namely, $E$ and $F$ being both acyclic, multi-headed comets or multi-headed rose graphs.

We consider each case separately. 

(i) $E$ and $F$ are acyclic: In this case $\LL(E)$ and $\LL(F)$ are $K$-algebras (see~(\ref{skyone23})). The isomorphism between Leavitt path algebras now follows  from Theorem~\ref{catgrhsf}, by setting $A=K$ and $\Gamma_A=0$. 

(ii) $E$ and $F$ are multi-headed comets: By Theorem~\ref{polyheadb} and ~\ref{cothemp}, (since $K_0^{\gr}$ is an exact functor) $K_0^{\gr} (\LL(E)) \cong \bigoplus_{s=1}^h \bigoplus_{l_s} \mathbb Z$ and 
$K_0^{\gr} (\LL(F))\cong \bigoplus_{s=1}^{h'} \bigoplus_{l'_{s}} \mathbb Z$. Tensoring with $\mathbb Q$, the isomorphism $\phi$  extends to an isomorphism of $\mathbb Q$-vector spaces, which then immediately implies
$\sum_{s=1}^h l_s=\sum_{i=1}^{h'} l'_s$. We claim that $h=h'$ and (after a suitable permutation), $l_s=l'_s$ for any $1\leq s\leq h$ and the restriction of $\phi$ gives an isomorphism $\phi: \bigoplus_{l_s} \mathbb Z \rightarrow \bigoplus_{l'_s} \mathbb Z$, $1\leq s\leq h$. 

Let $r=\sum_{s=1}^h l_s=\sum_{i=1}^{h'} l'_s$. Consider the matrix $a=(a_{ij})$, $1\leq i,j \leq r$, representing the isomorphism $\phi$.
We first show that $a$ has to be a permutation matrix.  Since $\phi$ is an order preserving homomorphism, all the entires of the matrix $a$ are positive integers.  Since $\phi$ is an isomorphism, $a$ is an invertible matrix, with the inverse $b=(b_{ij})$ with $b_{ij} \geq 0$, $1\leq i,j \leq r$, as $b$ is also order preserving. Thus there is $1\leq j\leq r$ such that $a_{1j}\not = 0$.  Since $\sum_{i=1}^{r}a_{1i} b_{il}=0$, for $1<l\leq r$, $a_{1j}\not = 0$ and all the numbers involved are positive, it follows $b_{jl}=0$ for all $1<l \leq r$. That is the entires of $j$-th row of the matrix $b$ are all zero except $b_{1j}$, which has to be nonzero. If there is another $j'\not = j$ such that $a_{1j'}\not = 0$, a similar argument shows that all $b_{j'l}=0$ for all $1<l \leq r$. But this then implies that the matrix $b$ is singular which is not the case. That is there is only one non-zero entry in the first row of the matrix $a$. Furthermore, since $a_{1j}b_{j1}=1$, we get $a_{1j}=1$. Repeating the same argument for other rows of the matrix $a$, we obtain that $a$ consists of exactly one $1$  in each row and  column, i.e., it is a permutation matrix.

Let $e_s$ be the standard generators of $K_0^{\gr} (\LL(E))$. 
Since $\phi$ is a permutation, $\phi(e_1)=e_k$ for some $k$. Suppose that for a suitable permutation of $l'_s$, $1\leq s \leq h'$, we have $1\leq k\leq l'_1$, i.e., $e_k$ is a non-zero element of $\bigoplus_{l'_1} \mathbb Z$. Then \[x^{l_1} e_k=x^{l_1} \phi(e_1)= \phi(x^{l_1} e_1)= \phi(e_1) =e_k.\] 
This implies that $l_1 \geq l'_1$. On the other hand 
\[\phi(e_1)= e_k= x^{l'_1} e_k= x^{l'_1} \phi(e_1)= \phi(x^{l'_1} e_1).\] So $e_1= x^{l'_1} e_1$. This implies that $l'_1 \geq l_1$. Thus $l_1=l'_1$. Again, since $\phi$ is a $\mathbb Z[x,x^{-1}]$-module isomorphism, the generators of  $\bigoplus_{l_1} \mathbb Z$ maps to 
the generators of $\bigoplus_{l'_1} \mathbb Z$, so it follows that the restriction of $\phi$ gives an isomorphism $\phi: \bigoplus_{l_1} \mathbb Z \rightarrow \bigoplus_{l'_1} \mathbb Z$. Repeating the same argument for $2\leq s \leq h$, we obtain the claim. 

This shows that $\LL(E)$ and $\LL(F)$ are direct sum of matrices over $K[x^{l_s},x^{-l_s}]$, $1\leq s \leq h$ with the graded $K_0$-group of each summand of $\LL(E)$ is $\mathbb Z[x,x^{-1}]$-module isomorphisms with the corresponding summand of $\LL(F)$.  The isomorphism between each summand and thus the isomorphism between the Leavitt path algebras $\LL(E)$ and $\LL(F)$ now follows  from applying Theorem~\ref{catgrhsf} $h$-times (i.e., for each summand), by setting $A=K[x^{l_s},x^{-l_s}]$,  $\Gamma_A=l_s\mathbb Z$, $1\leq s \leq h$, and $\Gamma=\mathbb Z$.

(iii) $E$ and $F$ are multi-headed rose graphs: By~(\ref{dampai}) $K_0^{\gr}(\LL(E)) \cong \bigoplus_{s=1}^t \mathbb Z[1/n_s]$ and $K_0^{\gr}(\LL(F)) \cong \bigoplus_{s=1}^{t'} \mathbb Z[1/n'_s]$. Tensoring with $\mathbb Q$, the isomorphism $\phi$  extends to an isomorphism of $\mathbb Q$-vector spaces, which then immediately implies
$t=t'$. Since the group of $n$-adic fractions $\mathbb Z[1/n]$ is a flat $\mathbb Z$-module, $\phi$ can be represented by a $t\times t$ matrix with entries from $\mathbb Q$. Write the sequence $(n_1,\dots,n_t)$ in an ascending order, respectively (changing the labels if necessary) and write the vertices $\{v_1,\dots,v_t\}$ (again, changing the labels if necessary)  such that $L_{n_1},\dots , L_{n_t}$ correspond to $(n_1,\dots,n_t)$. Do this for $(n'_1,\dots,n'_t)$ and $\{v'_1,\dots,v'_{t'}\}$ as well. 

We first show that  $n_s=n_s'$, $1\leq s \leq t$ and $\phi$ is a block diagonal  matrix.  Write the ascending sequence 
\begin{equation}\label{potrw}
(n_1,\dots,n_t)=(\sf{ n_1,\dots,n_1,n_2,\dots,n_2,\dots,n_k,\dots,n_k}),
\end{equation} 
where each ${\sf n}_l$ occurring $r_l$ times, $1\leq l \leq k$.  Let $\phi$ be represented by the matrix $a=(a_{ij})_{t\times t}$. Choose the smallest $s_1$, $1\leq s_1\leq t$, such that starting from $s_1$-th row all entires below the $s_1$-th row and between the $1$st and $r_1$-th column (including the column $r_1$) are zero. Clearly $s_1\geq 1$, as there is at least one nonzero entry on the $r_1$-th column, as $\phi$ is an isomorphism so $a$ is an invertible matrix.   Formally, $s_1$ is chosen such that $a_{ij}=0$ for all $s_1 < i \leq t$ and $1\leq j \leq r_1$ and there is $1< j \leq r_1$ such that $a_{s_1j}\not = 0$. We will show that $s_1=r_1$. Pick 
$a_{s_1j}\not = 0$, for $1< j \leq r_1$. Then $\phi(e_{j})=\sum_{s=1}^{t}a_{sj}e'_s$, where $e_{j}\in \bigoplus_{s=1}^t \mathbb Z[1/n_s] $ and $e'_j  \in  \bigoplus_{s=1}^{t} \mathbb Z[1/n'_s]$,  
with $1$ in $j$-th entry and zero elsewhere, respectively. This shows $a_{s_1j}e'_{s_1} \in \mathbb Z[1/n'_{s_1}]$. 
But $\phi$ is a $\mathbb Z[x,x^{-1}]$-module homomorphism. Thus
\begin{equation}\label{chrisdeb}
\sum_{s=1}^{t}n'_s a_{sj}e'_s=x \phi(e_j) =\phi(x e_j)=\phi(n_1 e_j)= \sum_{s=1}^{t}n_1 a_{sj}e'_s. 
\end{equation}
 Since $a_{s_1j}\not = 0$, we obtain $n'_{s_1}=n_1$. Next we show $a_{ij}=0$ for all 
$1 \leq i \leq s_1$ and $r_1 < j \leq t$ (see the matrix below). 
\begin{equation}\label{jh7h6g}
a=\left(\begin{array}{cccccc}
a_{11} & \cdots & a_{1r_1} & 0 & \cdots & 0\\
\vdots & \vdots  & \vdots & 0 & \cdots & 0 \\
a_{s_1,1}    &  \cdots & a_{s_1,r_1} & 0 &  \cdots& 0 \\
0    & 0 & 0 &  a_{s_1+1,r_1+1} &   \cdots & a_{s_1+1,t} \\
\vdots    & \vdots & \vdots  &  \vdots &   \ddots & \vdots \\
0    & 0 & 0 &  a_{t,r_1+1} &   \cdots & a_{t,t} \\
\end{array}\right)_{t\times t}.
\end{equation}

Suppose $a_{ij} \not = 0$ for some $1 \leq i \leq s_1$ and $r_1 < j \leq t$. Since $j>r_1$, $n_j>n_1$. Applying $\phi$ on $e_j$, a similar calculation  as~(\ref{chrisdeb}) shows that, since  $a_{ij} \not = 0$, $n'_i=n_j$. But the sequence $(n'_1,\dots,n'_t)$  is  arranged in an ascending order, so if $i\leq s_1$, then $n_j=n'_i \leq n'_{s_1}=n_1$, which is a contradiction. Therefore the matrix $a$, consists of two blocks as in~(\ref{jh7h6g}). Next we show that for $1\leq i\leq s_1$ $n'_i=n'_{s_1}=n_1$. If $1\leq i\leq s_1$, there is $a_{ij}\not = 0$, for some $1\leq j\leq r_1$ (as all $a_{ij}= 0$, $j>r_1$, and $a$ is invertible). Then consider $\phi(e_j)$ and again a similar calculation as in ~(\ref{chrisdeb}) shows that $n'_i=n_1$. 

Now restricting $\phi$ to the first $r_1$ summands of $K_0^{\gr}(\LL(E))$, we have a homomorphism 
\begin{equation}\label{beau54}
\phi_{r_1}:  \bigoplus_{s=1}^{r_1} \mathbb Z[1/n_1]\longrightarrow  \bigoplus_{s=1}^{s_1} \mathbb Z[1/n_1].
\end{equation}
Being the restriction of the isomorphism $\phi$,  $\phi_{r_1}$ is also an isomorphism. Tensoring~(\ref{beau54}) with $\mathbb Q$, we deduce $s_1=r_1$ as claimed. Repeating the same argument ($k-1$ times) for the second block of the matrix $a$ in ~(\ref{jh7h6g}), we obtain that $a$ is a diagonal block matrix consisting of $k$ blocks of $r_i\times r_i$, $1\leq i \leq k$ matrices. Since $\phi$ is an order preserving homomorphism, all the entires of the matrix $a$ are positive numbers. 
Similar to the case of multi-headed comets, we show that each of $\phi_{r_i}$ is a permutation matrix.  Consider the matrix $a_{r_1}=(a_{ij})$, $1\leq i,j \leq r_1$, representing the isomorphism $\phi_{r_1}$. Since $\phi_{r_1}$ is an isomorphism, $a_{r_1}$ is an invertible matrix, with the inverse $b_{r_1}=(b_{ij})$ with $b_{ij} \geq 0$, $1\leq i,j \leq r_1$, as $b_{r_1}$ is also order preserving. Thus there is $1\leq j\leq r_1$ such that $a_{1j}\not = 0$.  Since $\sum_{i=1}^{r_1}a_{1i} b_{il}=0$, for $1<l\leq r_1$, $a_{1j}\not = 0$ and all the numbers involved are positive, it follows $b_{jl}=0$ for all $1<l \leq r_1$. That is the entires of $j$-th row of the matrix $b_{r_1}$ are all zero except $b_{1j}$, which has to be nonzero. If there is another $j'\not = j$ such that $a_{1j'}\not = 0$, a similar argument shows that all $b_{j'l}=0$ for all $1<l \leq r_1$. But this then implies that the matrix $b_{r_1}$ is singular which is not the case. That is there is only one non-zero entry in the first row of the matrix $a_{r_1}$. Furthermore, since $a_{1j}b_{j1}=1$, $a_{1j}={n_1}^{c_1}$, where $c_1 \in \mathbb Z$. Repeating the same argument for other rows of the matrix $a_{r_1}$, we obtain that $a_{r_1}$ consists of exactly one number of the form ${n_1}^c$, $c\in \mathbb Z$, in each row and  column, i.e., it is a permutation matrix with entries invertible in $\mathbb Z[1/n_1]$. As in Theorem~\ref{re282}, we can write 
\begin{equation*}
\LL(E) \cong_{\gr} \bigoplus_{s=1}^t \M_{n(v_s)} \big(\LL(1,n_s)\big)\big (|p^{v_s}_1|,\dots, |p^{v_s}_{n(v_s)}|\big).
\end{equation*}
Furthermore, one can re-group this into $k$ summands, with each summand the matrices over the Leavitt algebras $\LL(1,{\sf n}_u)$, $1\leq u \leq k$ (see~(\ref{potrw})). Do the same re-grouping for the $\LL(F)$. On the level of $K_0$, since the isomorphism $\phi$ decomposes into a diagonal block matrix,  it is enough to consider each $\phi_{r_i}$ separately. For $i=1$, 
\begin{equation}
\phi_{r_1}: \bigoplus_{s=1}^{r_1} K^{\gr}_0\Big(\M_{n(v_s)} \big(R\big)\big (|p^{v_s}_1|,\dots, |p^{v_s}_{n(v_s)}|\big)\Big)
\longrightarrow  \bigoplus_{s=1}^{r_1} K^{\gr}_0\Big(\M_{n(v'_s)} \big(R\big)\big (|p'^{v'_s}_1|,\dots, |p'^{v'_s}_{n(v'_s)}|\big)\Big)
\end{equation} with \[\phi_{r_1}\Big(\bigoplus_{s=1}^{r_1} \Big[\M_{n(v_s)} \big(R\big)\big (|p^{v_s}_1|,\dots, |p^{v_s}_{n(v_s)}|\big)\Big ]\Big )= \bigoplus_{s=1}^{r_1} \Big [\M_{n(v'_s)} \big(R\big)\big (|p'^{v'_s}_1|,\dots, |p'^{v'_s}_{n(v'_s)}|\big)\Big],\] where $R=\LL(1,{\sf n}_1)$, ${\sf n}_1=n_1$. Using the graded Morita theory, (as in Theorem~\ref{re282}), we  can write 
\[\phi_{r_1}\Big( \bigoplus_{s=1}^{r_1} \big [R(-|p^{v_s}_1|)\oplus\dots\oplus R(-|p^{v_s}_{n(v_s)}|)\big ]\Big)=
 \bigoplus_{s=1}^{r_1} \big[R(-|p'^{v'_s}_1|)\oplus\dots\oplus R(-|p'^{v'_s}_{n(v'_s)}|)\big]. 
\] But $\phi_{r_1}$ is a permutation matrix (with power of $n_1$ as entires). Thus we have (by suitable relabeling if necessary),  for any $1\leq s \leq r_1$, 
\begin{equation}\label{h32319}
  n_1^{c_s}\Big([R(-|p^{v_s}_1|)\oplus\dots\oplus R(-|p^{v_s}_{n(v_s)}|)\big ]\Big)=
 \big[R(-|p'^{v'_s}_1|)\oplus\dots\oplus R(-|p'^{v'_s}_{n(v'_s)}|)\big],
\end{equation}
 where $c_s \in \mathbb Z$. Since for any $i \in \mathbb Z$, 
\[\LL(1,n_1)(i)\cong \bigoplus_{n_1}\LL(1,n_1)(i-1),\] if $c_s>0$,  we have 
\begin{equation}\label{gfqw32}
  \big[R(-|p^{v_s}_1|+c_s)\oplus\dots\oplus R(-|p^{v_s}_{n(v_s)}|+c_s)\big ]=
 \big[R(-|p'^{v'_s}_1|)\oplus\dots\oplus R(-|p'^{v'_s}_{n(v'_s)}|)\big],
\end{equation} otherwise 
\begin{equation}\label{gfqw35}
 \big [R(-|p^{v_s}_1|)\oplus\dots\oplus R(-|p^{v_s}_{n(v_s)}|)\big ]=
 \big[R(-|p'^{v'_s}_1|+c_s)\oplus\dots\oplus R(-|p'^{v'_s}_{n(v'_s)}|+c_s)\big]. 
\end{equation} 
Suppose, say,~(\ref{gfqw32}) holds. The case of ~(\ref{gfqw35}) is similar. Recall that by Dade's theorem (see \S\ref{dadestmal}), the 
functor $(-)_0:\mbox{gr-}A\rightarrow \mbox{mod-}A_0$, $M\mapsto M_0$, is an equivalence  so that it induces 
an isomorphism $K_0^{\gr}(R)  \rightarrow K_0(R_0)$. Applying this to~(\ref{gfqw32}), we obtain  
\begin{equation}\label{gfqw30021}
  \big[R_{-|p^{v_s}_1|+c_s}\oplus\dots\oplus R_{-|p^{v_s}_{n(v_s)}|+c_s}\big ]=
 \big[R_{-|p'^{v'_s}_1|}\oplus\dots\oplus R_{-|p'^{v'_s}_{n(v'_s)}|}\big]
\end{equation}
in $K_0(R_0)$. Since $R_0$ is an ultramatricial algebra, it is unit-regular and by Proposition~15.2 in~\cite{goodearlbook},
\begin{equation}\label{gfqw3002}
  R_{-|p^{v_s}_1|+c_s}\oplus\dots\oplus R_{-|p^{v_s}_{n(v_s)}|+c_s}\cong 
 R_{-|p'^{v'_s}_1|}\oplus\dots\oplus R_{-|p'^{v'_s}_{n(v'_s)}|} 
\end{equation}
as $R_0$-modules. Tensoring this with the graded ring $R$ over $R_0$, since $R$ is strongly graded (Theorem~\ref{hazst}), and so $R_{\alpha} \otimes_{R_0} R \cong_{\gr} R(\alpha)$ as graded $R$-modules (\cite[3.1.1]{grrings}), we have a (right) graded $R$-module isomorphism 
\begin{equation}
R(-|p^{v_s}_1|+c_s)\oplus\dots\oplus R(-|p^{v_s}_{n(v_s)}|+c_s)\cong_{\gr}
 R(-|p'^{v'_s}_1|)\oplus\dots\oplus R(-|p'^{v'_s}_{n(v'_s)}|).
\end{equation}
Applying $\End$-functor to this isomorphism, we get (see \S\ref{matgrhe}), 
\[\M_{n(v_s)} \big(\LL(1,{\sf n}_1)\big)\big (|p^{v_s}_1|-c_s,\dots, |p^{v_s}_{n(v_s)}|-c_s\big) \cong_{\gr}
\M_{n(v_s)} \big(\LL(1,{\sf n}_1)\big)\big (|p'^{v'_s}_1|,\dots, |p'^{v'_s}_{n(v'_s)}|\big),
\] which is a $K$-algebra isomorphism. But clearly (\S\ref{matgrhe})
\[\M_{n(v_s)} \big(\LL(1,{\sf n}_1)\big)\big (|p^{v_s}_1|-c_s,\dots, |p^{v_s}_{n(v_s)}|-c_s\big) \cong_{\gr}
\M_{n(v_s)} \big(\LL(1,{\sf n}_1)\big)\big (|p^{v_s}_1|,\dots, |p^{v_s}_{n(v_s)}|\big).
\] So 
\[
\M_{n(v_s)} \big(\LL(1,{\sf n}_1)\big)\big (|p^{v_s}_1|,\dots, |p^{v_s}_{n(v_s)}|\big) \cong_{\gr} 
\M_{n(v_s)} \big(\LL(1,{\sf n}_1)\big)\big (|p'^{v'_s}_1|,\dots, |p'^{v'_s}_{n(v'_s)}|\big).
\]
Repeating this argument for all $1 \leq s \leq r_1$ and then for all $r_i$, $1\leq i\leq k$, we get 
\[\LL(E) \cong_{\gr} \LL(F)\] as $K$-algebras.  
\end{proof}

Using Theorem~\ref{mani543} we show that Conjecture~\ref{cofian} holds for the category  polycephaly graphs. 

\begin{corollary}\label{griso7dgre}
Let $E$ and $F$ be polycephaly graphs. Then $\LL_K(E) \cong_{\gr} \LL_K(F)$ as  rings if and only if $\LL_K(E) \cong_{\gr} \LL_K(F)$ as $K$-algebras.  
\end{corollary}
\begin{proof}
Let $E$ and $F$ be two polycephaly graphs such that $\LL(E)\cong_{\gr} \LL(F)$. This gives a $\mathbb Z[x,x^{-1}]$-module isomorphism 
\[\big (K_0^{\gr}(\LL_K(E)),[\LL_K(E)]\big ) \cong \big (K_0^{\gr}(\LL_K(F)),[\LL_K(F)]\big ).\] 
As in the first part of the proof of  Theorem~\ref{mani543}, the graded $K_0$ of each type of the graph $E$ (acyclic, comets or rose graphs) is isomorphic to the corresponding type of the graph $F$.   
For the acyclic and comets part of $E$ and $F$, the combination of part (1) and (2) of Theorem~\ref{catgrhsf} shows that their corresponding Leavitt path algebras  are $K$-algebra isomorphism. (In fact if $E$ and $F$ are $C_n$-comets, then there is an $K[x^n,x^{-n}]$-graded isomorphism.)  The $K$-isomorphism of Leavitt path algebras of rose graphs part of $E$ and $F$ was proved in the last part of the proof of Theorem~\ref{mani543}. 
\end{proof}

\forget

\begin{remark}\label{ng4352}
Let $E$ and $F$ be graphs with no sinks such that the weak classification Conjecture~\ref{weakconj} holds. Then   $C^*(E)\cong C^*(F)$ implies $\LL_{\mathbb C}(E) \cong_{\gr} \LL_{\mathbb C} (F)$. Indeed, if $C^*(E)\cong C^*(F)$, then 
$C^*(E)\rtimes_{\gamma} \mathbb T \cong C^*(F)\rtimes_{\gamma} \mathbb T$. 
By \cite[Lemma~3.1]{Raeburn113}, $C^*(E\times_1 \mathbb Z)\cong C^*(E)\rtimes_{\gamma} \mathbb T$. From sequence~(\ref{lol1532}) now follows that \[K^{\gr}_0(\LL_{\mathbb C}(E)) \cong K_0(C^*(E\times_1 \mathbb Z))  \cong  K_0(C^*(F\times_1 \mathbb Z)) \cong K^{\gr}_0(\LL_{\mathbb C}(F)).\] So if  Conjecture~\ref{weakconj} holds, we then have $\LL_{\mathbb C}(E) \cong_{\gr} \LL_{\mathbb C}(F)$.  Thus by Theorem~\ref{mani543}, if $E$ and $F$ are comet or polycephaly graphs and    $C^*(E)\cong C^*(F)$ then $\LL_{\mathbb C}(E) \cong_{\gr} \LL_{\mathbb C} (F)$.
\end{remark}

\forgotten

\begin{remark}
In Theorem~\ref{mani543}, the assumption that the isomorphism between $K$-groups is $\mathbb Z[x,x^{-1}]$-module homomorphism is used in significant way. The following example shows one can not relax this assumption (see also Example~\ref{smallgraphs2}). 

Consider $R=\LL_2$ and $S=\LL_4$. By Theorem~\ref{re282}, $(K^{\gr}_0(R),[R])=(\mathbb Z[1/2],1)$ and 
$(K^{\gr}_0(S),[S])=(\mathbb Z[1/4],1)$. So the identity map, gives an isomorphism $(K^{\gr}_0(R),[R]) \cong (K^{\gr}_0(S),[S])$. But we know $\LL_2 \not \cong \LL_4$, since by Theorem~\ref{wke}, $K_0(\LL_2)=0$ whereas $K_0(\LL_4)=\mathbb Z/3\mathbb Z$. 
Theorem~\ref{mani543} can't be applied here, as this isomorphism is not $\mathbb Z[x,x^{-1}]$-module isomorphism. For, otherwise 
$2[S(-1)]=2[R(-1)]=[R]=[S]=4[S(-1)]$. This implies $[R]=2[S(-1)]=0$ which is clearly absurd.  
\end{remark}

\subsection{Matrices over Leavitt algebras}

In \cite{arock}, Abrams gives the necessary and sufficient conditions for matrices over Leavitt path algebras to be graded isomorphic. For a $\Gamma$-graded ring $A$, he considers grading induced on $\M_k(A)$ by setting  
$\M_k(A)_\gamma=\M_k(A_\gamma)$, for $\gamma \in \Gamma$, i.e., with $(0,\dots,0)$-suspension. For positive
integers $k$ and $n$, {\it factorization of $k$ along $n$}, means $k=td$ where $d\mid n^i$ for some positive integer $i$ and $\gcd(t,n)=1$. 
Then Theorem~1.3 in \cite{arock} states that for positive integers $k,n,n'$, $n \geq 2$, $M_k(\LL_n)\cong_{\gr} M_{k'}(\LL_n)$ if and only if $t=t'$, where $k=td$,  respectively $k'=t'd'$, are the factorization of $k$, respectively $k'$, along $n$. With this classification, \[M_3(\LL_2)\not \cong_{\gr} M_4(\LL_2).\] However
combining Theorem~\ref{re282} and Theorem~\ref{mani543},  we shall see that the following two graphs 
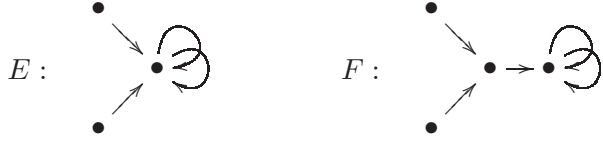
\begin{equation*}
\xymatrix@=10pt{
      &\bullet \ar[dr]   & &&&&  & \bullet \ar[dr] \\
 E: &&  \bullet \ar@(ur,rd)  \ar@(u,r) & & && F: & & \bullet  \ar[r] &  \bullet \ar@(ur,rd)  \ar@(u,r) &\\
      &\bullet \ar[ur]  & &&&&  & \bullet \ar[ur]
}
\end{equation*}
give isomorphic Leavitt path algebras, i.e.,
\begin{equation}\label{b53s}
\M_3(\LL_2)(0,1,1) \cong_{\gr} \M_4(\LL_2)(0,1,2,2).
\end{equation} This shows, by considering suspensions, we get much wider classes of graded isomorphisms of matrices over Leavitt algebras. Indeed, by Theorem~\ref{re282}, 
$(K^{\gr}_0(\LL(E),[\LL(E)])=( \mathbb Z[1/2],2)$ and $(K^{\gr}_0(\LL(F),[\LL(F)])= (\mathbb Z[1/2],2)$. Now the idenitity map gives an ordering preserving isomorphisms between these $K$-groups and so by Theorem~\ref{mani543}, we have 
$\LL(E)\cong_{\gr} \LL(F)$ which then applying ~(\ref{dampai}) we obtain the graded isomorphism~(\ref{b53s}). 

We can now give a criterion when two matrices over Leavitt algebras are graded isomorphic.  It is an easy exercise to see this covers Abrams' theorem \cite[Proposition~1.3]{arock}, when there is no suspension.

\begin{theorem}\label{cfd2497} Let $k,k',n,n'$ be positive integers, where $n$ is prime. Then 
\begin{equation}\label{tes5363}
\M_k(\LL_n)(\la_1,\dots,\la_k) \cong_{\gr} \M_{k'}(\LL_{n'})(\gamma_1,\dots,\gamma_{k'})
\end{equation}
if and only if $n=n'$ and there is  $j \in \mathbb Z$ such that $n^j\sum_{i=1}^k n^{-\la_i}=\sum_{i=1}^{k'} {n}^{-\ga_i}$.

\end{theorem}
\begin{proof}
By Theorem~\ref{re282} (and its proof), the isomorphism~(\ref{tes5363}), induces an $\mathbb Z[x,x^{-1}]$-isomorphism 
$\phi:\mathbb Z[1/n]\rightarrow \mathbb Z[1/n']$ such that $\phi(\sum_{i=1}^k n^{-\la_i})=\sum_{i=1}^{k'} {n'}^{-\ga_i}$. Now 
\[n \phi(1)=\phi(n 1) =\phi(x 1) =x\phi(1)=n'\phi(1)
\]
which implies $n=n'$. (This could have also been obtained by applying non-graded $K_0$ to isomorphism~(\ref{tes5363}) and applying Theorem~\ref{wke}.) Since $\phi$ is an isomorphism, it has to be multiplication by  $n^j$, for some $j \in \mathbb Z$.  Thus $\sum_{i=1}^{k'} {n}^{-\ga_i}=\phi(\sum_{i=1}^k n^{-\la_i})=n^j \sum_{i=1}^k n^{-\la_i}$. 

Conversely suppose there is $j \in \mathbb Z$,  such that $n^j\sum_{i=1}^k n^{-\la_i}=\sum_{i=1}^{k'} {n}^{-\ga_i}$. Set
$A=\LL_n$. Since Leavitt algebras are strongly graded (Theorem~\ref{hazst}), it follows 
$R=\M_k(A)(\la_1,\dots,\la_k)$ and $S=\M_{k'}(A)(\gamma_1,\dots,\gamma_{k'})$ are strongly graded (\cite[2.10.8]{grrings}). Again Theorem~\ref{re282} shows that the multiplication by $n^j$ gives an isomorphism $(K^{\gr}_0(R),[R]) \cong (K^{\gr}_0(S),[S])$. Therefore $n^j[R]=[S]$. Using the graded Morita, we get 
\[ n^j [A(-\la_1)\oplus \dots \oplus A(-\la_k)]= [A(-\gamma_1)\oplus\dots\oplus A(-\gamma_{k'})].
\] 
Using Dade's theorem (\S\ref{dadestmal}), passing this equality to $A_0$ , we get 
\[ n^j [A_{-\la_1}\oplus\dots \oplus A_{-\la_k}]= [A_{-\gamma_1}\oplus\dots\oplus A_{-\gamma_{k'}}].
\]
If $j\geq 0$ (the case $j<0$ is argued similarly),  since $A_0$ is an ultramatricial algebra, it is unit-regular and by Proposition~15.2 in~\cite{goodearlbook},
\[A^{n^j}_{-\la_1}\oplus\dots \oplus A_{-\la_k}^{n^j} \cong
A_{-\gamma_1}\oplus\dots\oplus A_{-\gamma_{k'}}.\]
Tensoring this with the graded ring $A$ over $A_0$, since $A=\LL(1,n)$ is strongly graded (Theorem~\ref{hazst}), and so $A_{\la} \otimes_{A_0} A \cong_{\gr} A(\la)$ as graded $A$-modules (\cite[3.1.1]{grrings}), we have a (right) graded $R$-module isomorphism
\[
A(-\la_1)^{n^j}\oplus\dots \oplus A(-\la_k)^{n^j} \cong_{\gr}
A(-\gamma_1)\oplus\dots\oplus A(-\gamma_{k'}).
\]
Since for any $i \in \mathbb Z$, $A(i)\cong \bigoplus_{n} A(i-1).$ 
we have 
\[
A(-\la_1+j)\oplus\dots \oplus A(-\la_k+j) \cong_{\gr}
A(-\gamma_1)\oplus\dots\oplus A(-\gamma_{k'}).
\]
Now applying $\End$-functor to this isomorphism, we get (see \S\ref{matgrhe}), 
\[ \M_k(A)(\la_1-j,\dots,\la_k-j) \cong_{\gr} \M_{k'}(A)(\gamma_1,\dots,\gamma_{k'}) \] which gives the isomorphism ~(\ref{tes5363}). 
\end{proof}


\forget

Applying Dade's theorem (see \S\ref{dadestmal}), we obtain the isomorphism $\phi_0$ below, 
\begin{equation}\label{gsqqwq}
\xymatrix{
\big(K^{\gr}_0(R),[R]\big) \cong \big (\mathbb Z[1/n], \sum_{i=1}^k n^{-\la_i} \big)  \ar[rr]^{\phi} && \big (\mathbb Z[1/n], \sum_{i=1}^{k'} n^{-\ga_i} \big) \cong \big(K^{\gr}_0(S),[S]\big) \ar[d]^{(-)_0} \\
\big(K_0(R_0),[R_0]\big) \ar[u]^{-\otimes_{R_0}R} \ar@{.>}[rr]^{\phi_0} && \big(K_0(S_0),[S_0]\big)
}
\end{equation}
Since $R_0$ and $S_0$
 are ultramatricial algebras, by Elliott's theorem (see~\cite[Theorem 15.26]{goodearlbook}), there exists an isomorphism 
$R_0 \cong S_0$. Tensoring this with $\LL_n$, since $\LL_n$ is strongly graded, we obtain 
\begin{multline*}
\M_k(\LL_n)(\la_1,\dots,\la_k) \cong_{\gr} \M_k(\LL_n)(\la_1,\dots,\la_k)_0 \otimes_{\mathbb Z} \LL_n
\cong_{\gr} \\  \M_{k'}(\LL_{n})(\gamma_1,\dots,\gamma_{k'})_0 \otimes_{\mathbb Z} \LL_n \cong_{\gr}  
\M_{k'}(\LL_{n})(\gamma_1,\dots,\gamma_{k'})
\end{multline*}

\forgotten


\begin{remarks}\hfill
\begin{enumerate}

\item The if part of Theorem ~\ref{cfd2497} does not need the assumption of $n$ being prime. In fact, the same proof shows 
that if $\M_k(\LL_n)(\la_1,\dots,\la_k) \cong_{\gr} \M_{k'}(\LL_{n'})(\gamma_1,\dots,\gamma_{k'})$ then $n=n'$ and there is  $j \in \mathbb Z$ such that $p^\delta \sum_{i=1}^k n^{-\la_i}=\sum_{i=1}^{k'} {n}^{-\ga_i}$, where $p | n^j$ and $\delta=\pm 1$. By setting $k=1$, $\la_1=0$ and $\ga_i=0$, $1\leq i \leq k'$, we recover Theorem~6.3 of \cite{aapardo}. 

\item The proof of Theorem ~\ref{cfd2497} shows that the isomorphism (\ref{tes5363}) is in fact an induced graded  isomorphism. Recall that for graded (right) modules $N_1$ and $N_2$ over the graded ring $A$, the graded ring isomorphism $\Phi: \End(N_1) \rightarrow \End(N_2)$ is called {\it induced}, if there is a graded $A$-module isomorphism $\phi:N_1\rightarrow N_2$ which lifts to $\Phi$ (see \cite[p.4]{arock}). 

\end{enumerate}
\end{remarks} 

Imitating a similar proof as in Theorem ~\ref{cfd2497} we can obtain the following result.

\begin{corollary} \label{k43shab}
There is a graded $\LL(1,n)$-isomorphism 
\[\LL(1,n)^k (\la_1,\dots,\la_k) \cong_{\gr} \LL(1,n)^{k'}(\ga_1,\dots,\ga_{k'})\] if and only if 
$\sum_{i=1}^k n^{\la_i}=\sum_{i=1}^{k'} {n}^{\ga_i}$.
\end{corollary}

\subsection{Graded $K_0$ of finite graphs with no sinks}

In the last two theorems we restrict ourselves to Leavitt path algebras  associated to finite graphs with no sinks. 
By Theorem~\ref{hazst}, these are precisely strongly graded unital Leavitt path algebras. One would think  the conjectures stated in Introduction ~\S\ref{introf} should be first checked for this kind of graphs. Theorem~\ref{ictp1} (and Corollary~\ref{hgfwks6}) shows that the structure of $\LL(E)$ is closely reflected  in the structure of the group $K_0^{\gr}(\LL(E))$.

\begin{theorem}
Let $E$ and $F$ be finite graphs with no sinks. Then the following
are equivalent:

\begin{enumerate}[\upshape(1)]

\item $K_0^{\gr}(\LL(E)) \cong K_0^{\gr}(\LL(F))$ as partially ordered groups.

\smallskip

\item $\LL(E)$ is  gr-Morita equivalent to $\LL(F).$

\smallskip

\item $\LL(E)_0$ is Morita equivalent to $\LL(F)_0.$

\end{enumerate}
\end{theorem}
\begin{proof}
The equivalence of the statements follow from Theorems~\ref{hazst}
and~\ref{ultrad} and  Dade's Theorem (\S\ref{dadestmal}) in combination with
Corollary~15.27 in~\cite{goodearlbook} which states that two
ultramatricial algebras are Morita equivalent if their $K_0$ are
isomorphism as partially ordered groups.
\end{proof}

Let $E$ be a finite graph with no sinks. Set $A=\LL_K(E)$. For any $u \in E^0$ and $i \in \mathbb Z$, $uA(i)$ is a right graded finitely generated projective $A$-module and any graded finitely generated projective $A$-module is generated by these modules up to isomorphism, i.e., 
$$\mathcal V^{\gr}(A)=\Big \langle \big [uA(i)\big ]  \mid u \in E^0, i \in \mathbb Z \Big \rangle, 
$$ and $K_0^{\gr}(A)$ is the group completion of $\mathcal V^{\gr}(A)$. The action of $\mathbb Z[x,x^{-1}]$ on $\mathcal V^{\gr}(A)$ and thus on $K_0^{\gr}(A)$ is defined on generators by $x^j [uA(i)]=[uA(i+j)]$, where $i,j \in \mathbb Z$. We first observe that for $i\geq 0$, 

\begin{equation}\label{hterw}
x[uA(i)]=[uA(i+1)]=\sum_{\{\alpha \in E^1 \mid s(\alpha)=u\}}[r(\alpha)A(i)].
\end{equation}

First notice that for $i\geq 0$, $A_{i+1}=\sum_{\alpha \in E^1} \alpha A_i$. It follows $uA_{i+1}=\bigoplus_{\{\alpha \in E^1 \mid s(\alpha)=u\}} \alpha A_i$ as $A_0$-modules. Using (\ref{veronaair}), and the fact that $\alpha A_i \cong r(\alpha) A_i$ as $A_0$-module,
we get $uA(i+1) \cong \bigoplus_{\{\alpha \in E^1 \mid s(\alpha)=u\}} r(\alpha) A(i)$ as graded $A$-modules. This gives~(\ref{hterw}).

A subgroup $I$ of  $K_0^{\gr}(A)$ is called a {\it graded ordered ideal} if $I$ is closed under the action of $\mathbb Z[x,x^{-1}]$, $I=I^{+}-I^{+}$, where $I^{+}=I\cap  K_0^{\gr}(A)^{+}$ and $I^{+}$ is hereditary, i.e., if $a,b \in K_0^{\gr}(A)$, $0 \leq a \leq b \in I$ then $a\in I$.

\begin{theorem}\label{ictp1}
Let $E$ be a finite graph with no sinks. Then there is  one-to-one correspondences between the subsets of hereditary and saturated subsets of $E^0$ and graded ordered ideals of $K_0^{\gr}(\LL_K(E))$.
\end{theorem}

\begin{proof}

Let $\mathcal H$ be the set of all hereditary and saturated subsets of $E^0$ and $\mathcal L(K_0^{\gr}(A))$ the set of all graded ordered ideals of $K_0^{\gr}(A)$, where $A=\LL_K(E)$. Define $\phi: \mathcal H \longrightarrow \mathcal L(K_0^{\gr}(A))$ as follows, for $H \in \mathcal H$, $\phi(H)$ is the graded ordered ideal generated by $\big \{[vA(i)]\mid v\in H, i\in \mathbb Z\big\}$. 
Define also $\psi: \mathcal L(K_0^{\gr}(A)) \longrightarrow  \mathcal H$ as follows, for a graded ordered ideal $I$ of $K_0^{\gr}(A)$, 
set $\psi(I)=\big \{ v\in E^0 \mid [vA(i)] \in I, \text{ for some } i \in \mathbb Z \big \}$.  We show that these maps are order-preserving mutually inverse maps. 

Let $I$ be an graded ordered ideal of $K_0^{\gr}(A)$. We show $H=\psi(I)$ is hereditary and saturated. Let $v \in H$ and $w\in E^0$ be an adjacent vertex to $v$, i.e., there is a $e\in E^1$ such that $s(e)=v$ and $r(e)=w$. Since $v\in H$, $[vA(i)]\in I$ for some and therefore all  $i \in \mathbb Z$. Since $I$ is closed under the action of $\mathbb Z[x,x^{-1}]$, by (\ref{hterw}), for $i\geq 0$,  
\begin{equation}\label{hgfd}
x[vA(i)]=\sum_{\{\alpha \in E^1 \mid s(\alpha)=v\}}[r(\alpha)A(i)]=[wA(i)]+\sum_{\{\alpha \in E^1\backslash \{e\} \mid s(\alpha)=v\}}[r(\alpha)A(i)]
\in I. \end{equation}
So (\ref{hgfd}) shows $[wA(i)] \leq x[vA(i)] \in I$. Since $I$ is an ordered-ideal, it follows $[wA(i)]\in I$ so $w \in H$. Now an easy induction shows that if $v\in H$ is connected to $w$ by a path (of length $\geq 1$), then $w \in H$. We next show that $H$ is saturated. Let $v\in E^0$ such that $r(e) \in H$ for every $e\in E^1$ emited from $v$. Then by (\ref{hterw}), $x[vA(i)]=\sum_{\{\alpha \in E^1 \mid s(\alpha)=v\}}[r(\alpha)A(i)]$ for any $i\in \mathbb Z^+$. But
$[r(\alpha)A(i)] \in I$ for some and therefore all $i \in \mathbb Z$. So $x[vA(i)] \in I$ and thus $[vA(i)] \in I$. Therefore $v \in H$.

Suppose $H$ is a saturated hereditary subset of $E^0$ and $I=\phi(H)$ is a graded ordered ideal generated by $\big \{[vA(i)]\mid v\in H, i\in \mathbb Z\big\}$. We show that $v\in H$ if and only if $[vA(i)]\in I$ for some $i \in \mathbb Z$. It is obvious that if $v\in H$ then $[vA(i)]\in I$, for any $i \in \mathbb Z$. For the converse, let $[vA(i)] \in I$ for some $i\in \mathbb Z$. 
Then there are $[\gamma]=\big [\sum_{t=1}^l h_tA(i_t) \big ]$, where $h_t \in H$ and $[\delta]=\big [\sum_{s=1}^{l'} u_sA(i_s)\big ]$, where $u_s\in E^0$ such that 
 \begin{equation}\label{khapp}
 [\gamma]=[v]+[\delta].
 \end{equation} There is a natural isomorphism $\mathcal V^{\gr}(A) \rightarrow K^{\gr}_0(A)^+$, $[uA(i)] \mapsto [uA(i)]$. 
 Indeed, if $[P]=[Q]$ in $K^{\gr}_0(A)^+$, then $[P_0]=[Q_0]$ in $K_0(A_0)$. But $A_0$ is an ultramatricial algebra so it is unit-regular and by Proposition~15.2 in~\cite{goodearlbook}, $P_0\cong Q_0$ as $A_0$-modules. Thus $P\cong_{\gr} P_0 \otimes_{A_0} A \cong_{\gr}  Q_0 \otimes_{A_0} A \cong_{\gr} Q$. (This argument shows $\mathcal V^{\gr}(A)$ is cancellative.) So the above map is well-defined. 
 
 Consider the epimorphisms 
\begin{align}
K^{\gr}_0(A)^+ \stackrel{\cong}{\longrightarrow}  \mathcal V^{\gr}(A) & \longrightarrow \mathcal V(A)  \stackrel{\cong}{\longrightarrow} M_E,\\
[vA(i)]  \longmapsto  [vA(i)] & \longmapsto [vA] \longmapsto [v] \notag
\end{align}

Equation~\ref{khapp} under these epimorphisms becomes $[\gamma]=[v]+[\delta]$ in $M_E$, where $[\gamma]= \sum_{t=1}^l [ h_t ]$ and 
$[\delta]= \sum_{s=1}^{l'} [ u_s ]$. The rest of the argument imitates the last part of the proof of~\cite[Proposition~5.2]{amp}. We include it here for completeness. 
 By~\cite[Lemma~4.3]{amp}, there is 
$\beta \in F(E)$ such that $\gamma \rightarrow \beta$ and $v+\delta \rightarrow \beta$. Since $H$ is hereditary and $\supp(\gamma)\subseteq H$, we get $\supp(\beta)\subseteq H$. By ~\cite[Lemma~4.2]{amp}, we have $\beta=\beta_1+\beta_2$, where $v\rightarrow \beta_1$ and $\delta\rightarrow \beta_2$. Observe that $\supp(\beta_1) \subseteq  \supp(\beta) \subseteq  H$. Using that $H$ is saturated, it is 
easy to check that, if $\alpha \rightarrow \alpha'$ and $\supp(\alpha') \subseteq H$, then $\supp(\alpha) \subseteq  H$. Using this and induction, we obtain 
that $v \in H$. 
\end{proof}


Recall the construction of the monoid $M_E$ assigned to a graph $E$ (see the proof of Theorem~\ref{wke} and~(\ref{phgqcu})). 

\begin{corollary}\label{hgfwks6}
Let $E$ be a finite graph with no sinks. Then there is a one-to-one order persevering correspondences between the following sets

\begin{enumerate}[\upshape(1)]
\item hereditary and saturated subsets of $E^0$,
\smallskip

\item graded ordered ideals of $K_0^{\gr}(\LL_K(E))$,

\smallskip

\item graded two-sided ideals of $\LL_K(E)$, 

\smallskip 
\item ordered ideals of $M_E$. 

\end{enumerate}
\end{corollary}

\begin{proof}
This follows from Theorem~\ref{ictp1} and Theorem~5.3 in~\cite{amp}. 
\end{proof} 


\begin{example}\label{smallgraphs}
In~\cite{aalp}, using a `change the graph' theorem~\cite[Theorem~2.3]{aalp}, it was shown that the following graphs give rise to isomorphic (purely infinite simple)  Leavitt path algebras. 

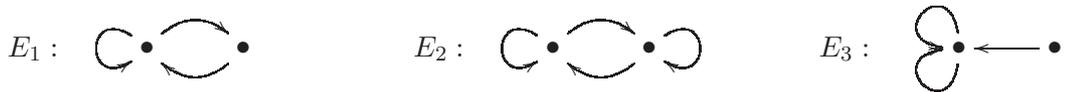
\begin{equation*}
\xymatrix{
E_1 : & \bullet  \ar@(lu,ld)\ar@/^0.9pc/[r] & \bullet \ar@/^0.9pc/[l] 
&&
E_2: & \bullet  \ar@(lu,ld)\ar@/^0.9pc/[r] & \bullet \ar@(ru,rd)\ar@/^0.9pc/[l] 
&&
E_3:& 
\bullet \ar@(u,l)\ar@(d,l) &\bullet \ar[l] 
}
\end{equation*}

\medskip 

As it is noted~\cite{aalp}, $K_0$  of Leavitt path algebras of these graphs are all zero (see Theorem~\ref{wke}). However this change the graph theorem does not respect the grading.  We calculate $K^{\gr}_0$ of these graphs and show that the Leavitt path algebras of these graphs are not graded isomorphic.  Among these graphs $E_3$ is the only one which is in the category of polycephaly graphs. 

{\bf Graph} $\mathbf {E_1}$: The adjacency matrix of this graph is $N=\left(\begin{matrix} 1 & 1\\ 1 & 0
\end{matrix}\right)$.  By Theorem~\ref{ultrad} \begin{equation*}
K_0^{\gr}(\LL(E_1)) \cong \varinjlim \mathbb Z\oplus \mathbb Z
\end{equation*}
of the inductive system $ \mathbb Z\oplus \mathbb Z
\stackrel{N^t}{\longrightarrow}  \mathbb Z\oplus \mathbb Z
\stackrel{N^t}\longrightarrow  \mathbb Z\oplus \mathbb Z
\stackrel{N^t}\longrightarrow \cdots$. Since $\det(N)\not = 0$, one can easily see that \[\varinjlim \mathbb Z\oplus \mathbb Z \cong \mathbb Z\Big [\frac{1}{\det(N)}\Big]\oplus  \mathbb Z\Big [\frac{1}{\det(N)}\Big].\] Thus \[ K_0^{\gr}(\LL(E_1)) \cong \mathbb Z\oplus \mathbb Z. \] In fact this algebra produces a so called the Fibonacci algebra (see \cite[Example IV.3.6]{davidson}).

\medskip 

{\bf Graph} $\mathbf {E_2}$:  The adjacency matrix of this graph is $N=\left(\begin{matrix} 1 & 1\\ 1 & 1
\end{matrix}\right)$. By Theorem~\ref{ultrad} \begin{equation*}
K_0^{\gr}(\LL(E_2)) \cong \varinjlim \mathbb Z\oplus \mathbb Z \cong \mathbb Z[\frac{1}{2}]
\end{equation*}

\medskip 

{\bf Graph} $\mathbf {E_3}$:  By Theorem~\ref{re282},  $K_0^{\gr}(\LL(E_3)) \cong  \mathbb Z[\frac{1}{2}]$.
This shows that $\LL(E_1)$ is not graded isomorphic to $\LL(E_2)$ and $\LL(E_3)$. Finally we show that  $\LL(E_2)\not \cong_{\gr} \LL(E_3).$ 
First note that the graph $E_2$ is the out-split of the graph 

\medskip
\[\xymatrix{
E_4:& \bullet \ar@(u,l)\ar@(d,l)}\] \medskip 

\noindent (see~\cite[Definition~2.6]{aalp}). 
Therefore by~\cite[Theorem~2.8]{aalp}  $\LL(E_2) \cong_{\gr}\LL(E_4)$. But by Theorem~\ref{re282}
\[\big ( K^{\gr}_0(\LL(E_4)),[\LL(E_4)]\big )=\big (\mathbb Z[\frac{1}{2}],1\big ),\] whereas \[\big ( K^{\gr}_0(\LL(E_3)),[\LL(E_3)]\big )=\big(\mathbb Z[\frac{1}{2}],\frac{3}{2}\big).\] Now Theorem~\ref{mani543} shows that $\LL(E_4) \not \cong_{\gr} \LL(E_3)$. Therefore  $\LL(E_2)\not \cong_{\gr} \LL(E_3).$  

\end{example}

\begin{example}\label{smallgraphs2}
In this example we look at three graphs which the graded $K_0$-group of their Leavitt path algebras and the position of their identities are the same but  the Leavitt path algebras are mutually not isomorphic.  We will then show that there is no $\mathbb Z[x,x^{-1}]$-isomorphisms between their graded $K_0$-groups. This shows that in our conjectures, not only the isomorphisms need to be order preserving, but it should also respect the  $\mathbb Z[x,x^{-1}]$-module structures.

Consider the following three graphs. 
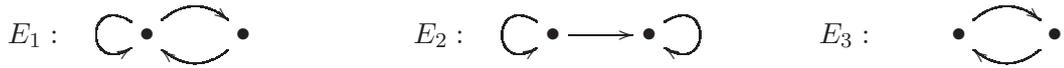
\begin{equation*}
\xymatrix{
E_1 : & \bullet  \ar@(lu,ld)\ar@/^0.9pc/[r] & \bullet \ar@/^0.9pc/[l] 
&&
E_2: & \bullet  \ar@(lu,ld)\ar[r] & \bullet \ar@(ru,rd)
&&
E_3: & \bullet  \ar@/^0.9pc/[r] & \bullet \ar@/^0.9pc/[l]  
}
\end{equation*}

The adjacency matrices of these graphs are $N_{E_1}=\left(\begin{matrix} 1 & 1\\ 1 & 0\end{matrix}\right)$,
$N_{E_2}=\left(\begin{matrix} 1 & 1\\ 0 & 1\end{matrix}\right)$ and $N_{E_3}=\left(\begin{matrix} 0 & 1\\ 1 & 0\end{matrix}\right)$.  We determine the graded $K_0$-groups  and the position of identity in each of these groups. First observe that the ring $\LL(E_i)_0$, $1\leq i \leq 3$,  is the direct limit of the following direct systems, respectively (see~(\ref{volleyb}) in the proof of Theorem~\ref{ultrad}).

\begin{align} 
\mathbb 
\LL(E_1)_0: &  \qquad K \oplus K
\stackrel{N_{E_1}^t}{\longrightarrow}  \M_2(K) \oplus K
\stackrel{N_{E_1}^t}\longrightarrow  \M_3(K)\oplus \M_2(K)
\stackrel{N_{E_1}^t}\longrightarrow \cdots \notag\\
\LL(E_2)_0:  &  \qquad K\oplus K
\stackrel{N_{E_2}^t}{\longrightarrow}  K\oplus \M_2(K)
\stackrel{N_{E_2}^t}\longrightarrow  K\oplus \M_3(K)
\stackrel{N_{E_2}^t}\longrightarrow \cdots  \label{hgdtrew} \\
\LL(E_3)_0: &  \qquad K\oplus K
\stackrel{N_{E_3}^t}{\longrightarrow}  K\oplus K
\stackrel{N_{E_3}^t}\longrightarrow  K\oplus K
\stackrel{N_{E_3}^t}\longrightarrow \cdots \notag
\end{align} 

Note that these are ultramatricial algebras with the following different Bratteli diagrams ({\it the Bratteli diagrams associated to Leavitt path algebras}). 
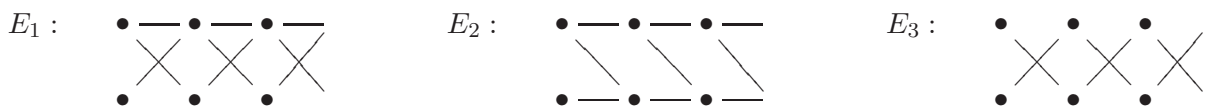
\begin{equation*}
\xymatrix@=15pt{
E_1: & \bullet \ar@{-}[r] \ar@{-}[dr] &  \bullet \ar@{-}[r] \ar@{-}[dr] & \bullet  \ar@{-}[r] \ar@{-}[dr]  &&& E_2: & \bullet  \ar@{-}[r] \ar@{-}[dr] &  \bullet \ar@{-}[r] \ar@{-}[dr] &  \bullet \ar@{-}[r] \ar@{-}[dr]  &&&  E_3: &  \bullet  \ar@{-}[dr] &  \bullet \ar@{-}[dr] & \bullet  \ar@{-}[dr]  &\\
&\bullet   \ar@{-}[ur] & \bullet \ar@{-}[ur]  & \bullet  \ar@{-}[ur]&&&  &\bullet  \ar@{-}[r] & \bullet \ar@{-}[r]  & \bullet \ar@{-}[r]&&&  & \bullet \ar@{-}[ur] & \bullet \ar@{-}[ur]  & \bullet \ar@{-}[ur]&
 }
\end{equation*}

Since the determinant of adjacency matrices are $\pm 1$, by Theorem~\ref{ultrad} (and~(\ref{veronaair})), for $1\leq i \leq 3$, 
\begin{multline}
\big(K_0^{\gr}(\LL(E_i)),[\LL(E_i)]\big)\cong \big(K_0(\LL(E_i)_0),[\LL(E_i)_0]\big)\cong \big(K_0(\varinjlim_n L^i_{0,n}),[\varinjlim_n L^i_{0,n}]\big)\cong \\ \big (\varinjlim_n K_0( L^i_{0,n}), \varinjlim_n [L^i_{0,n}]\big )\cong \big (\mathbb Z \oplus \mathbb Z, (1,1) \big ),
\end{multline}
where $L^i_{0,n}$ are as appear in the direct systems~(\ref{hgdtrew}). 

Next we show that there  is no $\mathbb Z[x,x^{-1}]$ isomorphisms between the graded $K_0$-groups. 
In  $K_0^{\gr}(\LL(E_1)) \cong \mathbb Z \oplus \mathbb Z$, one can observe that 
$[u\LL(E_1)]=\varinjlim_n [u L^1_{0,n}]= (1,0)$ and similarly $[v\LL(E_1)]=(0,1)$. Furthermore, 
by~(\ref{hterw}), the action of $\mathbb Z[x,x^{-1}]$ on these elements are, 
\begin{align}
x[u\LL(E_1)] & =[u\LL(E_1)(1)]=\sum_{\{\alpha \in E_1^1 \mid s(\alpha)=u\}}[r(\alpha)\LL(E_1)]=[u\LL(E_1)]+[v\LL(E_1)], \notag \\
x[v\LL(E_1)] &= [v\LL(E_1)(1)]=\sum_{\{\alpha \in E_1^1 \mid s(\alpha)=v\}}[r(\alpha)\LL(E_1)]=[u\LL(E_1)]. \notag
\end{align}
In particular this shows that $x(1,0)=(1,1)$ and $x(0,1)=(1,0)$. The above calculation shows that 
\begin{equation}\label{hpolor}
x (m,n)= N_{E_1}^t (m,n),
\end{equation}  where $(m,n) \in \mathbb Z\oplus \mathbb Z$. One can carry out the similar calculations observing that the action of 
$x$ on $K_0^{\gr}(\LL(E_2))$ and $K_0^{\gr}(\LL(E_3))$ are as in~(\ref{hpolor}) with $N^t_{E_2}$ and $N^t_{E_3}$ instead of $N^t_{E_1}$, respectively. 

Now if there is an order preserving isomorphism $\phi: \LL(E_i) \rightarrow \LL(E_j)$, $1\leq i,j \leq 3$, then $\phi$ has to be a permutation matrix (either $\left(\begin{matrix} 1 & 0 \\ 0 & 1\end{matrix}\right)$ or $\left(\begin{matrix} 0 & 1 \\ 1 & 0\end{matrix}\right)$). Furthermore, if $\phi$ preserves the $\mathbb Z[x,x^{-1}]$-module structure then $\phi(x a)=x\phi (a)$, for any $a \in \mathbb Z\oplus \mathbb Z$. Translating this into matrix representation, this means 
\begin{equation}\label{birdpeace}
\phi N_{E_i}^t = N_{E_j}^t \phi.
\end{equation} However, one can quickly check that, if $i\not = j$, then~(\ref{birdpeace}) can't be the case, i.e., there is no $\mathbb Z[x,x^{-1}]$-module isomorphisms between $\LL(E_i)$ and $\LL(E_j)$, $1\leq i \not = j \leq 3$.

By Corollary~\ref{hgfwks6}, $\LL_K(E_1)$ is a graded simple ring and $K_0^{\gr}(\LL(E_1))$ has no nontrivial graded ordered ideals. Notice that although $\mathbb Z\oplus 0$ and $0\oplus \mathbb Z$ are  hereditary ordered ideals of $K_0^{\gr}(\LL(E_1))$, but they are not closed under the action of 
$\mathbb Z[x,x^{-1}]$. On the other hand, one can observe that $0\oplus \mathbb Z$ is a hereditary ordered ideal of $K_0^{\gr}(\LL(E_2))$ which is also closed under the action of $\mathbb Z[x,x^{-1}]$. Indeed, one can see that the set containing the right vertex of $E_2$ is a  hereditary and saturated subset of $E_2^0$.

Using the similar arguments as above, we can show that, for two graphs $E$ and $F$, each with two vertices, such that $\det(N_E)=\det(N_F)=1$, we have $\LL_K(E) \cong_{\gr} \LL_K(F)$ if and only if $E=F$ (by re-labeling the vertices if necessary). 
\end{example}

\end{document}